\documentclass[11pt]{article}
\usepackage[utf8]{inputenc}
\usepackage[T1]{fontenc}
\usepackage{enumerate}
\usepackage{url}
\usepackage{graphicx}
\usepackage{xcolor}
\usepackage{verbatim}
\usepackage{leftidx}

\usepackage{xspace}
\usepackage{cite}

\usepackage{color}

\newcommand{\revision}[1]{{{#1}}}

\usepackage{amsmath,amsfonts,amssymb,amsthm}
\usepackage{wasysym}
\usepackage[retainorgcmds]{IEEEtrantools}

\usepackage{stmaryrd}

\usepackage{ gensymb }
\usepackage{ esint }

\usepackage{hyperref}

\newtheorem{theorem}{Theorem}[section]
\newtheorem{lemma}[theorem]{Lemma}

\newtheorem{proposition}[theorem]{Proposition}
\usepackage{notations} 

\theoremstyle{definition}
\newtheorem{definition}[theorem]{Definition}

\theoremstyle{remark}
\newtheorem{remark}[theorem]{Remark}

\DeclareMathOperator\supp{supp}

\renewcommand\subset{\subseteq}
\renewcommand\supset{\supseteq}

\newcommand\sphere{{\mathbb S}^1}

\makeatletter
\newcommand{\oset}[2]{%
  {\mathop{#2}\limits^{\vbox to -.5\ex@{\kern-\tw@\ex@
   \hbox{\scriptsize #1}\vss}}}}
\makeatother

\usepackage{algorithm,algpseudocode,algorithmicx}

\usepackage{authblk}
\author[1]{Jean-David Benamou}
\author[1]{Vincent Duval}
\affil[1]{INRIA Paris (MOKAPLAN) and CEREMADE, CNRS, Universit\'e Paris-Dauphine, PSL Research University, 75016 PARIS, FRANCE \texttt{\{jean-david.benamou,vincent.duval\}@inria.fr}}

\usepackage{diagbox} 
\usepackage{pgffor}

\usepackage{etoolbox} 

\newtoggle{arxiv}
\toggletrue{arxiv} 

\begin{document}
\title{Minimal convex extensions  and finite difference discretization of the quadratic Monge-Kantorovich problem} 
\maketitle
\date{}

\begin{abstract}
We present an adaptation of the MA-LBR scheme to the Monge-Amp\`ere equation with second boundary value condition, provided the target is a convex set. This yields a fast adaptive method to numerically solve the Optimal Transport problem between two absolutely continuous measures, the second of which has convex support. The proposed numerical method actually captures a specific Brenier solution which is minimal in some sense. We prove the convergence of the method as the grid stepsize vanishes and we show with numerical experiments that it is able to reproduce subtle properties of the Optimal Transport problem.
\end{abstract} 




\section{Introduction} 

Given two bounded open
domains $\SC$ and $\TG$ in  $\R^2$ and (strictly positive) 
probability  densities $f$ and $g$, defined respectively on $\SC$ and $\TG$,  
our goal is to numerically solve the quadratic Monge-Kantorovich problem 

\begin{equation}
\label{eq:MK}
\inf_{\{\push{T}{f} = g \}}  \int_{\SC} \| x-T(x) \|^2 \, f(x) \, dx 
\end{equation}
where $T: \SC \to \TG$ is a Borel map and $\push{T}{f} = g$ is a  mass conservation property, that is (see also \eqref{push})
\begin{equation}
\label{eq:PF}
\int_{T^{-1}(A)} f(x)\,dx  = \int_{A} g(y)\,dy , \mbox{  for all Borel subset $A$ of $Y$}. 
\end{equation}

This problem has been extensively studied, we refer the reader to 
the classical monograph of Villani \cite{MR2459454} and also the more recent book by Santambrogio \cite{santambrogio2015optimal} for a comprehensive review of its mathematical  theory and  applications.

\paragraph{Numerical approaches to optimal transport.} From the numerical point of view, the oldest approach is  the linear programming  formulation of Kantorovich  \cite{kant} which 
relaxes the problem in the product space  $\SC \times \TG$. This approach however does not  scale with the size of the  discretization.
The  so called ``Benamou-Brenier'' approach  \cite{MR1738163}  is based on a different  convex relaxation 
in a time extended space.  It is difficult to assess exactly its  efficiency but its  many numerical implementations suggest it does no better than $O(N^3)$.

Significant progress has been achieved this last decade  and  new algorithms are now available which can reach almost linear complexity. 
They can be classified in two groups~:  first, alternate projection methods (a.k.a. Split Bregman, Sinkhorn, IPFP) \cite{cuturi13, galichon2016optimal, BCCNP}  which are based on 
the entropic regularization of the Kantorovich problem. These methods are extremely flexible, they apply to a much wider range 
of Optimal Transport problems \revision{(OT)} and are easy to parallelize. The entropic regularization however  blurs the transport map. Decreasing this effect  requires 
more sophisticated tools, see  \cite{chizat2016scaling} and references therein.

The second class relies on the Monge-Amp\`ere (MA) interpretation of problem (\ref{eq:MK}). 
Because the optimal solution satisfies  $T = \nabla u$, u convex (see theorem \ref{thm:villani}) below), one seeks to solve
\begin{equation}
\label{eq:MA}
\begin{cases}
\det(\nabla^2 u) = \dfrac{f}{ g \circ \nabla u}  & \text{on } \SC, \\
  \revision{\nabla u(\SC) \subseteq  \overline{\TG}}, &\\
u  \text{ convex.}
\end{cases}\tag{MA-BV2}
\end{equation} 
Classical solutions of this problem have been studied in \cite{urbas}.  They are unique up to a constant 
which can be fixed, for example by adding $ \int_X u\, dx$ to the \revision{r.h.s. of the first equation in~\eqref{eq:MA}}, this additional term must then vanish because of the  densities  balance.
Weak solutions can also be considered either in the Aleksandrov sense or in the 
regularity framework developed after Cafarelli in the 90s using  Brenier weak solutions, that is solutions of~\eqref{eq:MK}. 
{\em Section~\ref{sec:ma}}   briefly reviews these notions as we will work with $\Cder{}^1$ solutions.

\paragraph{Numerical resolution of the Monge-Amp\`ere equation.}
Numerical methods based on Monge-Amp\`ere  subdivide again in two branches, the semi-discrete approach 
where $g$ is an empirical measure with a finite number $N$ of Dirac masses \cite{oliker, cullen, merigot , levy2015numerical}, and finite-difference methods (FD) where $f$ is discretized on a grid of size $N$ \cite{loeper, agueh, bfo}. Efficient semi-discrete algorithms rely on the  fast  computation  of  the measure of 
Laguerre cells which correspond to the  subgradient of the  dual  OT map  at the Dirac locations. 
In this paper we focus on the second approach, but we will also use that finite differences solutions yield an approximation of this subgradient at grid points.

The second boundary value condition  (BV2),
\begin{equation}
\label{eq:BV2}
\revision{\nabla u(\SC) \subseteq  \overline{\TG},}
\end{equation} 
is a non local condition and a difficulty for the Monge-Amp\`ere finite differences  approach.  Under a convexity hypothesis for $\TG$, an 
equivalent local  non-linear boundary condition is  given in \cite{bfo} which preserves the  monotonicity (aka ``degenerate ellipticity'' 
after  Oberman \cite{oberman2006convergent}) of the scheme. In particular a Newton method can be applied  for the solution of the discretized system \revision{for 
periodic \cite{loeper, agueh} or  Dirichlet \cite{obf, mirebeau3d} Boundary Conditions (BC)}. 
We provide in {\em Section~\ref{sec:cvxext}} a new interpretation of these BC in terms of an infinite domain  ``minimal'' convex extension. 
It can be used to build the same boundary conditions as in \cite{bfo}, it shows how to extend the FD scheme outside of possiby non convex supports of $f$ and also provides a suitable continuous interpolation tool for the convergence proof in {\em Section~\ref{sec:convergence}}.

\revision{For the sake of completeness, we mention that Optimal transport problems can be attacked trough  non-specific methods based on the Kantorovich linear programming relaxation and 
its Entropic regularization \cite{BCCNP, oberman_efficient_2015, schmitzer_sparse_2016} amongst many ...}

\paragraph{Convergence of the discretizations.} Existing convergence proofs for FD methods in \cite{obf,mirebeau3d} rely on the viscosity solutions  of Crandall and Lions  \cite{ug} and 
the abstract convergence theory  of  Barles and Souganidis \cite{BS}.  This is a powerful framework which can be applied, in particular, to general 
degenerate elliptic non-linear 
second order   equations. Our problem however has two specificities. First, the Monge-Amp\`ere operator is degenerate elliptic 
only on the cone of convex function and this constraint must somehow be satisfied by the discretization and preserved in the convergence process.  
Second, the theory requires a uniqueness principle stating that viscosity sub-solutions are below  super-solutions.   The first problem or more generally 
the approximation of convex functions on a grid, has attracted a lot of attention, see \cite{mirebeau2016adaptive} and the references therein. 
The  Lattice Basis Reduction (LBR)  technique in particular 
was  applied to the MA Dirichlet problem in \cite{malbr}. 
The second issue (uniquess principle) is more delicate and, even though the BV2 reformulation 
in \cite{bfo} clearly belongs to the family of oblique boundary conditions (see \cite{ug} for references)  there  is, to the best of our knowledge, no treatment of the specific (MA-BV2) and 
convexity constraint  in the viscosity theory literature.

We follow  a different path in this paper. We  build  a ``minimal''  convex extension interpolation of the discrete solutions and interpret it as an
Aleksandrov solution of an adapted semi-discrete problem.  We can then use classical optimal transport theory to prove convergence of our 
finite difference discretization.  
 Instead of  monotonicity and consistency of the scheme, the proof relies on three ingredients~:  specific 
 properties of the LBR discretization of the Monge-Amp\`ere operator ({\em Section~\ref{ss}}), the volume conservation enforced by the  BV2 
 boundary condition and the uniqueness and the $\Cder^{1}$ regularity of the limit problem. This last condition also requires the convexity of the target 
$\TG$. We borrow here some of the techniques used in \cite{carlier2015discretization} to prove convergence  of semi-discrete approximation of JKO gradient steps  problems. 
In the case where $g$ the target density is not constant, we show that simple centered FD is sufficient for the discrete gradient. 
In summary,  we provide a convergent finite difference method for the optimal transport problem  which applies to Lebesgue integrable source densities $f$ and 
Lipschitz target densities $g$ with convex support.

\paragraph{Newton solvers.} A common feature of the semi-discrete and FD approaches is the successful use of a damped  Newton method to solve the discrete set of equations. It  results in a numerically  observed linear complexity of these methods.  M\'erigot, Kitagawa and Thibert~\cite{kitagawa2016convergence} have proven  in the semi-discrete case that the Jacobian of the discrete non-linear system 
is strictly positive definite, they show convergence of the Newton method and  provide speed convergence. For finite difference 
similar results are available for the periodic and Dirichlet problems \cite{loeper, obf, mirebeau3d}.  The convergence of the Newton method relies on the invertibility  of the Jacobian of the non-linear scheme. It remains open in  the case of BV2 boundary conditions. We  provide a numerical study  in  {\em Section~\ref{sec:numerics}}  
indicating convergence and that the method  has  linear complexity.

\section{The Monge-Amp\`ere problem} 
\label{sec:ma}

\subsection{Weak solution theory for the Monge-Amp\`ere equation} 
\revision{
  \subsubsection{Some background on Optimal Transport}

  As the optimal transport problem is our main motivation for solving the Monge-Amp\`ere problem \eqref{eq:MA} with BV2 conditions, it also guides the notion of generalized solution we are interested in. 
  Let $\mu$, $\nu$ be Borel probability measures on $\RR^n$. The \textit{Monge problem} consists in finding a Borel map $T:\RR^n\rightarrow \RR^n$, which solves
  \begin{align}
    \label{eq:monge}
    \min_{\{\push{T}{\mu} = \nu \}}  \int_{\RR^n} \| x-T(x) \|^2 \, \, d\mu(x), 
  \end{align}
where $\push{T}{\mu} = \nu$ means that $\mu(T^{-1}(A))=\nu(A)$ for all Borel subset $A$ of $\RR^n$.

Its relaxation, the \textit{Monge-Kantorovich problem}, consists in finding a transport plan $\gamma$ solution to 
\begin{align}
  \label{eq:mongekanto}
  \min_{\gamma \in \Pi(\mu,\nu)} \int_{\RR^n\times \RR^n}|x-y|^2\d\gamma(x,y),
\end{align}
where $\Pi(\mu,\nu)$ is the set of Borel probability measures on $\RR^n\times \RR^n$ with respective marginals $\mu$ and $\nu$, \ie $\gamma(A\times \RR^n)=\mu(A)$ and $\gamma(\RR^n\times B)=\nu(B)$ for all Borel subsets $A,B\subseteq \RR^n$.

The following theorem is a simplified version of~\cite[Th.~2.12]{villani_topics_2003} which summarizes results of Knott-Smith (for the first point) and Brenier~\cite{brenierCPAM} (for the remaining points).

\begin{theorem}[\protect{\cite[Th.~2.12]{villani_topics_2003}}]\label{thm:villani}
  Let $\mu$, $\nu$ be probability measures on $\RR^n$ with compact support.

\begin{enumerate}
  \item $\gamma\in \Pi(\mu,\nu)$ is optimal in~\eqref{eq:mongekanto} if and only if there exists a convex lower semi-continuous function $u$ such that $\supp \gamma\subseteq \graph(\partial u)$.
  \item If $\mu$ is absolutely continuous with respect to the Lebesgue measure, there is a unique optimal transport plan
    $\gamma$ for~\eqref{eq:mongekanto}; it has the form $\gamma=(I\times \nabla u)_\sharp \tmu$, where $u:\RR^n\rightarrow \RR\cup\{+\infty\}$ is convex lower semi-continuous. Moreover, $T\eqdef \nabla u$ is the unique transport map solution to~\eqref{eq:monge}, and 
    \begin{align}
      \label{eq:densite}
      \supp(\nu)=\overline{\nabla u(\supp(u))}.
    \end{align}
  \item If both $\mu$ and $\nu$ are absolutely continuous w.r.t. the Lebesgue measure, then
    \begin{align*}
      \nabla u^*\circ \nabla u(x)&=x \quad \mu-\mbox{a.e.}\ x\in \RR^n,\\
      \nabla u\circ \nabla u^*(y)&=y \quad \nu-\mbox{a.e.}\ y\in \RR^n,
    \end{align*}
   and $\nabla u^*$ is the unique optimal transport map from $\nu$ to $\mu$.
\end{enumerate}
\end{theorem}

\begin{remark}\label{rem:infty}
  In Theorem~\ref{thm:villani}, second item, $T=\nabla u$ may be regarded as a map $T:\RR^n\to \RR^n$ defined almost everywhere. It is actually the uniqueness of the \textit{restriction of $T$ (or $\nabla u$) to $\supp(\mu)$} which is asserted. It is obviously possible to modify $T$ (or $u$) on $\R^n\setminus \supp(\mu)$ without changing the optimality of $T$. In Section~\ref{sec:minconvex}, we use this fact to choose a solution with a ``good''  behavior outside $\oSC=\supp(\mu)$.
\end{remark}

\begin{remark}\label{rem:subdiff}
  It is convenient to consider~\eqref{eq:densite} in terms of the subdifferential $\partial u$ of $u$. We say that $p\in \partial u(x)\subset \RR^n$ if and only if
  \begin{align}
    \label{eq:subdiff}
    \forall x'\in \RR^n, \quad u(x')\geq u(x)+\dotp{p}{x-x'}.
  \end{align}
  The subdifferential is a multivalued function which extends the gradient in the sense that $u$ is differentiable at $x$ iff $\partial u(x)$ is single-valued, in which case $\partial u(x)=\{\nabla u(x)\}$ (see~\cite{rockafellar1986convex} for more detail). We deduce from~\eqref{eq:densite} that the closure of $\partial u(\RR^n)$ contains $\supp(\nu)$. In general, $\partial u(\RR^n)$ is not closed, but one may show that $\supp(\nu)\subseteq\partial u(\RR^n)$. Indeed, given $y\in \supp(\nu)$, it suffices to consider $x_n\in \supp(\mu)$ such that $\nabla u(x_n)\to y$; by the compactness of $\supp(\mu)$, there is a subsequence of $(x_n)_{n\in\NN}$ which converges to some $x\in \supp(\mu)$, passing to the limit in inequality~\eqref{eq:subdiff} we get $y\in \partial u(x)$.
\end{remark}

\subsubsection{ Brenier Solutions} 
The connection between the Monge-Ampere equation and the optimal transport problem was pointed out formally in~\cite{brenierCPAM}, and stated more precisely in~\cite{mccann_convexity_1997}. The convex function $u$ of Theorem~\ref{thm:villani} has an (Aleksandrov) second derivative at almost every point of the interior of its domain, and if both $\mu$ and $\nu$ are absolutely continuous (with respective densities $f$ and $g$), 
\begin{equation}
\label{push} 
 g(\nabla u(x)) \abs{\det (\nabla^2 u(x))} = {f(x)}{} \mbox{ for $\mu$-a.e. $x\in \R^n$}.
\end{equation}  
Together with~\eqref{eq:densite}, that property tells that $u$ is in some generalized sense a solution to~\eqref{eq:MA}.

This motivates the following definition.
\begin{definition}
  Let $\mu$ and $\nu$ be two probability measures on $\RR^n$ with compact support, such that $\mu$ is absolutely continuous w.r.t.\ the Lebesgue measure. We say that $u:\RR^n\rightarrow \RR\cup\{+\infty\}$ is a \textit{Brenier solution} to~\eqref{eq:MA} if it is convex lower semi-continuous and $\push{(\nabla u)}{\mu}=\nu$.
\end{definition}

}



\subsubsection{Aleksandrov Solutions and Semi-Discrete OT} 
\revision{Aleksandrov solutions are useful to tackle the Semi-Discrete Optimal Transport problem, which corresponds to the situation when one of the two measures is an empirical measure, the other being absolutely continuous.}
\begin{definition}[Aleksandrov solution]\label{def:aleks}
Let $\mu$ and $\nu$ be two compactly supported probability measures on $\R^n$. 
A convex l.s.c. function $u$ is an \emph{Aleksandrov solution} of~\eqref{eq:MA} if for every measurable set $E\in \R^2 $,
\[ \nu(\partial u(E))  = \mu(E). \]
\end{definition}

When $\mu$ \revision{is absolutely continuous w.r.t.\ the Lebesgue measure and $\supp\nu$ is convex, Aleksandrov and Brenier solutions coincide (see~\cite{figaloeper}). 

If $\mu$ is absolutely continuous and the {\em target} measure $\nu$ has only atoms, for instance  $\nu= \sum_{i=1}^{N} g_i \, \delta_{y_i} $ where the $g_i$'s are  positive weights and the $y_i$'s are the locations of Dirac masses in the plane, Brenier and Aleksandrov solution also coincide: the Brenier map is still well defined and maps the source domain onto the finite set $\{y_i\}_{1\leq i\leq N}$.}

Conversely, if  the {\em source} measure is an empirical measure $\mu = \sum_{i=1}^{N} f_i \, \delta_{x_i} $ then it is not 
possible to define a Brenier map,  but the Aleksandrov solution $u$ still makes sense and satisfies 

\begin{equation}\label{eq:AL}
\int_{\partial u(E)} g(y)\, dy  =  f(E) = \sum_{x_i \in E} f_i    
\end{equation} 
or equivalently for all $i$ 
\begin{equation}\label{eq:AL1}
\int_{\partial u(x_i)} g(y)\, dy  =   f_i.
\end{equation} 
The mass concentrated at the Dirac locations is mapped to cells  corresponding to the sub-gradients. 
The Semi-Discrete  numerical approach is based on solving system (\ref{eq:AL1}) using a Newton method and fast computations of the subgradients $\partial u(x_i)$. These subgradients  are known as   Laguerre cells tesselation in computational geometry \cite{merigot , levy2015numerical}.
It should be noted that  $u^*$, its Legendre Fenchel transform is a  Brenier solution for the reverse  mapping  $\nu$ to $\mu$.
For more on the duality properties of Semi-Discrete OT see \cite{bfo}.

\subsection{Regularity} 

Based on   Brenier solutions  Caffarelli has developed a regularity theory for MA. 
In particular, when  $\TG$ is convex,  the following result holds\footnote{\revision{Here we follow the presentation of~\cite[Th. 1.2]{philippis_partial_2014}}}~:

\begin{theorem}[\protect{\cite{caffarelli1992boundary, caffarelli1996boundary}}]\label{thm:regul}   Let $\SC,\TG\subset \R^n$ be two bounded open sets, $f:\SC\rightarrow \R_+$, $g:\TG\rightarrow \R_+$ be two probability densities, bounded away from zero and infinity respectively on $\SC$ and $\TG$, \textit{i.e.}\ 
  \begin{equation}\label{eq:densbounded}
    \exists \lambda>0,\quad \frac{1}{\lambda} \leq f\leq \lambda \ \mbox{a.e. on $\SC$}, \quad \frac{1}{\lambda} \leq g\leq \lambda \ \mbox{a.e. on $\TG$}
  \end{equation}
  and let $T=\nabla u:\SC\rightarrow \TG$ be the optimal transport maps sending $f$ to $g$. If $\TG$ is convex, then
  \begin{enumerate}[(i)]
    \item $T\in C^{0,\alpha}_{loc}(\SC)$ for some $\alpha>0$.
    \item If in addition $f\in C^{k,\beta}_{loc}(\SC)$ and $g\in C^{k,\beta}_{loc}(\TG)$ for some $\beta\in (0,1)$, then $T\in C^{k+1,\beta}_{loc}(\SC)$.
    \item If $f\in C^{k,\beta}_{loc}(\overline{\SC})$, $g\in C^{k,\beta}_{loc}(\overline{\TG})$ and both $X$ and $Y$ are smooth and uniformly convex, then $T:\overline{\SC}\rightarrow \overline{\TG}$ is a global diffeomorphism of class $C^{k+1,\beta}$.
  \end{enumerate}
\end{theorem}

In the following, we shall assume that $f$ and $g$ are bounded away from zero and infinity. 
The first conclusion of the theorem thus implies that $u\in C^{1,\alpha}_{loc}(\SC)$. \\

 We finally  recall a slightly more general  regularity result due to Figalli and Loeper wich does not require 
 lower bounds on the source density $f$ and which holds in the plane ($n=2$)~:
\begin{theorem}[\protect{\cite[Theorem 2.1]{figaloeper}}]\label{thm:figallireg}
 Let $\SC,\TG\subset \R^2$ be two bounded open sets, $\TG$ convex, $f:\SC\rightarrow \R_+$, $g:\TG\rightarrow \R_+$ be two probability densities, such that there exist $\lambda >0$ with 
 $ f \le \frac{1}{\lambda}$ in $\SC$ and $\lambda  \le g$ in $\TG$. 
 Let $u:\R^2\rightarrow \R$ be a Brenier solution \revision{to~\eqref{eq:MA}
such that $\partial u(\RR^2)=\oTG$.}

Then $u\in \Cder^1(\R^2)$.
 \end{theorem}


\section{Minimal Brenier Solutions in $\RR^2$}\label{sec:minconvex}
\label{sec:cvxext}
 
 \subsection{Definition}
From now on, we consider $\SC\subset \R^2$, $\TG\subset \R^2$ two bounded open sets, and we assume that $\TG$ is convex.
We assume that the probability densities $\dsc$, $\dtg$ are bounded away from zero and infinity respectively on $\SC$ and $\TG$ (see~\eqref{eq:densbounded}). \revision{We note in the following Proposition that the property $\partial u(\RR^2)=\oTG$ assumed by Theorem~\ref{thm:figallireg} allows to single out a particular Brenier solution to the Monge-Ampère problem.
}

\begin{proposition}\label{prop:defue}
  Assume that $\TG$ is convex. Then there is a unique (up to an additive constant) convex lower semi-continuous function $\ue:\R^2\rightarrow \RR\cup\{+\infty\}$ which is a Brenier solution to~\eqref{eq:MA} and which satisfies $\partial \ue(\R^2)= \overline{\TG}$. 

  Moreover, $\ue\in \Cder^1(\RR^2)$.
\end{proposition}

\begin{remark}[Uniqueness] \label{rem:unique} 
As the Brenier map is unique, note that uniqueness of the potential $\ue$ up to a constant carries over 
to the other notions of solution recalled in Section~\ref{sec:ma}.
\end{remark}

\begin{proof}

We first prove uniqueness. Let $u_1$, $u_2$ be two such functions. Then the gradients of their Legendre-Fenchel conjugates $u_1^*$ and $u_2^*$ solve the optimal transport problem from $\nu$ to $\mu$, and by Theorem~\ref{thm:villani}, $\nabla u_1^*(y)=\nabla u_2^*(y)$ for a.e.\ $y\in \oTG$. As a result, there is some $C\in \RR$ such that those two convex functions satisfy $u_1^*(y)=u_2^*(y)+C$ for all $y\in \TG$.
Let us prove that the equality also holds on $\partial \TG$. Let $y_0\in \partial \TG$, $y_1\in\TG$. By lower semi-continuity and convexity
\begin{align*}
  u_1^*(y_0)&\leq \liminf_{\substack{y\to y_0\\y\in\TG}}u_1^*(y)\\
            &=\liminf_{\substack{y\to y_0\\y\in\TG}}\left(u_2^*(y)+C\right)\\
            &\leq \liminf_{t\to 0^+} \left((1-t)u_2^*(y_0)+tu_2^*(y_1)\right)+C =u_2^*(y_0)+C.
\end{align*}
The converse inequality is obtained by swapping the role of $u_1^*$ and $u_2^*$, and the equality is proved.

Now, since $u_1$, $u_2$ are proper convex lower semi-continuous, we know from~\cite{rockafellar1986convex} that
\revision{%
\begin{equation*}
 \forall i\in\{1,2\},\quad\qquad  \rint \left(\dom u_i^*\right)\subset  \dom \left(\partial u_i^*\right),
 \end{equation*}
where $\rint(\dom u_i^*)$ refers to the relative interior of $\dom u_i^*$, the effective domain of $u_i^*$ (see~\cite{rockafellar1986convex}).
 Since $\TG\subset \dom u_i^*$ is nonempty open, we get
 \begin{equation*}
\TG\subseteq \inte\left(\dom u_i^*\right)= \rint \left(\dom u_i^*\right) \subset  \dom \left(\partial u_i^*\right) \subseteq \oTG,
 \end{equation*}
 hence $\dom u_i^*\subseteq \oTG$.} The double conjugate reconstruction formula then yields
 \begin{equation*}
   \forall x\in \RR^2,\quad   u_1(x)=\sup_{y\in \oTG}\left(\dotp{y}{x}-u_1^*(y)\right) = \sup_{y\in \oTG}\left( \dotp{y}{x}-(u_2^*(y)+C)\right) =u_2(x)-C.
 \end{equation*}

To prove the existence of $\ue$, let $u_0$ be a Brenier solution to the Monge-Ampère problem~\eqref{eq:MA}.
We define
\begin{equation}\label{eq:defue}
\forall x\in \R^2,\quad  \ue(x)\eqdef \sup_{y\in \overline{\TG}} \left\{\dotp{x}{y}-u_0^*(y)\right\},
\end{equation}
where $u_0^*$ is the Legendre-Fenchel conjugate of $u_0$.
\begin{equation}
  \forall y\in \overline{\TG},\quad  u_0^*(y)\eqdef \sup_{x\in \R^2} \{\dotp{x}{y} - u_0(x)\}.
\end{equation}
It is immediate that $\ue$ is convex lower semi-continuous.
From Theorem~\ref{thm:regul}, we know that $u_0$ is a convex function which is $\Cder^{1,\alpha}_{loc}$ in $\SC$. Moreover, $u_0$ is continuous up to $\partial\SC$ since, for all $x\in \SC$, $\nabla u_0(x)\in \oTG$ which is a bounded set.
As the supremum of a (finite) upper semi-continuous function on the compact set $\oTG$, $\ue$ is finite on $\R^2$.

Now, we prove that $\ue(x)=u_0(x)$ for all $x\in \SC$. Since $u_0$ is proper convex l.s.c., \revision{for a.e.\ $x\in \SC$,}
\begin{equation*}
  u_0(x)=\dotp{x}{\nabla u_0(x)}-u_0^*(\nabla u_0(x)) = \sup_{y\in \RR^2}\left(\dotp{x}{y}-u_0^*(y) \right)\geq \sup_{y\in \oTG}\left(\dotp{x}{y}-u_0^*(y) \right)=\ue(x).
\end{equation*}
But since $\nabla u_0(x)\in \oTG$, $\dotp{x}{\nabla u_0(x)}-u_0^*(\nabla u_0(x))\leq \sup_{y\in \oTG}\left(\dotp{x}{y}-u_0^*(y)\right)$, and the above inequality is in fact an equality. \revision{That equality holds almost everywhere in the open set $\SC$ hence in fact everywhere.}

To prove that $\partial \ue(\RR^2)\subset \oTG$, let $x\in \RR^2$, $p\in \partial\ue(x)$.
\revision{\begin{align*}
    \mbox{Since}\quad  \forall h\in \RR^2\quad \ue(x)+\dotp{h}{p}&\leq \ue(x+h) ,\\
  \mbox{we get}\quad 
  \sup_{y'\in\oTG}\left(\dotp{x}{y'}-u_0^*(y')\right) +\dotp{h}{p} &\leq \sup_{y'\in\oTG}\left(\dotp{x+h}{y'}-u_0^*(y')\right)\\
                                                                   &\leq   \sup_{y'\in\oTG}\left(\dotp{x}{y'}-u_0^*(y')\right)+\sup_{y\in\oTG}\dotp{h}{y}.
\end{align*} }
As a result, \revision{$\dotp{h}{p}\leq \sup_{y\in\oTG}\dotp{h}{y}$} for all $h\in \RR^2$, hence $p\in \oTG$, and $\partial \ue(\RR^2)\subset \oTG$.

We conclude that $\ue$ is also a Brenier solution to~\eqref{eq:MA}, hence by Remark~\ref{rem:subdiff}, $\partial \ue(\RR^2)=\oTG$, and by Theorem~\ref{thm:figallireg}, we deduce that $\ue$ is $\Cder^1(\R^2)$.
\end{proof}

\begin{remark}
  From~\eqref{eq:defue}, one may observe that $\ue(x)\leq \uz(x)$ for all $x\in \RR^2$. It is  the  \textit{minimal convex extension} of $\uz$ outside $\SC$ in the sense that it is the smallest convex function defined on $\RR^2$ which coincides with $\uz$ on $\SC$. Additionally, it is minimal among all Brenier solutions in the sense that its subdifferential is the smallest possible. \revision{Indeed, by Remark~\ref{rem:subdiff}, any Brenier solution satisfies $\partial u(\RR^2)\supseteq \oTG$.}
\end{remark}


\subsection{The affine ray property}
The aim of this section is to give some insight on the behavior of $\ue$ outside $\SC$, which helps motivate the discrete scheme of Section~\ref{sec:discretize}. In the following, $\cob(\SC)$ denotes the closed convex hull of $\SC$.
\begin{proposition}\label{prop:affine} 
The function $\ue$ has the following properties.
\begin{enumerate}[(i)]
  \item For all $x\in \R^2\setminus\cob(\SC)$, there exists $x_0\in   \partial \cob({\SC})  $ 
  such that    $\ue$ is affine on the half-line $\enscond{x_0+t(x-x_0)}{t\geq 0}$. 
Moreover, $\nabla\ue(x)\in\argmax_{y\in\oTG}\dotp{y}{x-x_0}\subset \partial\TG$.

\item For all $x\in \cob(\SC)\setminus{\oSC}$, there exists $x_0\in \partial \SC$ such that $\ue$ is affine on the line segment $[x,x_0]$.
\end{enumerate}
\end{proposition}

\begin{remark}\label{rem:affine}
  The condition $\nabla\ue(x)\in\argmax_{y\in\oTG}\dotp{y}{x-x_0}$ actually means that $x-x_0$ is in $\Nn_{\oTG}(y)$, the normal cone to $\oTG$ at $y=\nabla \ue(x)$. In fact, all the points $x'$ such that $x'-x_0\in \Nn_{\oTG}(y)$ are mapped to the same gradient value $y=\nabla\ue(x)$, and $\ue$ is affine on that set.
\end{remark}

Figure~\ref{pacman} provides an illustration of Proposition~\ref{prop:affine} and Remark~\ref{rem:affine}. The connected blue regions in the source space represent points which are mapped to the same gradient value. Such regions are invariant by translation  by $\Nn_{\oTG}(y)$. In $\cob(\SC)\setminus{\oSC}$ (here $X$ is not convex) points are mapped into the interior of $\TG$, onto a line segment (black dashed line) which is a discontinuity set for the gradient of the inverse optimal transport map.  Gradients are also constant on some convex sets corresponding to the subgradients of the inverse optimal map on this line segment. These regions are naturally connected to $X$. The behavior of optimal transport maps in the case of non-convex supports is analysed in depth in \cite{figalli}.


\begin{figure}[htb]
\centering
{\includegraphics[width=0.65\paperwidth]{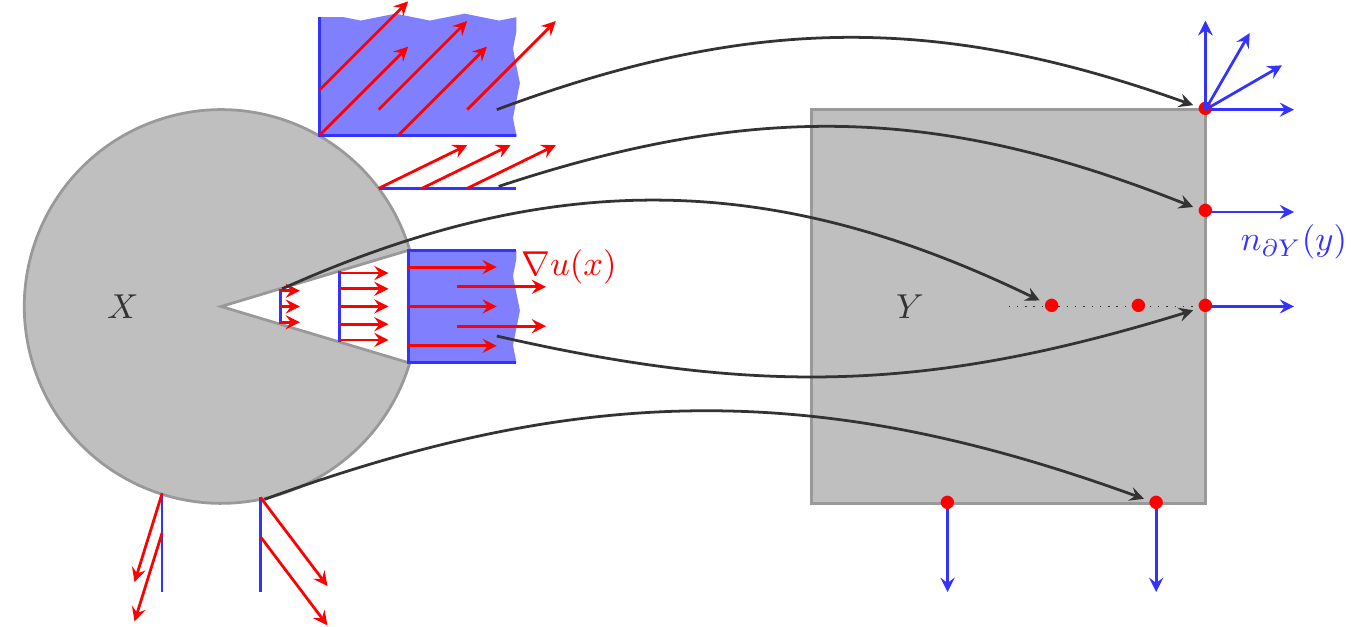}} 
\caption{Illustration of Proposition \ref{prop:affine} showing the correspondence between gradients (red)  and their support (blue) in source space and target space.
The connected blue regions in the source space represent points which are mapped to the same gradient value (in red).}
\label{pacman} 
\end{figure}

\begin{proof}
  In view of the uniqueness (up to a constant) stated in Proposition~\ref{prop:defue}, we may assume without loss of generality that the function $\uz$ used in the construction of $\ue$ (see~\eqref{eq:defue}) is convex lower semi-continuous and satisfies 
  \begin{equation}\label{eq:uzinfty}
\forall x\in \RR^2\setminus \cob{\SC},\quad  \uz(x)=+\infty,
\end{equation}
 as this does not change its being a Brenier solution.

 Let $y_0\in \oTG$ be a slope, and $x\in \RR^2$. Then $y_0$ is optimal for $x$ in~\eqref{eq:defue} iff 
 \begin{equation}\label{eq:ueoptim}
   x\in \partial\left(\uz^*+\chi_{\oTG}\right)(y_0)=\partial\uz^*(y_0)+\Nn_{\oTG}(y_0)=(\partial\uz)^{-1}(y_0)+\Nn_{\oTG}(y_0),
 \end{equation}
 as there is a point $y'\in \TG$ where $\chi_{\oTG}$ is continuous and $\uz^*$ is finite. In the above equation, $\chi_{\oTG}$ and $\Nn_{\oTG}$ respectively stand for the characteristic function and the normal cone of $\oTG$ at $y_0$,
\begin{equation}
  \chi_{\oTG}(y)=\begin{cases}
    0 &\mbox{if $y\in\oTG$,}\\
    +\infty &\mbox{otherwise,}
  \end{cases},\qquad
  \Nn_{\oTG}(y_0)=\enscond{x'\in \RR^2}{\forall y\in\oTG, \dotp{x'}{y-y_0}\leq 0}.
\end{equation}
Equation~\eqref{eq:ueoptim} is equivalent to the existence of some $x_0\in(\partial\uz)^{-1}(y_0)\subset\cob{\SC}$ such that $\dotp{x-x_0}{y-y_0}\leq 0$ for all $y\in \oTG$ (or equivalently $y_0\in\argmax_{y\in\oTG}\dotp{y}{x-x_0}$). 
Clearly, if $y_0$ is optimal for $x$, it is also optimal for $x+t(x-x_0)$ for $t>-1$.

Moreover, sets of the form~\eqref{eq:ueoptim} cover the whole space $\RR^2$, since by compactness and semi-continuity, there always exists an optimal $y_0$ for~\eqref{eq:ueoptim}. Incidentally that slope is in fact $\nabla\ue(x)$ since, provided $y_0$ is optimal for $x$,
\begin{align*}
  \forall e\in\RR^2, \forall t>0,\quad  \ue(x+te)-\ue(x)&\geq  \dotp{x+te}{y_0}-\uz^*(y_0)- \left(\dotp{x}{y_0}-\uz^*(y_0)\right),\\
  \mbox{hence}\quad  t\dotp{e}{\nabla\ue(x)}+o(t)&\geq t\dotp{e}{y_0}.
\end{align*}
Dividing by $t\to 0^+$ yields $y_0=\nabla\ue(x)$.

To summarize, we have proved $(i)$: since for all $x\in \RR^2\setminus\cob(\SC)$, $x-x_0\neq 0$, the set $\enscond{x+t(x-x_0)}{t>0}$ does indeed define a half line. As for $(ii)$, $\nabla\ue(x)\in \oTG$, hence there exists $x_0\in (\partial\uz)^{-1}(\nabla\ue(x))\cap \oSC$, so that $\nabla\ue(x_0)=\nabla\ue(x)$ (and since $\enscond{x'}{\nabla\ue(x')=\nabla\ue(x)}$ is convex, it is not restrictive to assume $x_0\in\partial\SC$).
\end{proof}

\begin{remark}
 Another point of view, using standard tools of convex analysis (see e.g.~\cite[Th. 16.4]{rockafellar1986convex}) is to interpret $\ue$ as an infimal convolution
\begin{equation}\label{eq:infconvol}
        \forall x\in \R^2,\quad   \ue(x)=\inf\enscond{\uz(x')+\sigtg(x-x')}{x'\in \co{\overline{\SC}}}. 
\end{equation}
where $\sigtg:x\mapsto \sup_{y\in\oTG}\dotp{y}{x}$ is the support function of $\oTG$.
\end{remark}

\subsection{Discussion}\label{sec:EMA}

From Proposition~\ref{prop:affine}, we see that the function $\ue$ defined in Proposition \ref{prop:defue} formally satisfies the equations
\begin{align}
  \det(D^2 \,u) \, g(\nabla  u )  = f\ \mbox{on $\SC$},\label{eq:emaMA}\\
  \det(D^2 \,u ) = 0\ \mbox{on $\R^2 \setminus \SC$},\label{eq:emaDGN}\\
  \min_{e \in \sphere } \dotp{(D^2u) e}{e}   = 0\ \mbox{on $\R^2 \setminus \SC$},\label{eq:emaOB}\\
  \sup_{e \in \sphere} \{ \dotp{\nabla u }{e} -\sigma_{\overline{\TG}}(e) \} = 0\ \mbox{on  $\R^2 \setminus \cob(\SC)$}.\label{eq:emaSUP}
\end{align}
Indeed,~\eqref{eq:emaMA} is the Monge-Amp\`ere equation, \eqref{eq:emaDGN} and \eqref{eq:emaOB} follow from the affine property in Proposition~\ref{prop:affine}.  If $\ue$ is smooth and the minimum eigenvalue of $D^2u$ is null, this also enforces convexity  (see \cite{obconvex} for the connection with the convex enveloppe problem). 
As for ~\eqref{eq:emaSUP}, since $\TG$ is convex, the inequality $\sup_{e \in \sphere} \{ \dotp{\nabla u }{e} -\sigma_{\overline{\TG}}(e) \} \leq 0$ is an equivalent formulation of the BV2 boundary condition $\nabla\ue(\RR^2)\subset \oTG$. 
The equality actually means that $\nabla\ue(x)\in \partial \TG$ for $x\in\RR^2\setminus\cob(\SC)$.

The discretization strategy is presented in Section~\ref{sec:discretize}
and then the convergence proof in Section~\ref{sec:convergence}. The convergence will hold in the Aleksandrov/Brenier setting but 
the limit solution regularity itself will depend on the the regularity of $f$ and $g$.

\begin{remark}[Uniqueness]\label{rem:unique2}   It is not difficult to show that $\ue$ is a  viscosity solutions of equations  (\ref{eq:emaMA}-\ref{eq:emaSUP}) but it is 
is much harder to prove uniqueness for this class of equations, see \cite{ug} and its references to oblique boundary conditions . However, $\ue$ coincide with the unique Brenier solution on $\SC$ and the  $\R^2$ extension  is also unique.  More precisely (see remark \ref{rem:unique}) $\ue$ is unique up to a constant. 

\end{remark} 
 
\section{Finite Difference Discretization} 
\label{sec:discretize}

This section explains how our scheme is built from  the set  of the equations of Section~\ref{sec:EMA}
and discuss the properties of the resulting discrete system.\\

We will consider a sequence of discretization steps $(\stsizen)_{n\in\NN}$, $\stsizen>0$, $\stsizen\searrow 0^+$, and we define an infinite lattice of points $\gridn\eqdef\stsizen \,\ZZ^2$. We work in a compact square domain $\domD\subset\RR^2$ (say $\domD=[-1,1]^2$) which contains $\oSC$ in its interior. We assume without loss of generality that $0\in \SC$.
A discrete solution $\udisc{}\in \RR^{\card(\gridn\cap \domD)}$  is defined on that grid :
 if $u$ is a continuous solution of our problem, its discrete interpolant on the grid is  
 $\udisc[x]  = u(x)$ for all $x \in \gridn\cap\domD$. 


We will use the following finite differences formulae in each grid direction  $e$, 
$$ \de\udisc[x] \eqdef \udisc[x+\stsizen\,e] - \udisc[x]$$   
and 
$$ \De\udisc[x] \eqdef  \de\udisc[x]  + \dme\udisc[x].$$
\revision{We say that a vector $e\in \ZZ^2$ is \textit{irreducible} if it has coprime coordinates.}

\subsection{Discretization of the target $\TG$}
As it is defined on a finite grid, our  discrete scheme is only able to estimate the directional gradient in a finite number of directions. 
Hence, the constraint we can impose in practice when discretizing (\ref{eq:emaSUP})  is that the gradient belongs to a polygonal approximation $\TGn$ to $\oTG$.
More precisely given a finite  family of irreducible vectors, $\sten\subset \ZZ^2\setminus\{0\}$, we consider the corresponding set 
\begin{align}
  \label{eq:discrTG}
  \TGn\eqdef \enscond{y\in \RR^2}{\forall e\in \sten,\ \dotp{y}{e}\leq \sigtg(e)}
\end{align}
which is nonempty, closed and convex.
We assume that 
\begin{enumerate}
  \item $\setdir$ contains $\{(0,1),(1,0),(0,-1),(-1,0)\}$  (so that $\TGn$ is compact)
  \item $\setdir\subset \Vv_{n'}$ for $\revision{n' \geq n}$ (so that $\oTG\subset \TG_{n'}\subset \TGn$),
  \item $\bigcap_{n\in\NN}\TGn = \oTG$.
  \item The following inclusion holds for $n$ large enough \begin{equation}\label{eq:vhsubsetD}
  \oSC+\stsizen\setdir\subset \domD.
\end{equation}

\end{enumerate}
The third point holds as soon as $\oTG$ is defined by a finite number of inequalities of the form~\eqref{eq:discrTG} (in which case the constraint $\nabla u(x)\in \oTG$ is exactly imposed) or provided that
 $\enscond{e/\abs{e}}{e\in \bigcup_{n\in \NN} \setdir}$ is dense in $\SS^1$.
The fourth point will be useful in~Proposition~\ref{prop:udisc} to ensure that the variations of $\udisc$ inside $\oSC$ are bounded by $\sigtgn$.

 A general strategy to ensure the above four assumptions is to choose
 \begin{align}
   \sten=\enscond{e\in\ZZ^2}{e\  \mbox{irreducible and } \normi{e}\leq \rhost}
 \end{align}
 where $\rhost\to +\infty$ and \revision{${\rhost}{\stsizen}\to 0$ as $n\to+\infty$ (for instance $\rhost=1/\sqrt{\stsizen}$)}.
The result of this approximation strategy is illustrated in Figure~\ref{eq:approx-ellipse} for three values $\rho_n\in \{1,2,3\}$. It appears that $\rho_n=3$ already yields a sharp polygonal approximation of the considered ellipse.

 In any case, the first three points above ensure that $\TGn$ converges towards $\oTG$ in the sense of the 
 Hausdorff topology~\cite[Lemma~1.8.1]{schneider1993convex}, that is
\begin{align}
  \label{eq:hausdorff}
  \lim_{n\to+\infty} \max\left(\sup_{y\in \oTG}d(y,\TGn), \sup_{y\in\TGn}d(y,\oTG) \right)=0. 
\end{align}

\begin{figure}[htb]
\centering
{\includegraphics[height=0.27\textwidth,clip, trim=0 2.4cm 0 2.4cm]{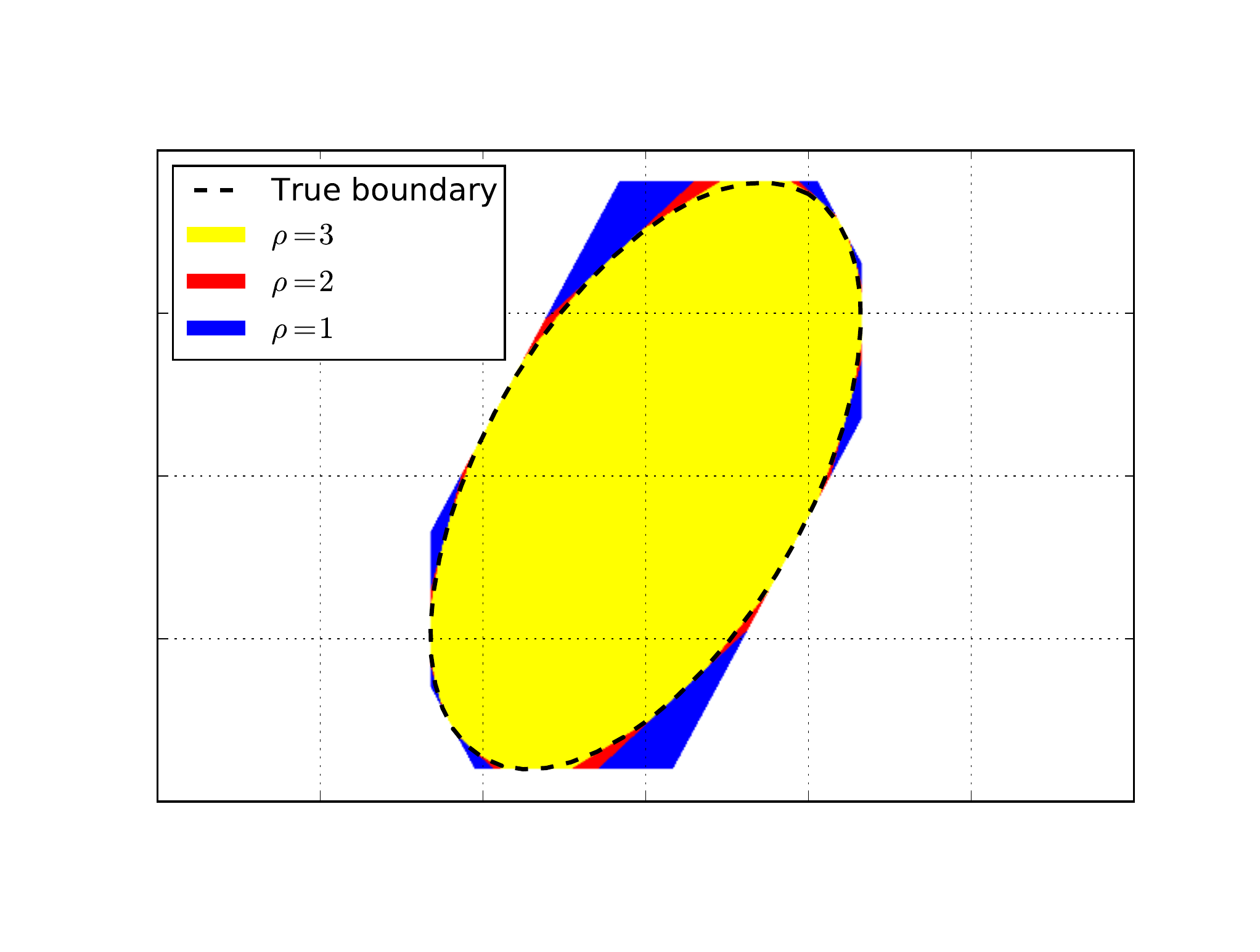}} 
{\includegraphics[height=0.27\textwidth,clip, trim=0 1.5cm 0 1.5cm]{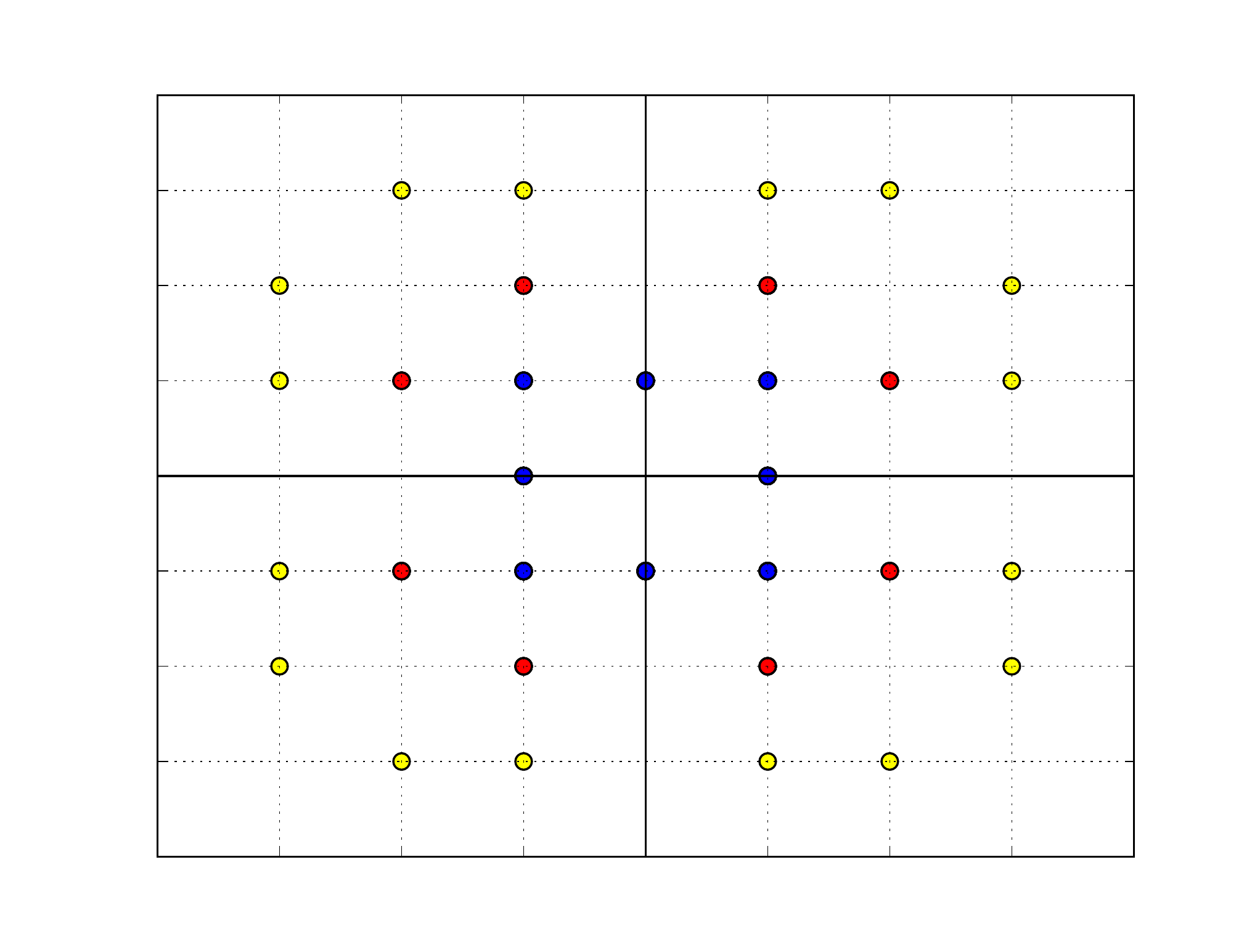}} 
\caption{Left: Approximation of an ellipse using~\eqref{eq:discrTG} for the values $\rho_n\in\{1,2,3\}$ (the true boundary is shown in dashed  red line). The ellipse is obtained by the rotation with angle $\pi/3$ of $\enscond{(y_1,y_2)\in \RR^2}{\frac{1}{a^2}y_1^2+ \frac{1}{b^2}y_2^2=1}$ with $a=2$, $b=1$. Right: the corresponding irreducible vectors used in $\sten$.}
\label{eq:approx-ellipse} 
\end{figure}

\subsection{Discretization of the Monge-Amp\`ere operator in   $\SC$}
We use the (MA-LBR)  scheme for the  Monge-Amp\`ere operator $\det(D^2 u)$  discretization~\cite{malbr}. 
It relies on the notion of superbase~: 
\begin{definition}
\label{superbase} 
A basis of $\ZZ^2$ is a pair $(e',e'')\in (\ZZ^2)^2$ such that $|\det(e',e'')| = 1$.\\
A superbase of $\ZZ^2$ is a triplet $(e,e',e'') \in (\ZZ^2)^3$ such that $e+e'+e''=0$, and $(e',e'')$ is a basis of $\ZZ^2$.
\end{definition}

The finite difference MA-LBR operator, is a consistent and monotone approximation of the Monge-Amp\`ere operator given  by  
\begin{equation}
\label{eq:SB}
\MA(\udisc) [x] \eqdef\frac{1}{\stsizen^4}  \min_{\substack{(e,e',e'') \in (\ZZ^2)^3\\ \text{superbase} } }  \hMA({\De}^+\udisc[x], {\Df}^+ \udisc[x], {\Dg}^+\udisc[x])
\end{equation}
where ${\De}^+\udisc[x]=\max({\De}\udisc[x],0)$ and for $a,b,c \in \R_+$ we define
\begin{equation}
\label{eqdef:h}
\hMA(a,b,c) := 
\begin{cases}
b c  \text{ if } a \geq b+c, \text{ and likewise permuting } a,b,c,\\
\frac 1 2 (ab+bc+ca) - \frac 1 4(a^2+b^2+c^2)  \text{ otherwise}.
\end{cases}
\end{equation}
Note that the minimum in~\eqref{eq:SB} is in fact restricted to superbases $(\ve,\vf,\vg)$ such that 
\begin{align}
  \{x,x\pm \est, x\pm \fst, x\pm \gst\}\subset \domD
\end{align}
(that is, both $x+\est$ and $x-\est$ belong to $\domD$ and similarly for $\vf$ and $\vg$). The number $N$ of such superbases is \textit{a priori} very large, but the adaptive algorithm proposed in~\cite{malbr} yields a dramatic speed-up as it is asymptotically sufficient to test $\log(N)$ superbases out of $N$.
We use this refinement in our simulations and we refer to~\cite{malbr} for the detail of the adaptive algorithm.
Also, it is shown in~\cite{malbr} that the largest width of the vectors in the optimal superbase grows with the condition number of 
the Hessian of $u$. In practice, it is therefore possible to limit the minimization to a stencil defined by its width, i.e. vectors on the grid $\ZZ \times \ZZ$ with a given maximum norm.  


The scheme consistency is remarkable.   Given a  quadratic form $u(x) = \frac{1}{2}\dotp{Mx}{x}$, where $M$ a strictly positive definite matrix with condition number $\kappa$, its grid interpolation $\udisc$ satisfies
\begin{align*}
\MA(\udisc) [x]  = \det(M)
\end{align*} 
\revision{provided that  $\{x,x\pm \est,x\pm \fst,x\pm \gst\}\subset\domD$ for all  superbases as in Definition~\ref{superbase} 
such that $\normd{e}\leq \kappa\sqrt{2}$ (see~\cite{malbr}).}

The MA-LBR operator also provides interesting ``discrete'' convexity properties (see \cite{malbr} and  proposition A.3 in \cite{mirebeau2016adaptive}). That notion is used in Appendix~\ref{sec:cvfindiff} to study the finite difference approximation of the gradient.


Finally the following property states that the MA-LBR operator overestimates the subgradient of the convex envelope of 
$\udisc$ at grid points. The proof was communicated to us by J.-M Mirebeau, and we reproduce it in Appendix~\ref{sec:apxmasub} with his kind permission. This result will be useful in the convergence proof.

\begin{lemma}\label{lem:subgradh2} 
  Let $\underbar{u}:\RR^2\rightarrow \RR\cup\{+\infty\}$ be the largest convex lower semi-continuous function which minorizes $\udisc[x]$ at all points $x\in \gridn\cap \domD$. Then, 
  \begin{align}
    \forall x\in \gridn\cap \SC,\quad   \abs{\partial\underbar{u}(x)}\leq \MA(\udisc)[x]\Vn,
    \end{align}
where $\Vn=\stsizen^2$ is the area of one cell of the grid.
\end{lemma}

\subsection{Discretization of the Monge-Amp\`ere operator and the BV2 conditions in $ \domD \setminus \SC$} 

The  MA-LBR scheme  is only suitable to discretize strictly convex functions. 
When at least one eigenvalue decreases to $0$, the stencil width  (see above) becomes infinite. 
To capture the flat behavior of solutions in $\RR^2\setminus \SC$, we need to add more and more directions to the minimization. 

Instead, we  apply  the Wide-Stencil (WS) formulation proposed by 
Oberman in \cite{obconvex} 
to discretize~\eqref{eq:emaOB}, that is, the minimum eigenvalue of the Hessian should be $0$. This simply yields the scheme
\begin{equation} 
\label{ma_dgnr} 
\MA^0(\udisc)[x] \eqdef \min_{e\in V(x)} \De\udisc[x],
\end{equation} 
where $V(x)$ denotes the set of irreducible vectors $e\in\ZZ^2\setminus\{0\}$ such that $\{x-e,x,x+e\}\subset\domD$  


Additionally, we take advantage of the points in \revision{ $\domD\setminus \SC$  } to impose the (BV2) boundary condition, by modifying the scheme (\ref{ma_dgnr}) as follows~: 
 \begin{equation} 
\label{dgnr} 
\tMAo(\udisc)[x] \eqdef\min_{e\in V(x)\cup \setdir}  \tDe \revision{\udisc[x]}
   \end{equation} 
 where for all $x\in \gridn\cap \domD$,
 \begin{align*}
  \tDe\udisc[x]&\eqdef  \tde\udisc[x]+\tdme\udisc[x],\\
   \mbox{and }  \tde\udisc[x]&\eqdef\begin{cases}
  \udisc[x+\est]-\udisc[x] & \mbox{ if $x+\stsize e \in \domD$,}\\
  \sigtg(\est) & \mbox{otherwise.}
\end{cases}
\end{align*}
The rationale of that scheme 
comes from~\eqref{eq:emaOB} and~\eqref{eq:emaSUP}: for fixed $e\in\RR^2$, imposing $\tDe\udisc[x]=0$ is consistent with
\begin{align*}
  \dotp{D^2u(x) e}{e} &=0 \ \mbox{if $x\in \inte(\domD)\setminus \SC $},\\
  \dotp{\nabla u(x)}{e}&= \sigtg(e) \mbox{ if $x\in \partial \domD$ and $e$ points outwards $\domD$.}
\end{align*}

The formal consistency with  with~\eqref{eq:emaOB} and~\eqref{eq:emaSUP} is straightforward. 
In pratice, the same stencil  can be used for $V(x)$, i.e. discretization of the degenerate Monge-Amp\`ere operator 
and $\setdir$, i.e. the discretization of the target geometry.

\subsection{Gradient Approximation}
\label{sec:gradapprox}
Except when  $g$ is the constant density on $\TG$, one needs to discretize the gradient $\nabla u$ in order to discretize the Monge-Ampère equation~\eqref{eq:emaMA}.

In Section~\ref{sec:convergence}, we prove the convergence of the scheme as $n\to +\infty$. The main assumption to obtain this convergence 
is that the discrete gradient $\gradn \udisc$ satisfies the following uniform convergence
\begin{align}
  \label{eq:cvgrad}
  \lim_{n\to+\infty}\sup_{x\in \oSC\cap\gridn}\sup_{y\in\partial\utc(x)}\abs{\gradn\udisc[x]-y}=0,
\end{align}
where $\utc$ is the continuous interpolation of $\udisc$ defined in the next Section (Eq.~\eqref{eq:defutc})  and $\udisc$ are solutions of our discrete scheme  summarized with his properties 
 in  section \ref{ss}. 
Provided a solution $\udisc$ to the scheme exists, Theorem~\ref{thm:convergence} then ensures that its interpolation $\utc$ converges towards the minimal Brenier solution $\ue$ such that $\ue(0)=0$. \\

In particular, \revision{we prove in Appendix~\ref{sec:cvfindiff} that~\eqref{eq:cvgrad} holds for}
\begin{itemize}
\item 
the centered finite difference on the cartesian grid,
\begin{equation} 
\label{d1hC}
\gradn \udisc [x]  \eqdef \frac{1}{2\stsizen} \begin{pmatrix}
\delta^{\stsizen}_{(1,0)} \udisc [x] - \delta^{\stsizen}_{(-1,0)} \udisc [x]\\
  \delta^{\stsizen}_{(0,1)} \udisc [x] - \delta^{\stsizen}_{(0,-1)} \udisc [x]
\end{pmatrix}, 
\end{equation}
\item the forward and backward finite differences on the cartesian grid,
\begin{align} 
\gradn \udisc [x]  \eqdef \frac{1}{\stsizen} \begin{pmatrix}
\delta^{\stsizen}_{(1,0)} \udisc [x]\\
  \delta^{\stsizen}_{(0,1)} \udisc [x]
\end{pmatrix},\quad 
\gradn \udisc [x]  \eqdef \frac{1}{\stsizen} \begin{pmatrix}
-\delta^{\stsizen}_{(-1,0)} \udisc [x]\\
  -\delta^{\stsizen}_{(0,-1)} \udisc [x]
\end{pmatrix}.
\label{d1hFB}
\end{align}
\end{itemize}
\revision{Please note that the proof strongly depends on the specific properties of our construction (discrete-convexity, boundedness of the gradient, convergence to a $\Cder^{1}$ function\ldots), and that~\eqref{eq:cvgrad} should not hold in general for $\udisc$ an arbitrary sequence of discrete functions.}

\revision{ Another approach in the framework of Hamilton-Jacobi equations and conservation laws \cite{RL}  the simplest and classic correction is to use a Lax-Friedrich style regularisation. 
Setting $F(x,q_1,q_2) = \frac{f(x)}{g(q_1,q_2)}$  the first order term in the Monge-Amp\`ere equation is discretized as  
\begin{equation} 
\label{d1hLF}
F(x,\gradlf \udisc [x])   \eqdef     F(x, \gradnc \udisc [x]) + \alpha_x \,  \frac{1}{2} (\delta^{\stsizen}_{(0,1)} \udisc [x] + \delta^{\stsizen}_{(0,-1)} \udisc [x] )   + \beta_x \, 
\frac{1}{2} (\delta^{\stsizen}_{(1,0)} \udisc [x] + \delta^{\stsizen}_{(-1,0)} \udisc [x] ) 
\end{equation}
where $\alpha_x  =  \max_{q_1,q_2} \| \partial_{q_1} F(x,q_1,q_2) \| $ and $\beta_x  =  \max_{q_1,q_2} \| \partial_{q_2} F(x,q_1,q_2) \| $. \\
The construction ensures that the derivatives of the scheme with respect to the $\delta^{\stsizen}_e$ remain positive and hence preserve the degenerate 
ellipticity. This formulation allows to prove the convergence of the Non-linear Newton solver  the price to pay is the 
introduction of artificial diffusion wich has no meaning in our Monge-Amp\`ere equation.   }

\subsection{Summary of the scheme and property of the discrete system}
\label{ss} 

Finally, for fixed $n\in\NN$, we plan on computing $\udisc$ solution of 
\begin{equation}\label{eq:systemMAdisc}
  \forall x\in \gridn \cap \domD, \quad  
  \begin{cases}
    \MA(\udisc)[x] - \frac{\adsc[x]}{\dtg((\gradn\udisc)[x])} = 0 & \mbox{ if $x\in \SC$}      \\
    \tMAo(\udisc)[x]=0 & \mbox{otherwise}      \\
    \udisc[0]=0 &
  \end{cases}
\end{equation}
where $\adsc[x]\eqdef \frac{1}{\stsizen^2}\int_{[-\stsizen/2,\stsizen/2]^2}\dsc(x+t)\d t$ is a local average of the density $\dsc$.
The added scalar equation   $\udisc[0]=0$ fixes the constant. 

Now, we list several properties of the scheme which will be useful for the proof of convergence.

\begin{proposition}
  \label{prop:udisc}
  If $\udisc$ is a solution to~\eqref{eq:systemMAdisc}, then
  \begin{enumerate}
    \item For all $x\in \gridn \cap \SC$, and all $e$ irreducible such that $\{x+\est, x, x-\est\}\subset \domD$, $\De\udisc[x]> 0$.
    \item For all $x\in \gridn\cap (\domD\setminus\SC)$, and all $e\in\setdir$ such that $\{x+\est, x, x-\est\}\subset \domD$, $\De\udisc[x]\geq 0$.
    \item For all $x\in \gridn \cap \domD$, $e\in\setdir$ and $(k,\ell)\in \NN^2$, such that $k\leq \ell$,
      \begin{equation}\label{eq:disclip}
      \begin{aligned}
        -\sigtg(-\est)&\leq \udisc[x+(k+1)\est]-\udisc[x+k\est] \\
        &\leq  \udisc[x+(\ell+1)\est]-\udisc[x+\ell\est]\leq \sigtg(\est) 
      \end{aligned}
    \end{equation}
whenever $x+i\est\in \domD$ for $i\in \{k,k+1,\ell,\ell+1\}$.

\item If $x\in \domD$ and $e\in\setdir$ irreducible are such that $\udisc[x+\est]-\udisc[x]=\sigtg(\est)$, then $\udisc[x+k\est]=\udisc[x]+k\sigtg(\est)$ for all integer $k\geq 0$ such that $x+k\est\in \domD$.
\item There exists $C>0$ \revision{(independent of $n$)} such that for all \revision{$x,x'\in \gridn\cap\domD$}, 
  \begin{equation}\label{eq:lipu}
    \abs{\udisc[x]-\udisc[x']}\leq C\normu{x-x'}.
  \end{equation}
  \end{enumerate}
\end{proposition}

\begin{proof}
  The first point follows from $\MA(\udisc)[x]=\frac{\adsc(x)}{\dtg((\gradn\udisc)[x])}>0$ and the inequality $\hMA(a,b,c)\leq \min\{ab, bc, ca\}$ (see~\cite{malbr}). As for the second, point it follows immediately from the definition of $\tDe\udisc$ and the scheme in $\domD\setminus\SC$.

  To prove the third point, let us first assume that $\ell\in \ZZ^2$ is such that $\{x+\ell\est,x+(\ell+1)\est\}\subset \domD$ but $x+(\ell+2)\est\notin\domD$. By~\eqref{eq:vhsubsetD} we deduce that $x+(\ell+1)\est\in D\setminus \oSC$, so that
  \begin{align*}
    0\leq \tDe\udisc[x+(\ell+1)\est]=\udisc[x+\ell\est]-\udisc[x+(\ell+1)\est]+\sigtg(\est),
  \end{align*}
  which yields $\udisc[x+(\ell+1)\est]-\udisc[x+\ell\est]\leq \sigtg(\est)$. Similarly $ -\sigtg(-\est)\leq \udisc[x+(k+1)\est]-\udisc[x+k\est]$ if $x+(k-1)\est\notin \domD$ but $\{x+k\est,x+(k+1)\est\}\subset \domD$. The other intermediate inequalities follow from $\De\udisc[x+i\est]\geq 0$.

  The fourth point is a consequence of~\eqref{eq:disclip}. 

  Now, we deal with the last point. Write $x-x'=k\stsize e_1+\ell \stsize e_2$ where $e_1\eqdef(1,0)\in\setdir$ and $e_2\eqdef(0,1)\in\setdir$. Applying~\eqref{eq:disclip}, we get
\begin{align*}
  \udisc[x]-\udisc[x']&\leq \stsize\left(k^+\sigtg(e_1)+(-k)^+\sigtg(-e_1)+\ell^+\sigtg(e_2)+(-\ell)^+\sigtg(-e_2)\right),\\
\mbox{and}\quad  \udisc[x']-\udisc[x]&\leq \stsize\left(k^+\sigtg(-e_1)+(-k)^+\sigtg(e_1)+\ell^+\sigtg(-e_2)+(-\ell)^+\sigtg(e_2)\right),
\end{align*}
where, as before, $k^+\eqdef \max(k,0)$. Hence~\eqref{eq:lipu} holds with $C=\max\{\sigtg(e_1),\sigtg(-e_1),\sigtg(e_2),\sigtg(-e_2)\}$.
\end{proof}

\paragraph{Degenerate Ellipticity. and well posedness}

In \cite{oberman2006convergent},  \revision{ Oberman develops  a framework to discretize degenerate Elliptic equations such as the Monge-Amp\`ere equation. 
The relevant finite difference scheme must satisfy a modified version of monotonicity called Degenerate Ellipticity (DE).
Consider an abstract scheme represented at each point $x\in\gridn\cap\domD$ by an equation of the form 
\begin{align}
  S(x,\udisc[x],\{\delta_{y-x}\udisc[x]\}_{y\in\gridn\cap\domD\setminus\{x\}})=0.
\end{align}
\begin{definition} 
\label{degell} 
The scheme $S$ is Degenerate Elliptic (DE) if for all $x \in \gridn\cap\domD$, $S(x,\cdot,\cdot)$ is nonincreasing in its first variable, and nondecreasing in any other variable.
\end{definition} 
This condition is in particular necessary to satisfy the general convergence result of  Barles and Souganidis \cite{BS} towards viscosity solutions as discussed in the introduction. 
To the best of our knowledge, there are only two instances of DE scheme for the Monge-Amp\`ere 
operator : the Wide-Stencil and MA-LBR schemes.  The BV2 discretization also satisfies this conditions. 

Another useful result in \cite{oberman2006convergent}  (Th. 8)  is  :  {\em A DE Lipschitz and proper scheme is well defined and has a unique solution.  }
These properties can be checked on (\ref{eq:systemMAdisc}) when $g$ is the uniform density.   Otherwise and 
even when $g$ is uniformly Lipschitz,  the inner gradient approximation destroys the  DE of the scheme. 
  A fix proposed by Froese and Oberman \cite{obf}  is to remark that  
a uniform positive lower bound on the Monge-Amp\`ere operator is sufficient to dominate the 
DE defect potentially induced by the gradient approximation. Unfortunately this bound is not 
straightforward to establish  (Froese and Oberman impose  it by truncating their scheme) and well-posedness for our scheme remains open in the general case. }

\paragraph{Newton solver.}

As in \cite{malbr, bfo, mirebeau3d, loeper, merigot, blevy} we use a damped Newton algorithm  to solve 
the system (\ref{eq:systemMAdisc}).  For the description of the algorithm and proof of its convergence we refer to 
the papers above.    The crucial ingredient is to prove global  convergence of Newton iterates 
is the invertibility of the Jacobian matrix of the system.  See for example \cite{mirebeau3d} where 
the case of Dirichlet boundary conditions is treated. \\

In our case  the invertibility of the Jacobian remains an open problem, regarding~: 
\begin{itemize} 
\item The proposed  discretization of the  BV2 condition. 
\item  The non constant $g$ density case 
\end{itemize} 

In practice and in the implementation, the system (\ref{eq:systemMAdisc}) is unchanged  but for the residual inversion 
in the Newton procedure  we use an inexact Jacobian \revision{which preserves diagonal dominance  as follows.}

For each line $x$ of the Jacobian~: 
set $dF_1[x] =   \partial_{q_1} F(x,D_C^h \udisc [x])  $  and $dF_2[x] =   \partial_{q_2} F(x,D_C^h \udisc [x])  $ 
and add  coefficients of an upwind type discretization on the classical five point stencil : 
\begin{equation} 
\label{rose}
\begin{array} {rl}
 \mbox{At column $x+(1,0)\,h$ : }  & Gw = -\max(0,dF_1[x]) \\
\mbox{At column $x+(-1,0)\,h$ : }  & Ge =  \min(0,dF_1[x]) \\
\mbox{At column $x+(0,1)\,h$ : }  & Gn = - \max(0,dF_2[x]) \\
\mbox{At column $x+(0,-1)\,h$ : }  & Gs = \min(0,dF_2[x]) \\
\mbox{On the diagonal : }  & Gn+Gs+Gw+Ge 
\end{array} 
\end{equation} 

Finally, $g$ needs to be defined globally in case the iterate generate gradients 
outside of the target $\TG$.

\section{Convergence}
\label{sec:convergence}

We now assume that for all 
discretization steps $(\stsizen)_{n\in\NN}$, $\stsizen>0$, $\stsizen\searrow 0^+$,
\revision{there exists a solution} $(\udisc[x])_{x\in \gridn\cap \domD}$, and we proceed to 
show that it converges   to the Brenier solution of the problem.

\revision{%
The proof is articulated along the following steps.
\begin{enumerate}
  \item We build functions $\utc$ which ``interpolate'' the values $(\udisc[x])_{x\in\gridn}$ and which converge (along a subsequence) towards some funtion $\vlim$ as $n\to +\infty$ (Prop.~\ref{prop:generalutc} and Lemma~\ref{lem:subgradh}).
  \item The function $\utc$ is an Aleksandrov solution for a semi-discrete OT problem between some measure discrete measure $\mun$ supported on the grid $\gridn$ and some absolutely continuous measure $\nun$ (Lemma~\ref{lem:vC1}). 

  \item The measures $\mun$ and $\nun$ respectively converge to absolutely continuous measures $\tmu$ and $\tnu$, and $\vlim$ is a \textit{minimal Brenier solution} for the transport between $\tmu$ and $\tnu$ (Lemma~\ref{lem:vC1}).

  \item As a result, $\vlim\in \Cder^{1}(\RR^2)$, and the gradient approximation $\gradn$ satisfies \eqref{eq:cvgrad} (see Section~\ref{sec:gradapprox}).
  \item This is used to show that $\tmu=\mu$, $\tnu=\nu$ and $\vlim=\ue$ (Lemma~\ref{lem:numu} and Theorem~\ref{thm:convergence}).
\end{enumerate}
}


 \subsection{Convex extension as an interpolation and its properties }
 Let us recall that we assume that $\dtg$ is continuous on $\oTG$, whereas $\dsc$ is only Lebesgue integrable.
We will assume that there exists $(\inff,\supf,\infg,\supg)\in (0,+\infty)^4$, such that
\begin{align}
  \forall x\in \oSC,\ \inff\leq f(x)\leq \supf,\label{eq:boundf}\\
  \forall y\in \oTG,\ \infg\leq g(y)\leq \supg. \label{eq:boundg}
\end{align}
In fact, possibly changing $g$ in $\RR^2\setminus \oTG$, it is not restrictive to assume that $g$ is continuous \revision{on a neighborhood of $\TG_{0}\supset \oTG$ and that~\eqref{eq:boundg} holds for all $y$ in that neighborhood.}

From the values $(\udisc[x])_{x\in \gridn\cap \domD}$ of the discrete problem (see Section~\ref{ss}), we build the following function $\utc:\RR^2\rightarrow \RR$, with
\begin{align}
\utc(x)&\eqdef\sup\enscond{L(x)}{L:\RR^2\rightarrow \RR \mbox{ is affine, $\nabla L\in\TGn$, and } \forall x'\in \gridn\cap \domD, L(x')\leq \udisc[x']}.
\end{align}
Equivalently, 
\begin{align}
  \forall x\in \RR^2, \quad   \utc(x)&\eqdef\sup_{y\in \TGn} \left(\dotp{x}{y}-\uast(y)\right)\label{eq:defutc}\\
  \mbox{where }\quad \uast(y)&\eqdef\sup_{x\in \gridn\cap \domD} \left(\dotp{y}{x}-\udisc[x] \right).
\end{align}

The following proposition gathers some properties which typically hold with such constructions (see also~\cite{carlier2015discretization}).

\begin{proposition}\label{prop:generalutc}The following properties hold.
  \begin{enumerate}
    \item   $\utc$ is convex, finite-valued, and by construction $\utc(x)\leq \udisc[x]$ for all $x\in \gridn\cap \domD$.
    \item $\partial \utc(\RR^2)=\partial \utc(\gridn\cap \domD)=\TGn$.
\item Let $x\in \gridn\cap\domD$. If $\utc(x)< \udisc[x]$, then $\abs{\partial\utc(x)}=0$.
  \end{enumerate}
\end{proposition}
\begin{proof}
  The first point is left to the reader. To prove the second one, consider any slope $y\in \TGn$. There is some $x\in \gridn\cap \domD$ which minimizes $x'\mapsto\udisc[x']-\dotp{y}{x'}$ over $\gridn\cap \domD$ (let $\alpha$ denote the corresponding minimum). Then the affine function $L:x'\mapsto \dotp{y}{x'}+\alpha$ satisfies $L(x')\leq \udisc[x']$ for all $x'\in\gridn\cap \domD$, and $L(x)=\udisc(x)$. Hence, by construction of $\utc$, $L(x')\leq \utc(x')$ for all $x'\in \RR^2$, and $L(x)=\utc(x)$, which means that $y\in \partial\utc(x)$.
  As a result $\TGn\subset \partial \utc(\gridn\cap \domD)$. The fact that $\partial\utc(\RR^2)\subset \TGn$ follows from $\partial\utc(\RR^2)\subset \dom(\utc^*)=\dom(\uast+\chi_{\TGn})$.

  Now, let us prove the third point. Assume by contradiction that $\abs{\partial\utc(x)}>0$. Since ${\partial\utc(x)}$ is convex, it must have nonempty interior, hence there is some $y_0\in \TGn$, $r>0$ such that $B(y_0,r)\subset \partial \utc(x)$. For all $x'\in \RR^2$, all $p\in B(y_0,r)$, $\utc(x')\geq \utc(x) +\dotp{p}{x'-x}$
hence 
  \begin{align*} 
    \utc(x')\geq \utc(x) +\dotp{y_0}{x'-x}+r\norm{x'-x}. 
  \end{align*}
For $x'\in \gridn\cap \domD \setminus\{x\}$, $\udisc[x']\geq \utc(x')$ and $\norm{x'-x}\geq \stsize$. Hence, for $\alpha\in (0,1]$ small enough, the affine function $L:x'\mapsto \utc(x) +\dotp{y_0}{x'-x}+\alpha r\stsize$ satisfies $L(x')\leq \udisc[x']$ for all $x'\in \gridn\cap \domD$, and $\utc(x)<L(x)$ which contradicts the definition of $\utc$.
\end{proof}

The next Lemma describes more specific properties of $\utc$ which follow from the construction of $\udisc$.

\begin{lemma}\label{lem:subgradh}
  The family $\{\utc\}_{n\in\NN}\subset \Cder{}(\RR^2)$ is relatively compact for the topology of uniform convergence on compact sets.
  Moreover, \revision{with $\Vn \eqdef \stsizen^2$,}
  \begin{align}
    \forall x\in \gridn\cap \SC,\quad   \abs{\partial \utc(x)}&\leq \MA(\udisc)[x]\revision{\Vn},\label{ineq:domD}\\
    \forall x\in \gridn\cap (\domD\setminus\SC),\quad   \abs{\partial \utc(x)}&=0.\label{eq:complement}
  \end{align}
\end{lemma}

\begin{proof}
  We observe that $(\utc)_{n\in\NN}$ is uniformly equicontinuous by Proposition~\ref{prop:generalutc} and the fact that $\TGn\subset \TG_{0}$. Moreover, from Proposition~\ref{prop:udisc} and $\udisc(0)=0$, we deduce that 
  \begin{equation*}
  \forall n\in\NN, \forall x\in \gridn\cap \domD,\quad \abs{\udisc[x]}\leq C\sup_{x'\in\domD}\normu{x'}.
  \end{equation*}
  As a result, ${\utc(0)}\in [-C\sup_{x'\in\domD}\normu{x'},0]$. We deduce the claimed compactness by applying the Ascoli-Arzel\`a theorem. 

  The inequality~\eqref{ineq:domD} follows from Lemma~\ref{lem:subgradh2}. Indeed, let $\underbar{u}:\RR^2\rightarrow \RR\cup\{+\infty\}$ be the largest convex l.s.c.\ function which minorizes $\udisc[x]$ at all points $x\in \gridn\cap \domD$. If $\utc(x)<\udisc[x]$, then $\abs{\partial \utc(x)}=0$ and there is nothing to prove. Otherwise, $\utc(x)=\underbar{u}(x)=\udisc[x]$ and $\utc\leq\underbar{u}$ imply that $\partial \utc(x)\subset \partial \underbar{u}(x)$. As a result,
  \begin{align*}
    \forall x\in \gridn\cap \SC,\quad   \abs{\partial \utc(x)}&\leq \abs{\partial\underbar{u}(x)}\leq \MA(\udisc)[x]\revision{\Vn}.
  \end{align*}

 Now, we prove~\eqref{eq:complement}. Let $x\in  \gridn\cap (\domD\setminus\SC)$. Again, if $\utc(x)< \udisc[x]$, then $\abs{\partial \utc(x)}=0$, so we assume that $\utc(x)= \udisc[x]$.
 Assuming by contradiction that $\abs{\partial \utc(x)}>0$ there must exist again some $y_0\in \TGn$, $r>0$ such that $B(y_0,r)\subset \TG$, and $\utc(x')\geq \utc(x) +\dotp{y_0}{x'-x}+r\norm{x'-x}$.

 Since $\min_{e\in\setdir}\tDelt_e \udisc[x]=0$, there exists $e\in \ZZ^2\setminus\{0\}$ such that $\tDelt_e \udisc[x]=0$.
  If $\{x-\stsizen e,x,x+\stsizen e\}\subset \domD$, then 
  \begin{align*}
    0=(\udisc[x+\stsizen e]-\udisc[x])+(\udisc[x-\stsizen e]-\udisc[x])
    &\geq (\utc(x+\stsizen e)-\utc(x))+(\utc(x-\stsizen e)-\utc(x))\\
    &\geq 2r\stsizen \norm{e}>0,
  \end{align*}
  which is impossible. On the other hand, if $x+\stsizen e\notin \domD$ (the case $x-\stsizen e\notin \domD$ is similar), then $x-\stsizen e\in \domD$ and $\udisc[x-\stsizen e]-\udisc[x]=-\stsizen\sigtg(e)$.
  The slope of $\utc$ is monotone in the direction $e$, and it cannot exceed $\sigtg(e/\abs{e})$, hence $\utc(x+te)=\utc(x)+t\sigtg(e)$ for $t\geq -\stsizen$. 
  But this contradicts the inequality
  \begin{align*}
    \utc(x+te)\geq \utc(x) +t\dotp{y_0}{e}+r\abs{t}\norm{e}.
  \end{align*}
We conclude that $\abs{\partial \utc(x)}=0$.
\end{proof}

\subsection{A semi-discrete optimal transport problem}
Let us define the following measures
\begin{equation}
  \mun\eqdef \sum_{x\in \gridn\cap \domD}\Fn[x]\delta_x, \quad \nun\eqdef \sum_{x\in \gridn\cap \domD}g(\gradn\udisc[x])\bun_{\partial \utc(x)}\Ll^2,
\end{equation}
where $\Fn[x]$ is defined for all $x\in \gridn\cap \domD$ as 
\begin{equation}
  \Fn[x]  \eqdef g(\gradn\udisc[x])\abs{\partial\utc(x)}\leq \dtg(\gradn\udisc[x])\MA(\udisc)[x]\Vn=\adsc(x)\Vn,
\end{equation}
$\adsc$ is the grid discretization of $f$ defined in section 4.5  and $\Vn=\stsizen^2$ is the area of one cell of the grid.

Let us note that from Proposition~\ref{prop:generalutc} and Lemma~\ref{lem:subgradh}, the Monge-Ampère measure (see~\cite{gutierrez}) associated with $\utc$ satisfies
\begin{align}
  \label{eq:ma-utc}
  \det(D^2\utc)\eqdef \abs{\partial \utc}=\sum_{x\in \gridn\cap \SC}\frac{\Fn[x]}{g(\gradn\udisc[x])}\delta_x.
\end{align}

\begin{lemma}\label{lem:vC1}
  \revision{There exists $\vlim\in\Cder^{1}(\RR^2)$, and non-negative absolutely continuous measures $\tmu\in\Mm(\RR^2)$, $\tnu\in\Mm(\RR^2)$, such that, up to the extraction of a subsequence,
  \begin{itemize}
    \item the function $\utc$ converges uniformly on compacts sets towards $\vlim$,
    \item  the measures $\mun$ and $\nun$ respectively weakly converge to $\tmu$ and $\tnu$ as $\stsizen\to 0^+$,
    \item $\tmu(\RR^2)=\tnu(\RR^2)>0$, and $\vlim$ is the minimal Brenier solution for the optimal transport of $\tmu$ to $\tnu$.
\end{itemize} }
\end{lemma}

\begin{proof}

Let us consider the transport plan $\gamn\in \Mm(\RR^2\times \RR^2)$
\begin{equation*}
  \gamn\eqdef \sum_{x\in \gridn\cap \domD}g(\gradn\udisc[x])\delta_x\otimes \bun_{\partial \utc(x)}\Ll^2
\end{equation*}
We note that $\gamn$ is an optimal transport plan between $\mun$ and $\nun$ since its support is contained in the graph of $\partial \utc$ (see Theorem~\ref{thm:villani}).

Moreover, we observe that 
\begin{align}
  \mun&\leq \Vn\sum_{x\in\SC\cap \gridn} \adsc(x)\delta_x=\sum_{x\in\SC\cap \gridn}\left(\int_{[x-\stsizen/2,x+\stsizen/2]^2}f(x+t)\d t\right)\delta_x, 
\end{align}
where the right hand-side converges towards $\mu$ in the weak-* topology (\ie induced by compactly supported continuous test functions). Additionally, since
\begin{align}
  \nun&\leq \supg \Ll^2\llcorner \TGn,
\end{align}
we deduce that the supports of $\mun$, $\nun$, $\gamn$ are respectively contained in the compact sets $\oSC$, $\TGn$, $\oSC\times \TGn$, and their masses are uniformly bounded. Hence, there exist Radon measures $\tmu$, $\tnu\in \Mm(\RR^2)$, $\tgam\in \Mm(\RR^2\times \RR^2)$ such that,  up to the extraction of a (not relabeled) subsequence,  $\mun$, $\nun$ and $\gamn$ respectively converge to $\tmu$, $\tnu$, $\tgam$ in the weak-* topology.  We note that $\tgam$ has respective marginals $\tmu$ and $\tnu$, and $\tmu$, $\tnu$ have densities with respect to the Lebesgue measure $\tf$, $\tg$ which satisfy
\begin{equation*}
  \tf\leq f \quad \mbox{and}\quad \tg \leq \supg\bun_{\oTG}.
\end{equation*}

Since $g(\gradn\udisc[x])\geq \infg$, we get 
\begin{align*}
  \nun\geq \infg\sum_{x\in\gridn\cap \domD}(\bun_{\partial\utc(x)}\Ll^2)=\infg \Ll^2\llcorner \TGn,
\end{align*}
and in the limit we deduce $\tg\geq \infg\bun_{\oTG}$. As a result, $\tmu(\RR^2)=\tnu(\RR^2)\geq \infg\abs{\TG}>0$.

 As for $\utc$, we already know from Lemma~\ref{lem:subgradh} that we may extract an additional subsequence so that $\utc$ converges uniformly on compact sets to some (convex function) $\vlim\in\Cder{}(\RR^2)$. Since any element of $\supp \tgam$ is the limit of  (a subsequence of) elements $(x_n,y_n$) of $\supp \gamn\subset \partial \utc$, passing to the limit in the subdifferential inequality 
 \begin{align*}
   \forall x'\in \RR^2,\quad \utc(x')\geq \utc(x_n) + \dotp{y_n}{x'-x_n},
 \end{align*}
 we obtain
\begin{equation*}
  \supp \tgam \subset \graph\partial \vlim.
\end{equation*}
As a result, and since $\tmu$ is absolutely continuous with respect to the Lebesgue measure $\Ll^2$, $\tgam=(I\times \nabla \vlim)_\sharp \tmu$ and $\nabla\vlim$ is the optimal transport map between $\tmu$ and $\tnu$.

Additionally, \revision{since $\partial \utc(\RR^2)=\oTGn$, we deduce that $\partial \vlim(\RR^2)\subseteq \oTG=\bigcap_{n\in\NN}\oTGn$. As a result, $\vlim$ is the minimal Brenier solution to the Monge-Amp\`ere problem from $\tmu$ to $\tnu$, hence  $\vlim\in \Cder^{1}(\RR^2)$.}

\end{proof}

\begin{lemma}\label{lem:numu}
  Assume that~\eqref{eq:cvgrad} holds. Then, with the notations of Lemma~\ref{lem:vC1},
$\tmu=\mu$ and $\tnu=\nu$.
\end{lemma}

\begin{proof} 
  For all $y\in \oTG\subset\TGn$, there exists $x\in \gridn\cap \domD$ such that $y\in \partial\utc(x)$. Except for a set of $y$ with zero Lebesgue measure, that $x$ is unique~\cite[Lemma 1.1.12]{gutierrez}, and $x\in\oSC$ (see Lemma~\ref{lem:subgradh}).
  Then, denoting by $\omega_g$ the modulus of continuity of $g$ over \revision{some neighborhood of} $\TG_{0}$, we get
\begin{align*}
  \abs{g(\gradn \udisc[x])-g(y)}\leq \omega_g\left(\max_{x'\in\gridn\cap \oSC}\max_{y'\in\partial \utc(x_n')}\abs{\gradn \udisc[x'])-y'} \right)\longrightarrow 0
\end{align*}
as $n\to+\infty$, where the convergence of the right-hand side follows from~\eqref{eq:cvgrad}.

As a result the following convergence holds
\begin{equation}\label{eq:gcvu}
    \lim_{n\to+\infty}\norm{\sum_{x\in \gridn\cap \domD}g(\gradn\udisc[x])\bun_{\partial \utc(x)}-g}_{L^\infty(\oTG)}=0.
  \end{equation}

  We deduce that $\tg\equiv g$ on $\oTG$ and from~\eqref{eq:gcvu} and~\eqref{eq:hausdorff} we obtain
\begin{align*}
  \int_{\RR^2}\tf(x)\d x&= \lim_{n\to+\infty}\mun(\RR^2)\\
                        &=\lim_{n\to+\infty}\nun(\RR^2)\\
                        &= \lim_{n\to+\infty} \int_{\TGn}\left(\sum_{x\in \gridn\cap \domD}g(\gradn\udisc[x])\bun_{\partial \utc(x)}(y)\right)\d y\\
                        &= \int_{\oTG}g(y)\d y =\int_{\RR^2} f(x)\d x.
\end{align*}
Since, $\tf\leq f$, we get $f=\tf$.
\end{proof}

To sum up, we are now in position to prove the following result.

\begin{theorem}\label{thm:convergence}
  As $n\to+\infty$, the function $\utc$ converges uniformly on compact sets towards the unique minimal Brenier solution $\ue$ which satisfies $\ue(0)=0$.
\end{theorem}

\begin{proof}
  Up to the extraction of a sequence, we obtain from Lemma~\ref{lem:vC1} and~\ref{lem:numu} that $\utc$ converges uniformly on compact sets towards some convex function $\vlim\in\Cder^{1}(\RR^2)$ which satisfies in the Brenier sense
  \begin{align*}
    \det(D^2\vlim)= f(x)/g(\nabla v(x)), \quad \mbox{and}\quad \partial \vlim(\RR^2)=\oTG.
  \end{align*}
However, the function $v$ may depend on the choice of the subsequence. By Proposition~\ref{prop:defue} which states the uniqueness of the minimal Brenier solution up to an additive constant, we can prove that $v=\ue$  by proving that $v(0)=0$. Then we obtain that the full family $(\utc)_{n\in\NN}$ converges towards $\ue$ by uniqueness of the cluster point.

We know that $v$ is a solution to the Monge-Ampère equation, hence (in our setting) also in the Aleksandrov sense.
As a result, for all $r>0$, $\int_{B(0,r)}f\d \Ll^2= \int_{\partial v(B(0,r))}g\d\Ll^2$, and
\begin{align*}
  0<\inff\abs{B(0,r)} \leq \supg\abs{\partial v(B(0,r))}.
\end{align*}
But by weak-convergence of the Monge-Ampère measures~\cite[Lemma~1.2.2]{gutierrez},
\begin{align}
  \label{eq:}
  \abs{\partial v(B(0,r))}\leq \liminf_{n\to+\infty} \abs{\partial \utc\left(B(0,r)\right)}=\liminf_{n\to+\infty} \abs{\partial \utc\left(B(0,r)\cap \gridn\right)}
\end{align}
As a result, for all $n$ large enough, there exists $x_{n,r}\in B(0,r)\cap \gridn$ such that $\abs{\partial \utc(x_{n,r})}>0$, and thus, by Proposition~\ref{prop:generalutc}, $\utc(x_{n,r})=\udisc[x_{n,r}]$. By a diagonal argument we construct a sequence $x_{n}^*$ such that $x_{n}^*\to 0$ and $\utc(x_{n}^*)=\udisc[x_{n}^*]$. By Proposition~\ref{prop:udisc}
\begin{align}
  \abs{\utc(x_{n}^*)}=\abs{\udisc[x_{n}^*]} \leq C\normu{x_{n}^*}.
\end{align}
Passing to the limit $n\to+\infty$, we get $v(0)=0$. As a result $v=\ue$ and the full family converges towards $\ue$.
\end{proof}

\section{Numerical study } 
\label{sec:numerics}

The numerical method, as we implemented it, depends on two main parameters~: 
\begin{itemize} 
\item $h$ the step size of the  cartesian grid discretizing the square $D$ wich contains the support of $f$. 
\item The stencil width, i.e. the maximum norm of the vectors on the grid $\ZZ^2$  which will be used 
in the variational schemes.  More precisely we use superbases with 
vector with maximum norm given by the stencil width in (\ref{eq:SB}), we use the same set of vectors for 
$V(x)$ in (\ref{eq:SB}) and finally we also use the same vectors to define $\sten$ in (\ref{eq:discrTG}) for 
the discretization of the target. 
\end{itemize} 

There are also parameters linked to the precision of the damped Newton algorithm, but they have 
a very limited impact on the efficiency of the method. \\

We run the Newton method until the residual is stationary in $\ell^\infty$ norm which usually takes a few dozen interations. The value of the residual 
depends on the discretization parameter but in all of our experiments was always between 1e-10 and 1e-15. \\

The target density $g$ is defined in all $\R^2$ using a constant extension.  In practice 
we initialized with a potential $u$ which gradient maps inside the target $\TG$ but it cannot 
be excluded to hit out along the iterations.

\revision{%
  \subsection{Imposing the constraint $\udisc[0]=0$}
  A straightforward implementation of~\eqref{eq:systemMAdisc} yields a non-square nonlinear system. When performing the Newton iterations with a Moore-Penrose pseudo-inverse of the Jacobian, we have observed convergence towards a function which approximately solves the Monge-Amp\`ere equation in $\SC$, but yields a stronger error in $\domD\setminus\SC$. In other words, 
  \begin{align}
    \label{eq:error}
    \mathcal{E}[x]\eqdef  \begin{cases}
    \MA(\udisc)[x] - \frac{\adsc[x]}{\dtg((\gradn\udisc)[x])} = 0 & \mbox{ if $x\in \SC$}      \\
    \tMAo(\udisc)[x]=0 & \mbox{ if  $x\in D\setminus\SC$}      \\
  \end{cases}
  \end{align}
  is small in $\SC$, but may be large in some parts of $\domD\setminus \SC$ (see Figure~\ref{error}, left). 

  An alternative way to impose the condition  $\udisc[0]=0$ is to rely on the mass conservation property. 
  Instead of trying to cancel $(\mathcal{E}[x])_{x\in \gridn\cap \domD}$ and $\udisc[0]$ separately, one tries to cancel $(\mathcal{E}[x]+\udisc[0])_{x\in \gridn\cap \domD}$. In principle, the mass conservation then implies that $\udisc[0]=0$. However, at the discrete level, the mass conservation is only approximate, since the densities and the target $\oTG$ are discretized. As a result $\udisc[x_0]$ is not exactly $0$ in practice.

  Numerically, we have noticed that this alternative approach has better convergence properties, and yields an error which is smaller and uniformly spread (see Figure~\ref{error}, right). In fact, up to some small residual, the error $\mathcal{E}[X]$ is uniformly equal to $\udisc[0]$ which is slightly different from zero because of the inexact mass conservation.
  Although the straightforward approach has similar consistancy properties as the alternative one when $\stsize\to 0$, in the following we describe the numerical results of the alternative approach, as it shows better numerical properties.
  We refer to the quantity $\left(\mathcal{E}[x]+\udisc[0]\right)_{x\in \gridn\cap \domD}$ as the residual error \emph{Res}.
  
\begin{figure}
\centering
\includegraphics[width=0.45\linewidth]{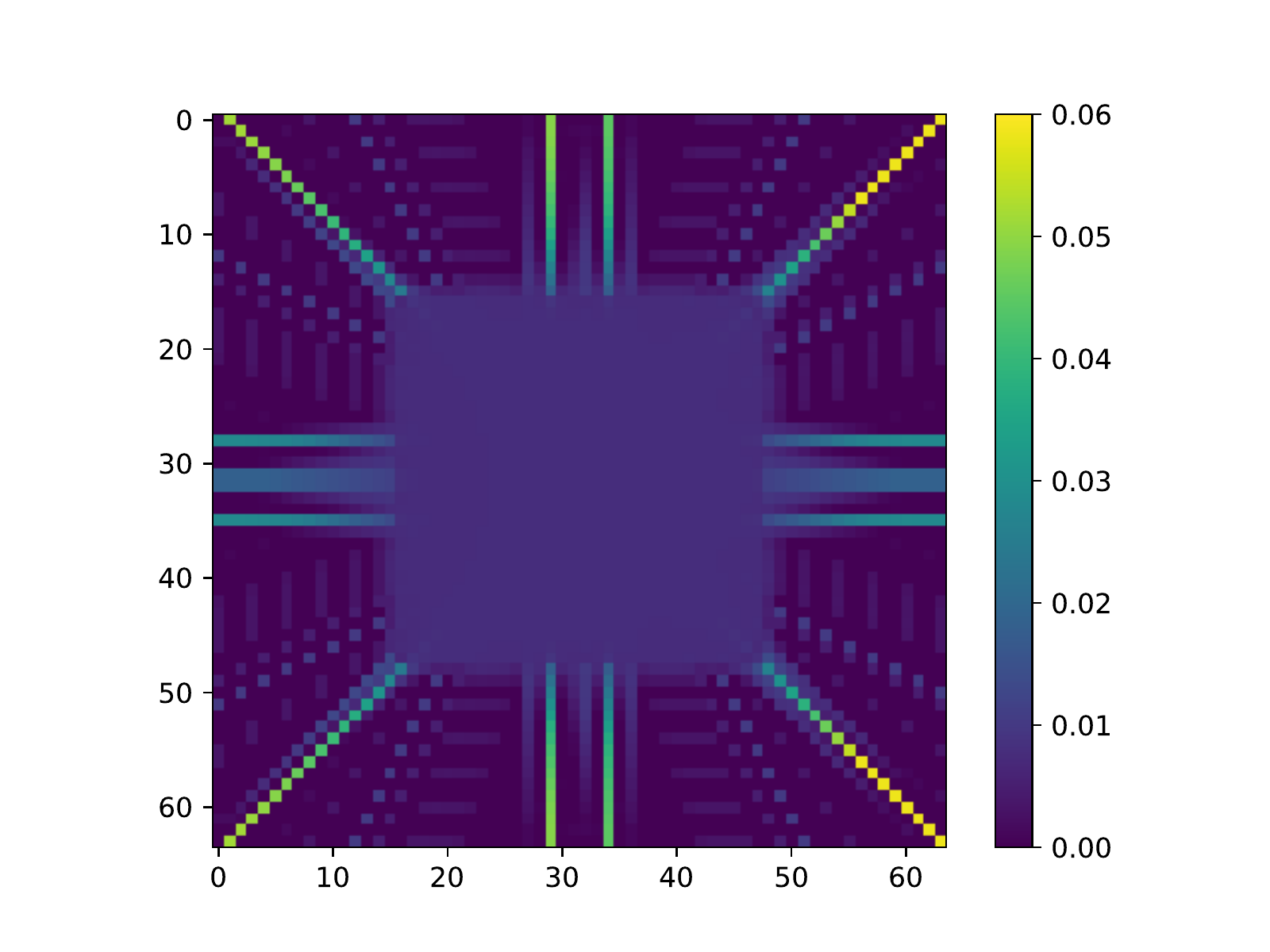}
\includegraphics[width=0.45\linewidth]{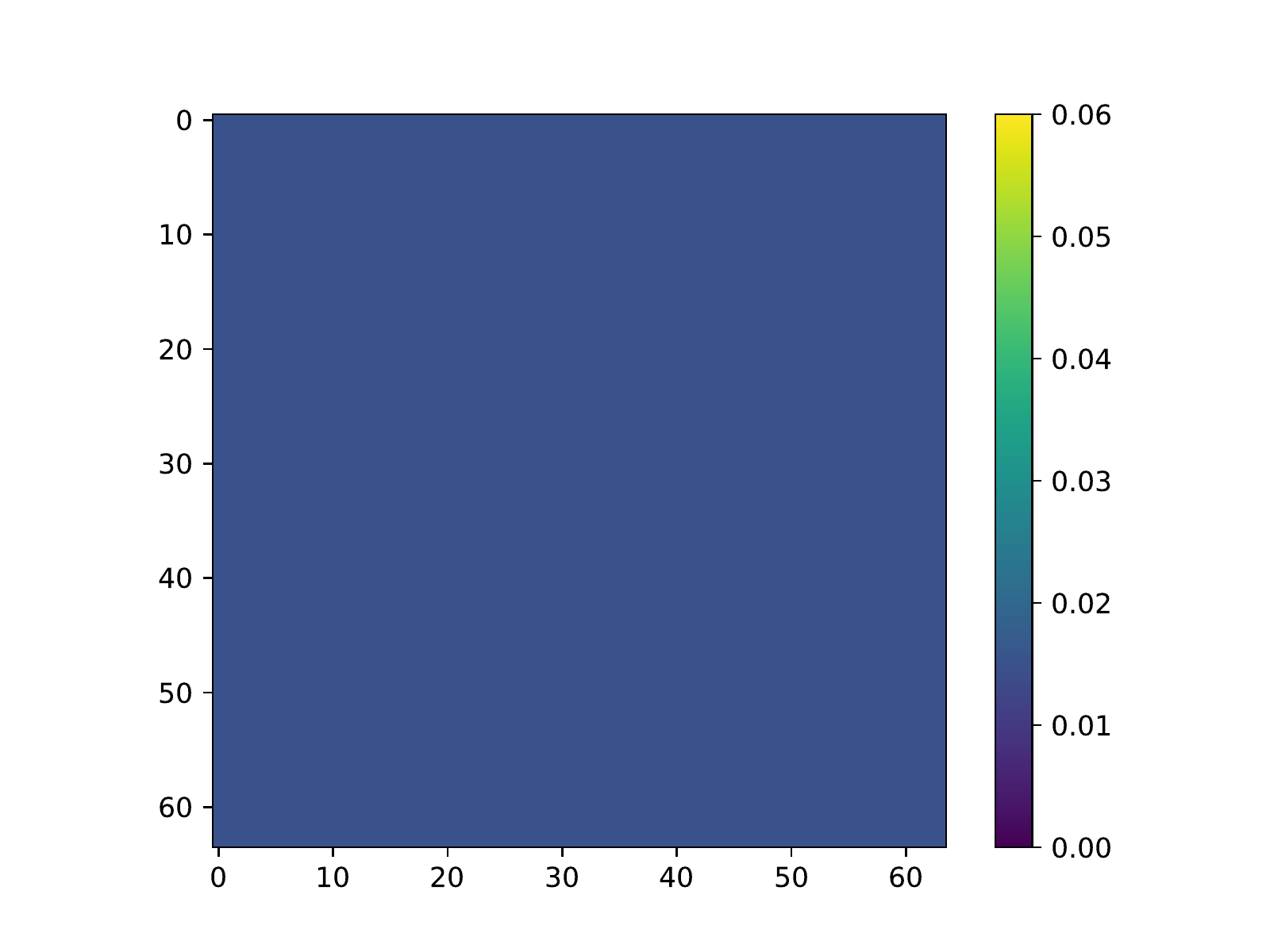} 
\caption{Error $\mathcal{E}[x]$ when mapping a square to a disc. The plain implementation of~\eqref{eq:systemMAdisc} yields some high error in $\domD\setminus \oSC$ (left), whereas the alternative approach yields a uniform small error.  }
\label{error} 
\end{figure}

 }

\subsection{ Experiment with $h$ } 
\label{sec:exph}
A standard test is to optimally map a uniform density square to a uniform density circle. 
We use here a  stencil width of $5$ (\ie{} we use vectors on the grid $\ZZ^2$ of maximum $\ell^\infty$ norm $5$). \\


$\SC$ is a square $[0.2,0.8]^2 $ immersed in a bigger square $D = [0, 1]^2$. \\

Table \ref{tbl1} shows, for different values of $h$, the number of iterations 
the norms of the reached residuals and the value of $\udisc[0]$. \\


Figure~\ref{sq2ball} shows the deformation of the computational grid under the gradient 
map. Notice that the grid volume is well preserved and that all grid points in $D\setminus \SC$ 
are collapsed onto the boundary of $\TG$ which is approximated as a polygon (see next section). 

\begin{table}[h]
  \begin{center}
    \renewcommand{\arraystretch}{1.2}
\begin{tabular}{ | c | c | c | c | c | }
\hline    $N = \frac{1}{h} $   &  \# It.  &   $\| Res. \|_{\infty}$  & $\| Res. \|_{2}$  & $\udisc[0]$   \\
\hline                 128             &     32       &         4e-12                         &      7e-13                      &    1.4e-03     \\
\hline                 256           &      72    &         3e-12                         &      3e-11                      &    4.5e-02     \\
\hline                 512             &     103      &         1e-10                         &      1e-11                     &    1.7e-02     \\
\hline                 1024             &     209      &         3e-10                         &      5e-11                      &    1.1e-02     \\
\hline          
\end{tabular} 
\end{center}
\caption{ Number of iterations / Norms of the reached residuals / Value of forced constant } 
\label{tbl1} 
\end{table} 
 
On a standard Laptop a non optimized Julia implementation needs less than 5 minutes to solve the \revision{$512 \times 512$} case and 
 1.5 hours for the $1024 \times 1024$ case. 
 
\begin{figure}
\centering
\includegraphics[width=0.45\linewidth]{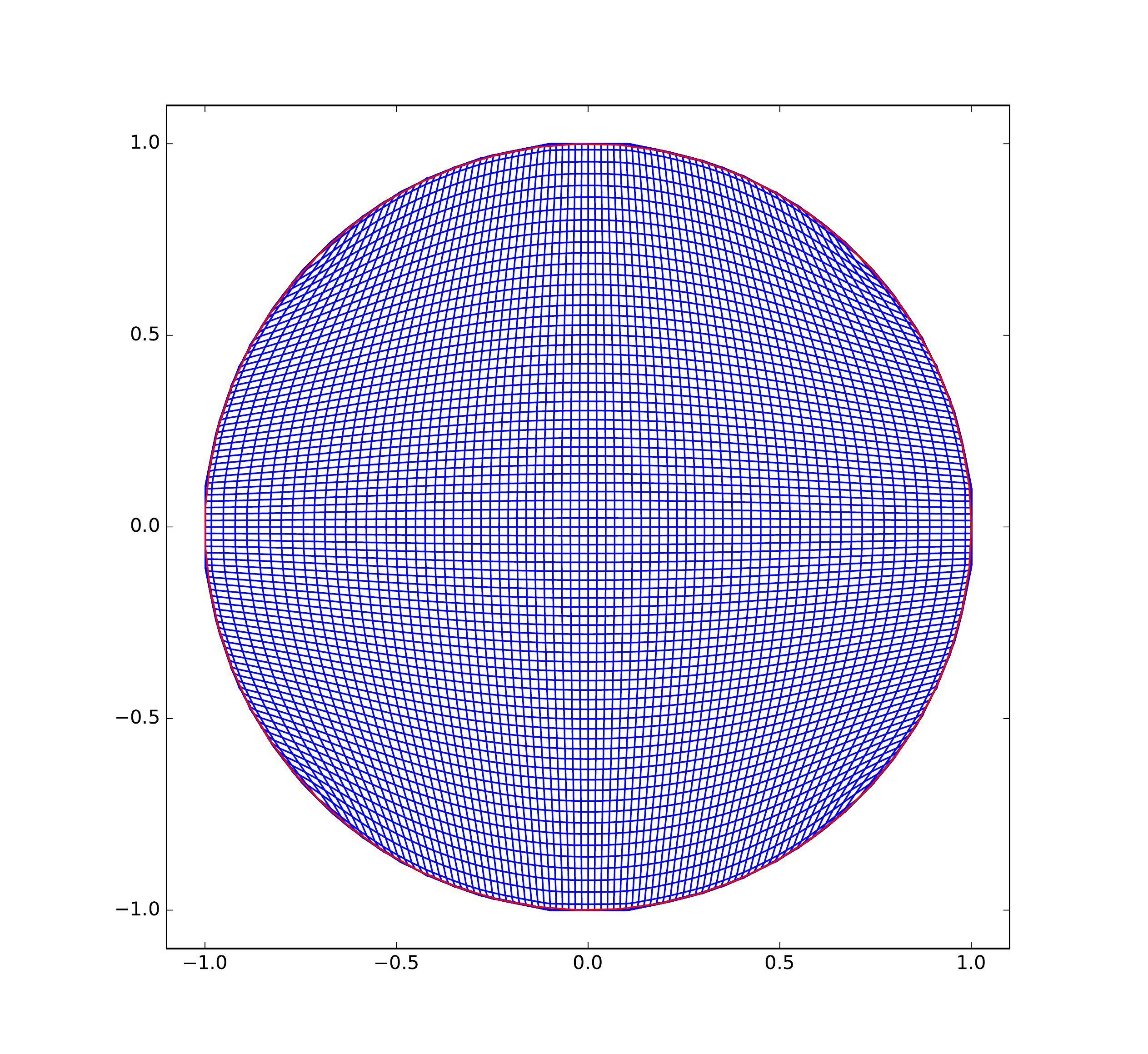}
\includegraphics[width=0.45\linewidth]{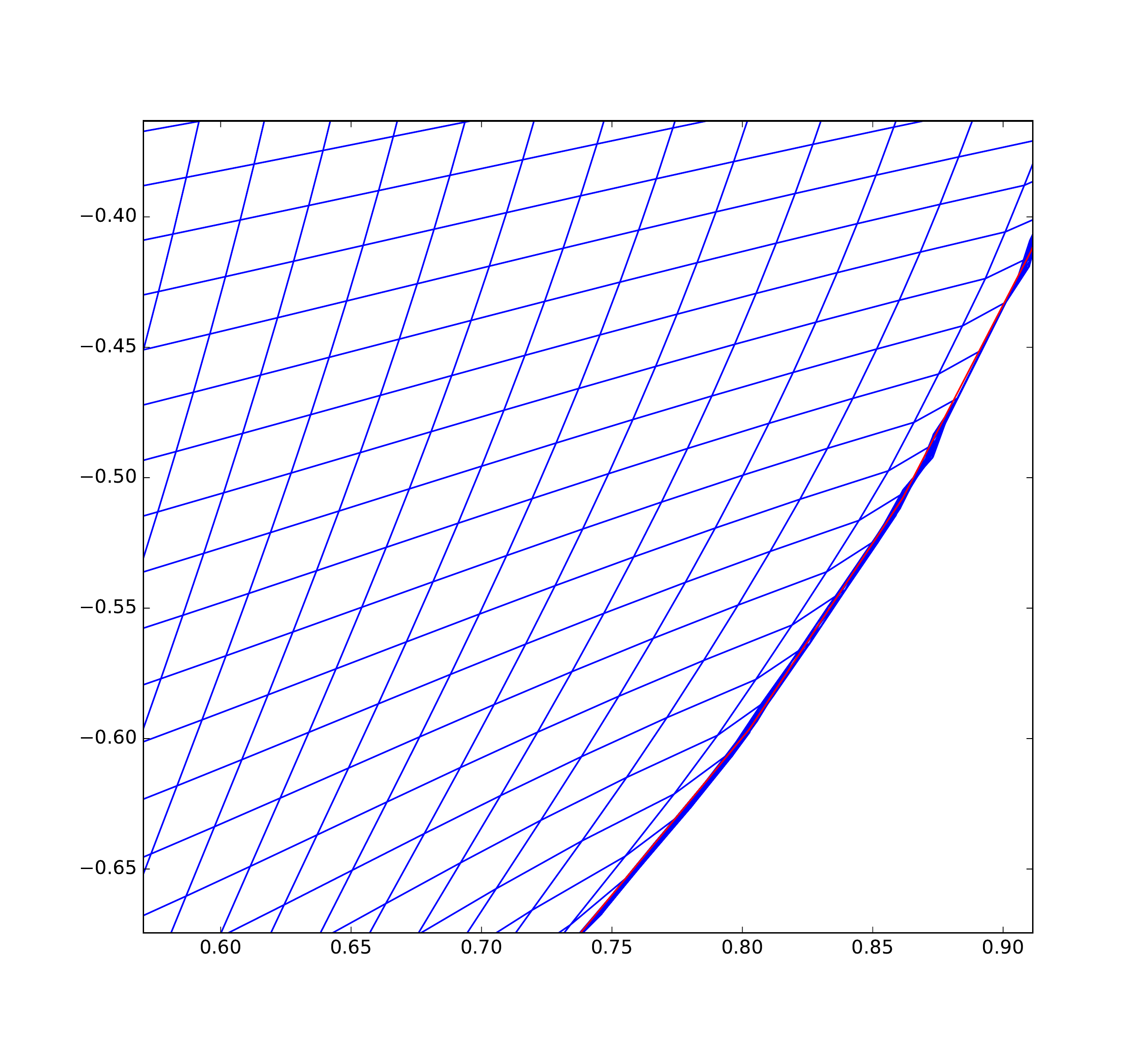} 
\caption{Left: Map deformation of the cartesian grid for square to ball test case. Right: Zoom of the Map deformation of the cartesian grid (as expected, $D\setminus \SC$ is mapped to $\partial \TG$). }
\label{sq2ball} 
\end{figure}

 \subsection{ Experiments with stencil width} 
 \label{sec:stencil}
In all this section $N = 128$, and we vary the stencil width, meaning the number of vector on the  
grid used in scheme. The target is an heptagon whose normals directions 
 are not necessarily vectors in the stencil.  \\

Table \ref{tbl2} shows, for increasing values of the stencil width , the number of iterations 
the norms of the reached residuals and the value of the $\udisc[0]$, which 
decreases as the accuracy of the discretization of  $\TG$ improves. \\

Figure \ref{sq2hpt2} to   \ref{sq2hpt16}  shows the deformation of the computational grid under the gradient 
map and zooms. They show how the geometry of the computed target improves with 
the domain discretization which also depends on the stencil width.  Notice again 
that grid volume is well preserved and that all grid points in $D\setminus \SC$ 
are collapsed onto the boudary of $\TG$.

In Section~\ref{sec:nonconvex} below, we provide another experiment which shows that the MA-LBR schemes uses very small stencil widths.

\begin{table}[h]
  \begin{center}
\begin{tabular}{ | c | c | c | c | c | }
\hline    S. width   &  \# It.  &   $\| Res. \|_{\infty}$  & $\| Res. \|_{2}$  & $\udisc[0]$   \\
\hline                 2           &     32       &         9e-12                         &      1.5e-12                      &    1.5e-01     \\
\hline                 4           &      42    &         9e-12                         &      6e-12                      &    6.4e-02     \\
\hline                 8             &     39    &         8e-12                        &      4e-12                     &    3.9e-02     \\
\hline                 16             &     42     &         4e-12                         &      7e-13                      &    2.2e-02     \\
\hline          
\end{tabular} 
\end{center}
\caption{  Stationary residual reached in $\ell^\infty$ and $\ell^2$ norms / Number of iterations  to reach stationary minimal residual   / Value of forced constant}
\label{tbl2} 
\end{table}

\begin{figure}[hp]
\centering
\includegraphics[width=0.45\linewidth]{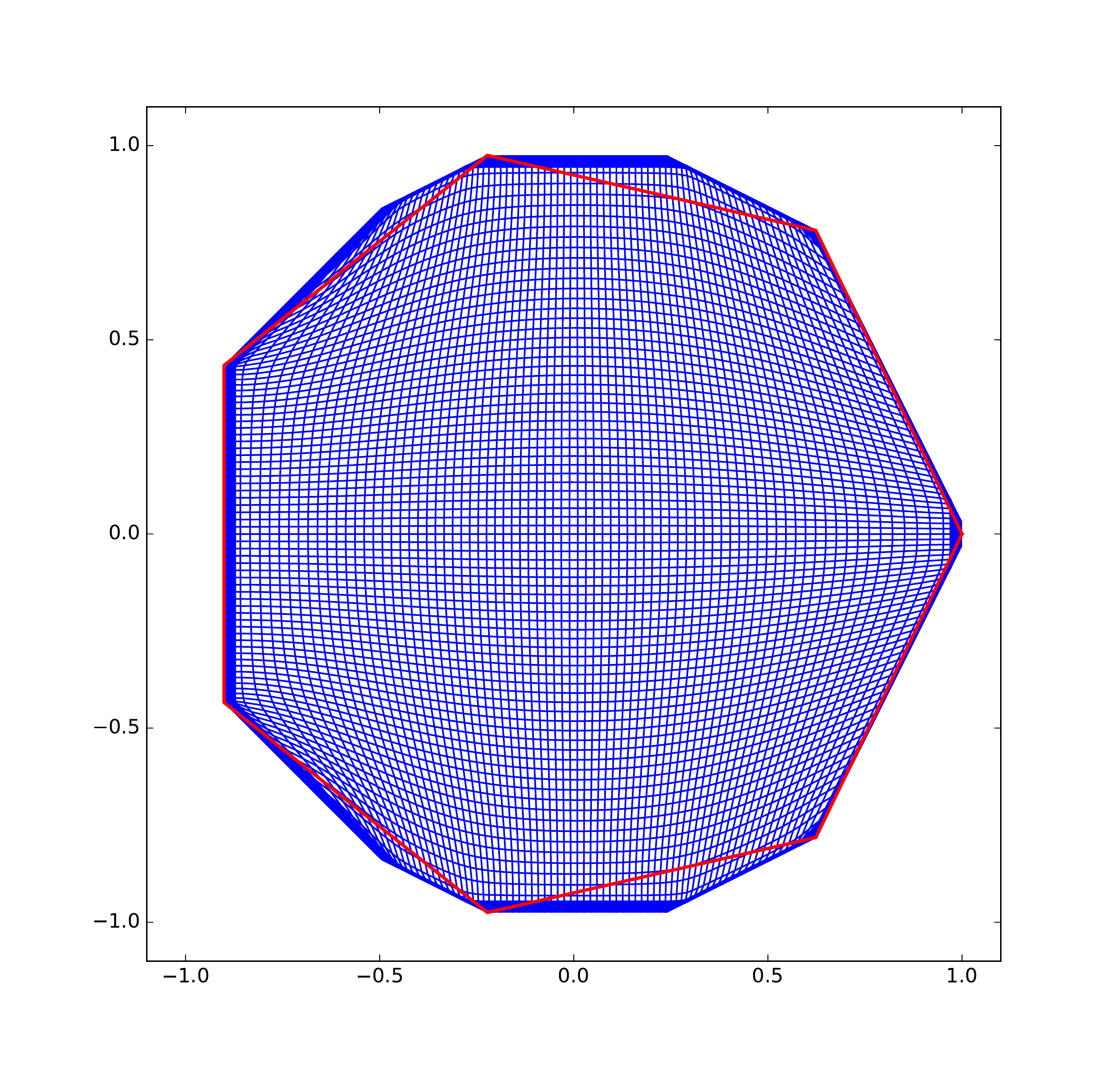} 
\includegraphics[width=0.45\linewidth]{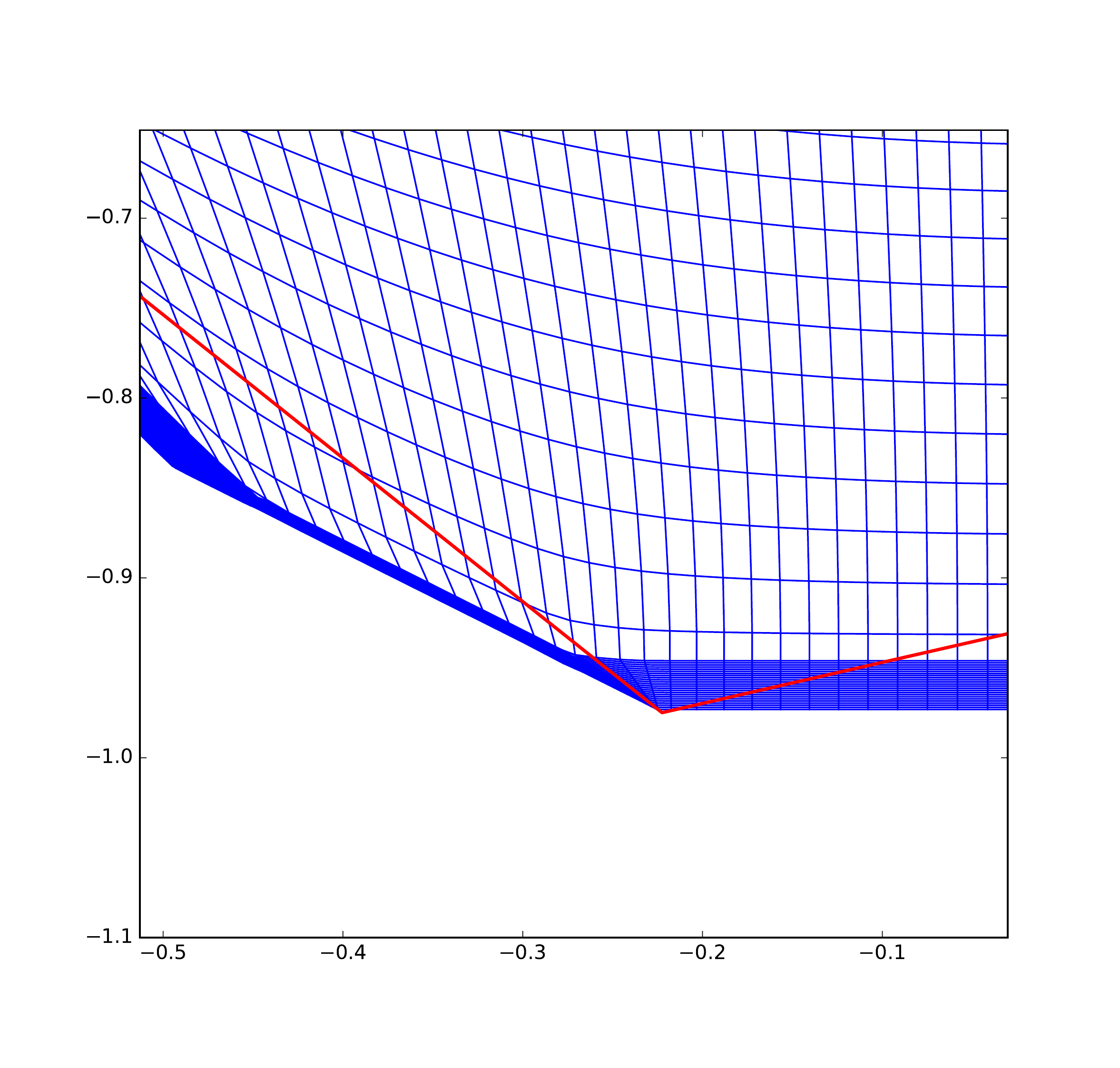} 
\caption{Left: Mapping a square to heptagon, Stencil width = 2. Right: zoom. }
\label{sq2hpt2} 
\end{figure}

\begin{figure}[hp]
\centering
\includegraphics[width=0.45\linewidth]{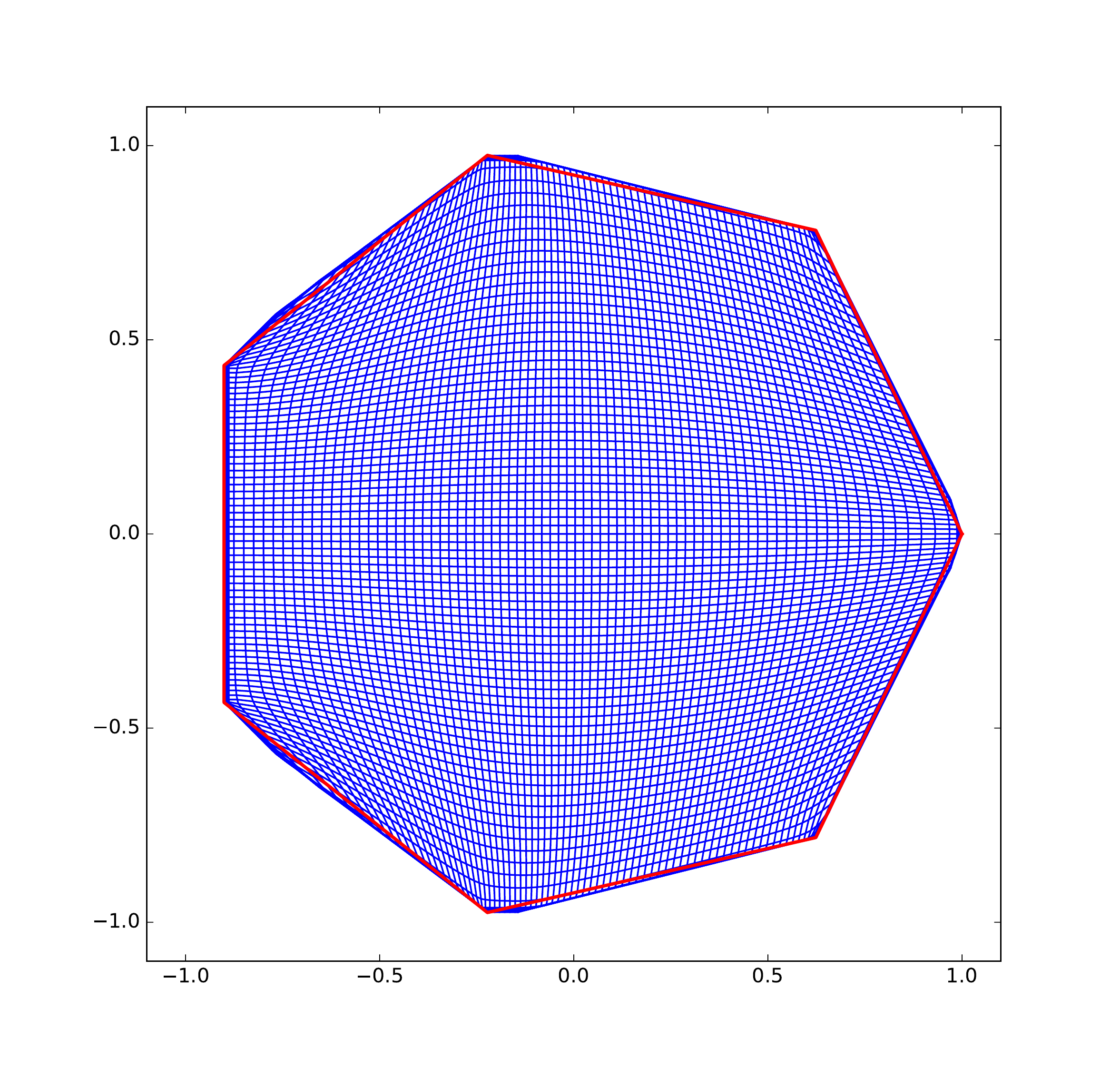} 
\includegraphics[width=0.45\linewidth]{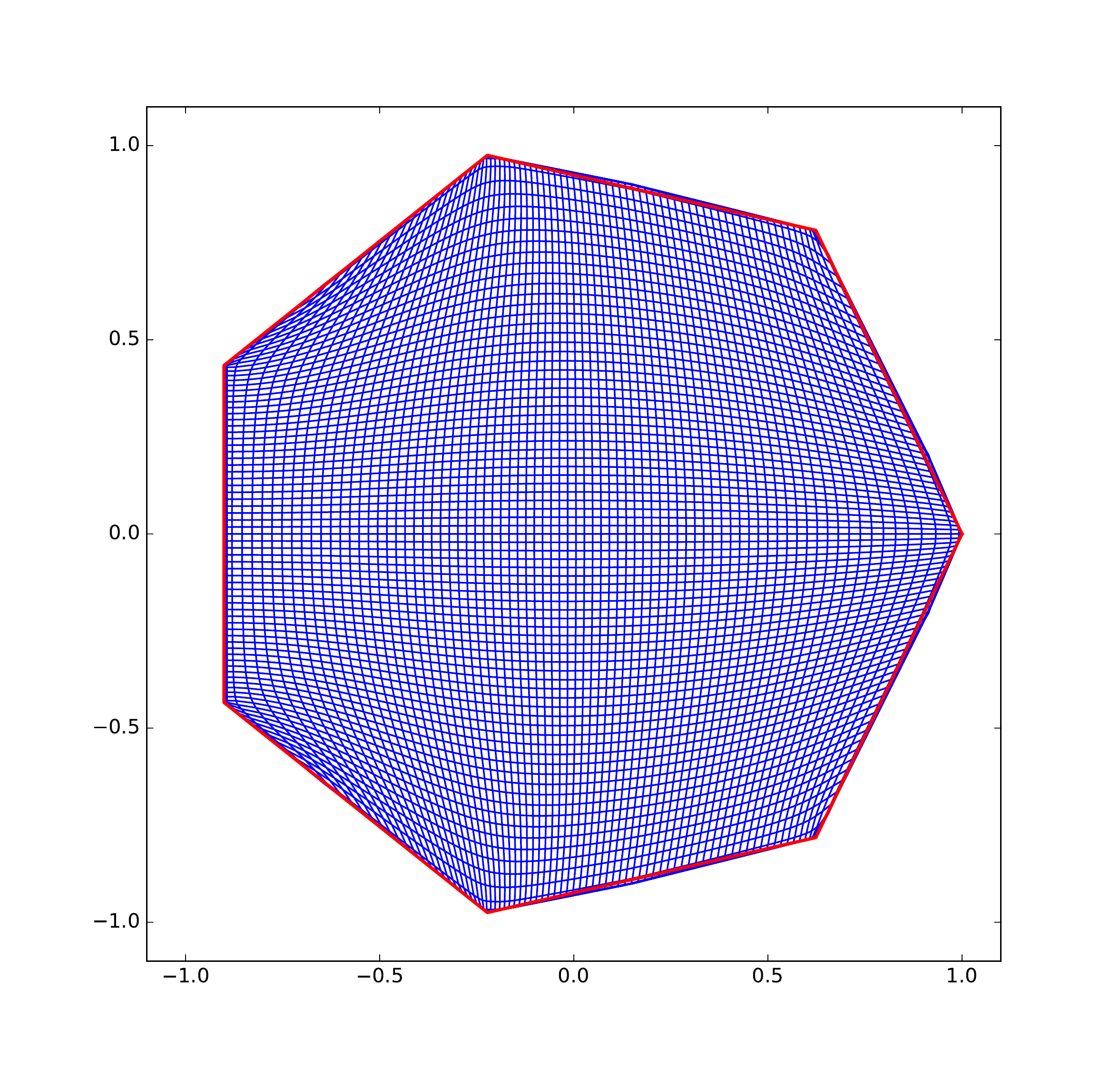} 
\caption{Mapping a square to  heptagon. Left: stencil width = 4. Right: stencil width = 8. }
\label{sq2hpt4} 
\end{figure}

\begin{figure}[hp]
\centering
\includegraphics[width=0.45\linewidth]{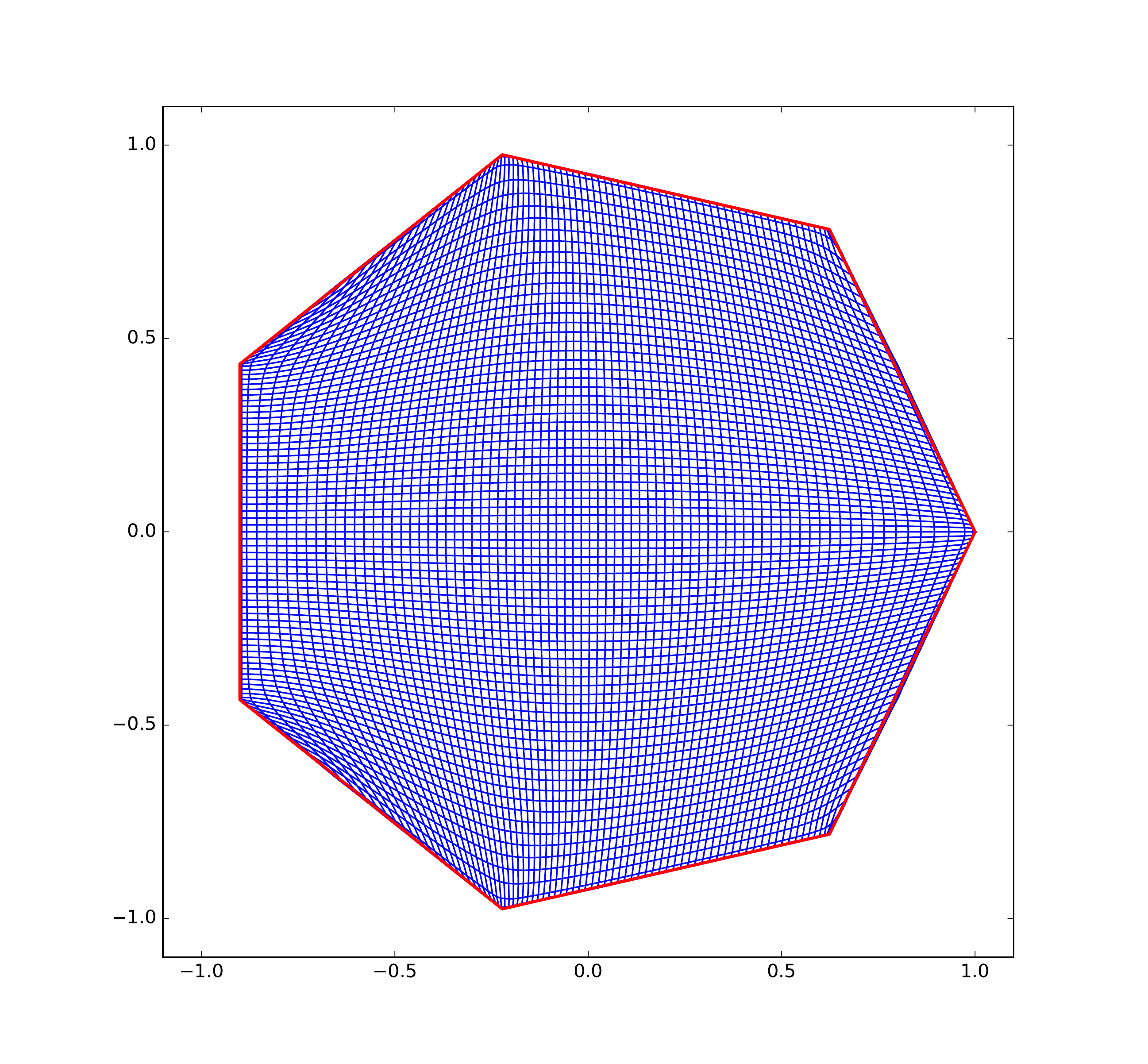} 
\includegraphics[width=0.45\linewidth]{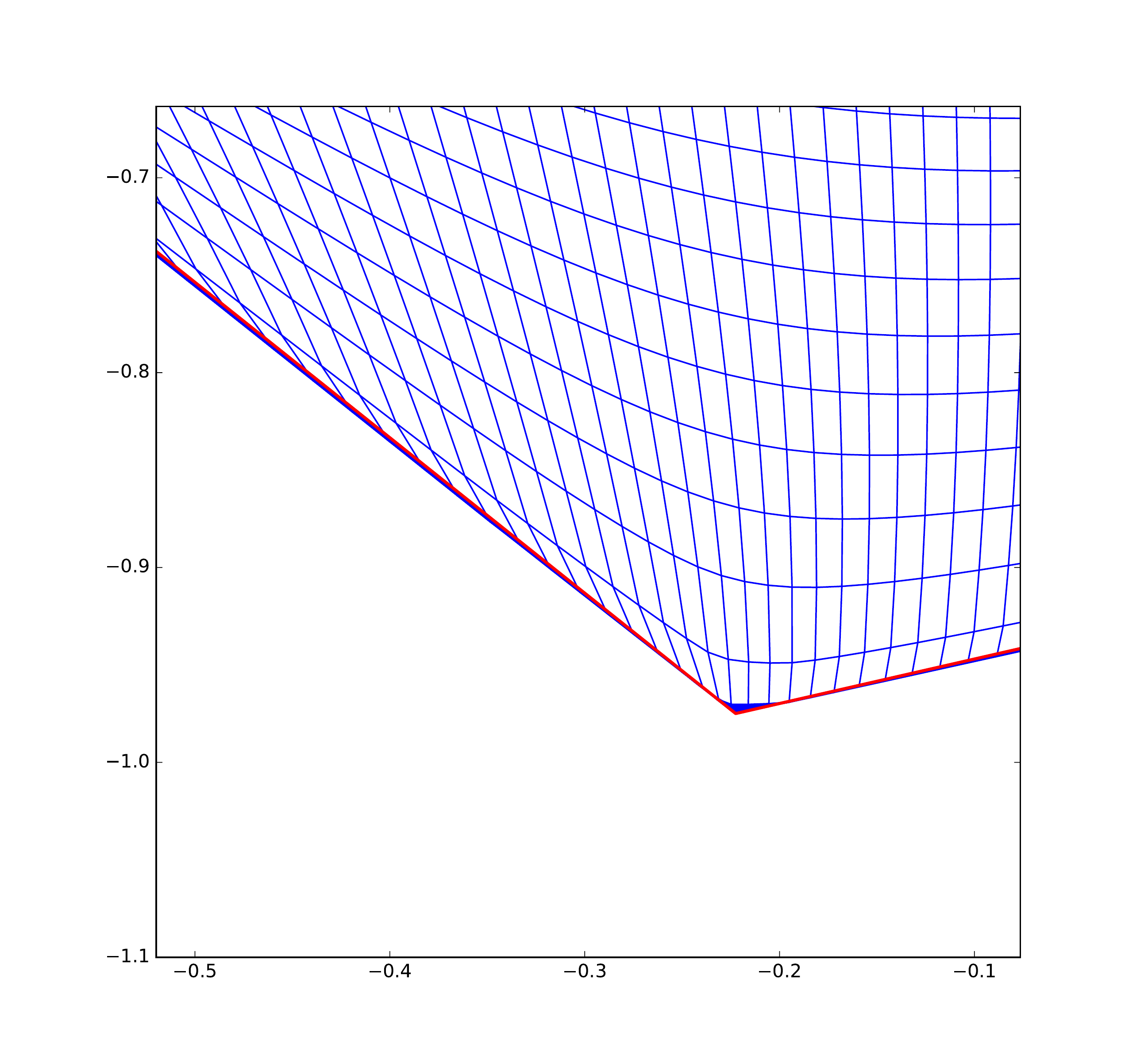} 
\caption{Left: Mapping a square to  heptagon, Stencil width = 16. Right: zoom. }
\label{sq2hpt16} 
\end{figure}

\clearpage

 \subsection{Experiments with inhomogeneous source density $f$ } 
 \label{sec:inhomo}

 \revision{ 
We start with a test case taken from \cite{bfo} and for which the anaytical solutions is given. \\
Set 
\[
q(z) = \left( - \frac{1}{8\pi} z^2 + \frac{1}{256 \pi^3} + \frac{1}{32 \pi}\right) \cos(8 \pi z) + \frac{1}{32 \pi^2} z \sin(8 \pi z) 
\]
The map from the density
  \[
    f(x_1,x_2) = 1+4 \, \left(q''(x_1) q(x_2) + q''(x_2) q(x_1)\right) + 16 \, \left( q(x_1) q(x_2) q''(x_2) q''(x_1)  -  q'(x_2)^2 q'(x_1)^2 \right)
  \]
  defined on the square $(-0.5, 0.5)^2$ onto a uniform density in the same
square  has the exact solution
  \[
  u_{x_1}(x_1,x_2) = x_1 +4 \, q^{'}(x_1) q(x_2)  \quad \quad  u_{x_2}(x_1,x_2) = x_2 +4 \, q^{'}(x_2) q(x_1) 
  \]

In table \ref{tblb} we give the error in relative $L^1$ nom on the gradient for different values of $N$, the space discretization and Stencil width. 
Convergence is set for a residual at least of $1e-10$.  Two main observations can me made:
First the results clearly indicate a first order accuracy on the gradient of the solution. Second and as already discussed in Section~\ref{sec:stencil} and 
\cite{malbr}, if the condition numbers of the Hessian of the solution are bounded, then the MA-LBR accuracy only depends on $h$. In this test case 
the directions contained in the stencil of length $2$ are sufficient to represent the  positive definite Hessians of the exact  solution. 

\begin{table}[h]
  \begin{center}
\begin{tabular}{ | c | c | c | c | }
\hline \diagbox[width=10em]{N}{S. Width}    &  2   &   4 &   6    \\
\hline                      64           &       1.6425e-2  &  1.6425e-2 &        1.6425e-2                       \\
\hline                      128           &      0.8045e-2  &   0.8045e-2 &       0.8045e-2                           \\
\hline                      256          &     0.3966e-2  &  0.3966e-2   &           0.3966e-2                      \\
\hline                      512         &      0.1968e-2  &   0.1968e-2&        0.1968e-2                      \\
\hline          
\end{tabular} 
\end{center}
\caption{  Error in relative $L^1$ norm on the gradient for different values of $N$, the space discretization and Stencil width}
\label{tblb} 
\end{table} 
}

 In the next example, the  heterogeneous densities $f$ mapped to a constant 
 density ball $\TG$, $N = 128$ and stencil width is $5$ .  
 
 Figures~\ref{inhomo1} and \ref{inhomo2} show heterogeneous sources and the corresponding deformation of the cartesian 
 grid.

 \begin{figure}[h]
\centering
\includegraphics[width=0.45\linewidth]{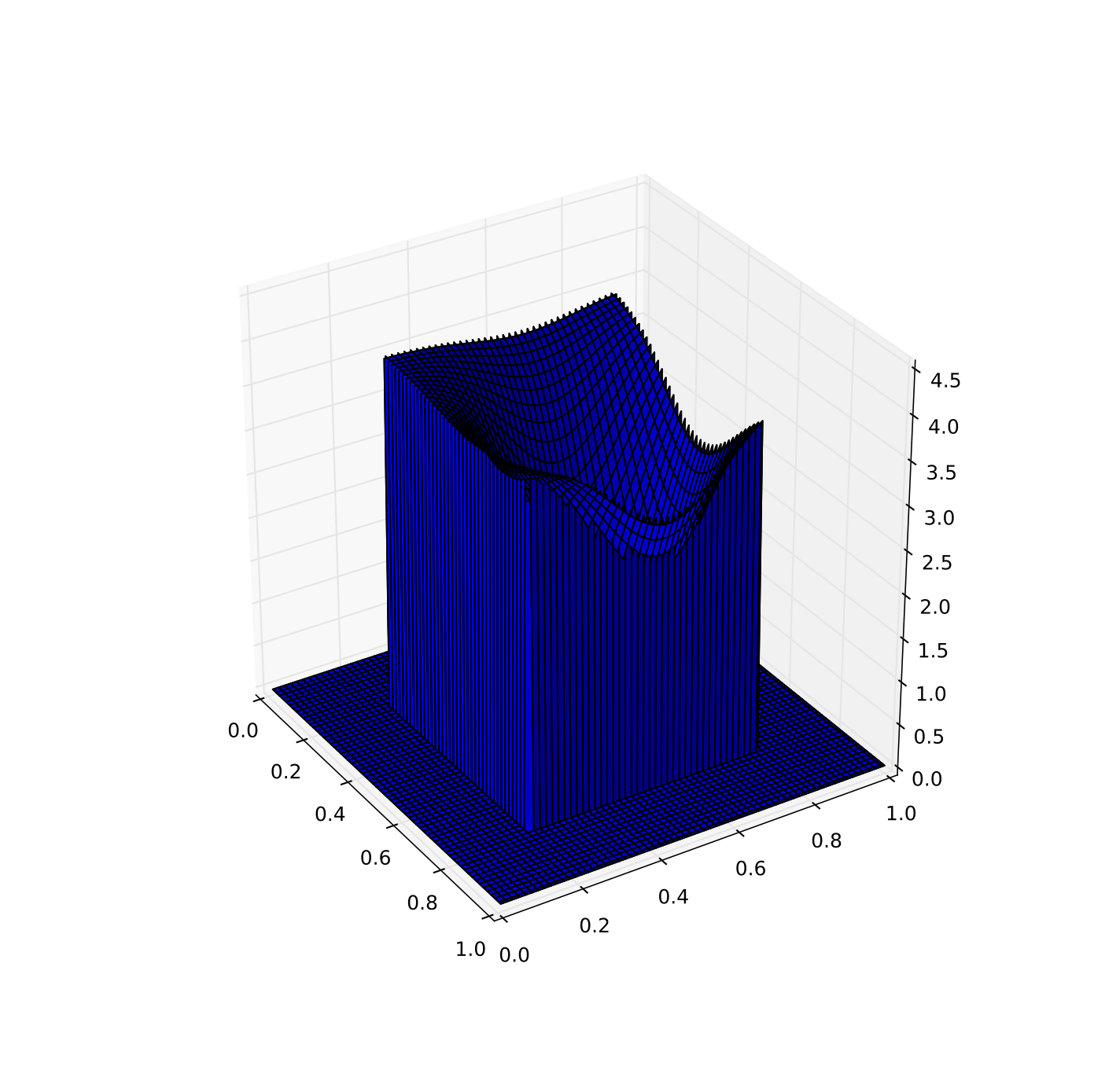} 
\includegraphics[width=0.45\linewidth]{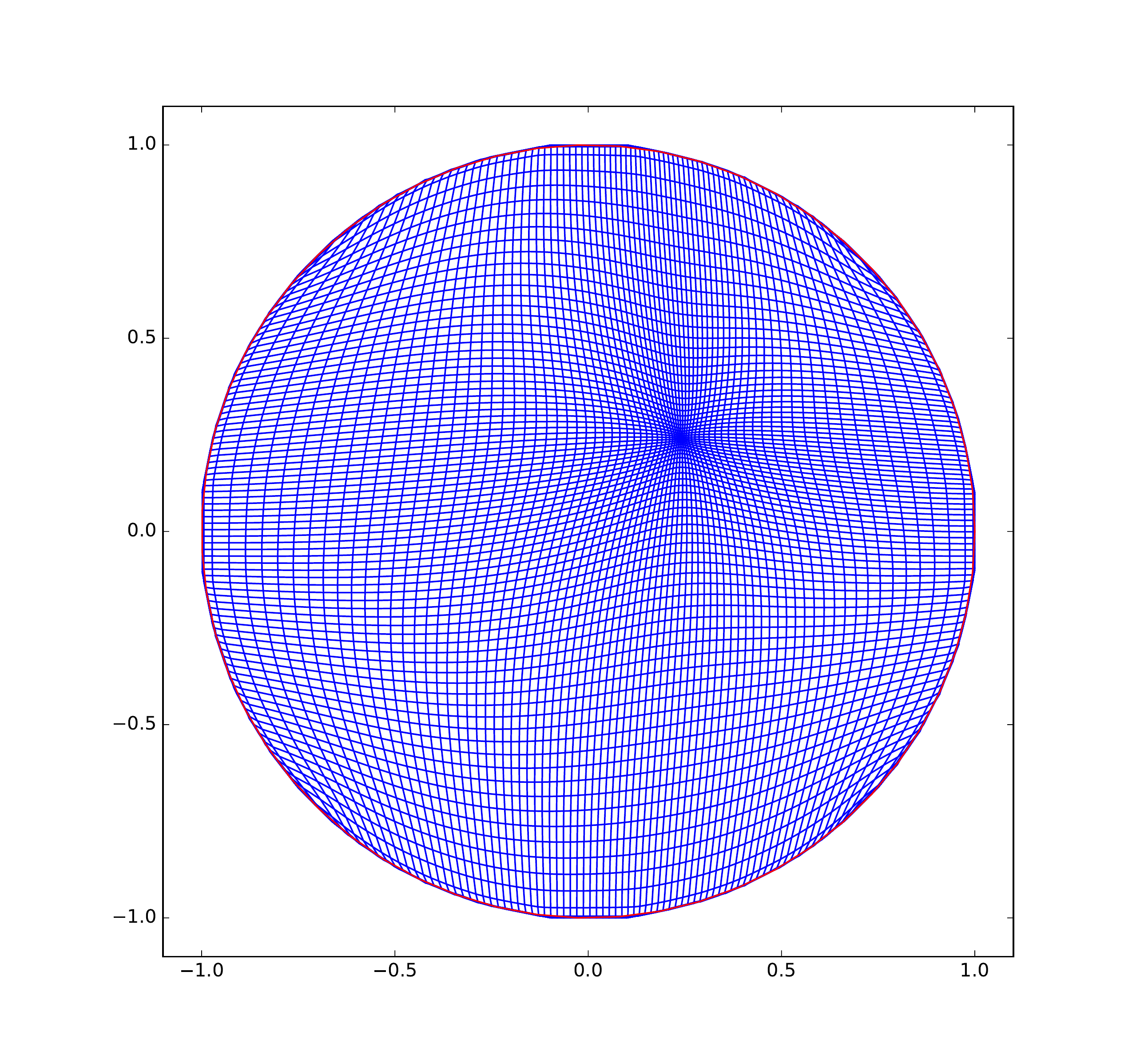}\\ 
\includegraphics[width=0.45\linewidth]{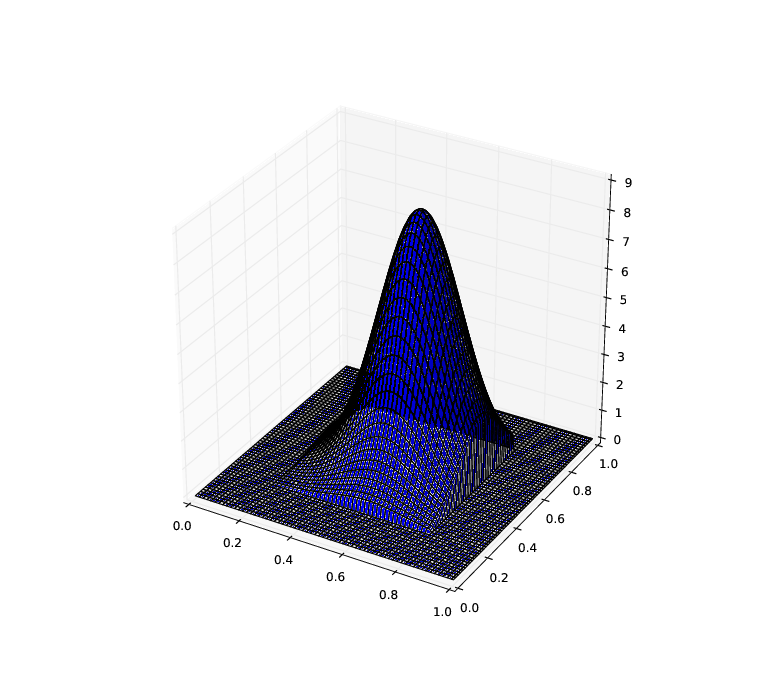}
\includegraphics[width=0.45\linewidth]{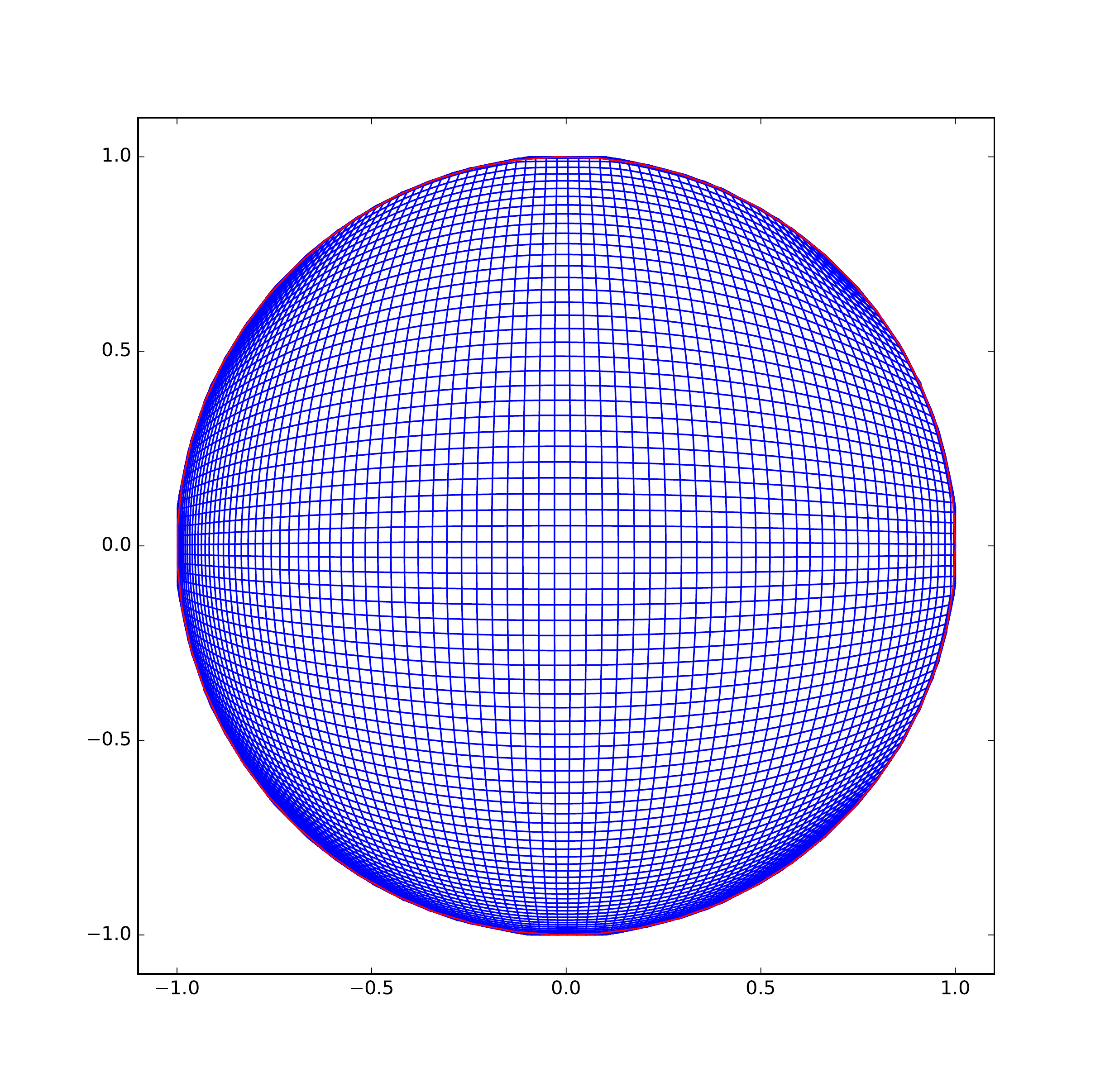} 
\caption{Left column: Source density $f$. Right column: Corresponding optimal map deformation of the grid.}
\label{inhomo1}
\end{figure}

In  Figure~\ref{inhomo2} the source density is random on the square. This test case demonstrates the applicability of the 
method to Lebesgue integrable densities.  

 \begin{figure}[h!]
\centering
\includegraphics[width=0.45\linewidth]{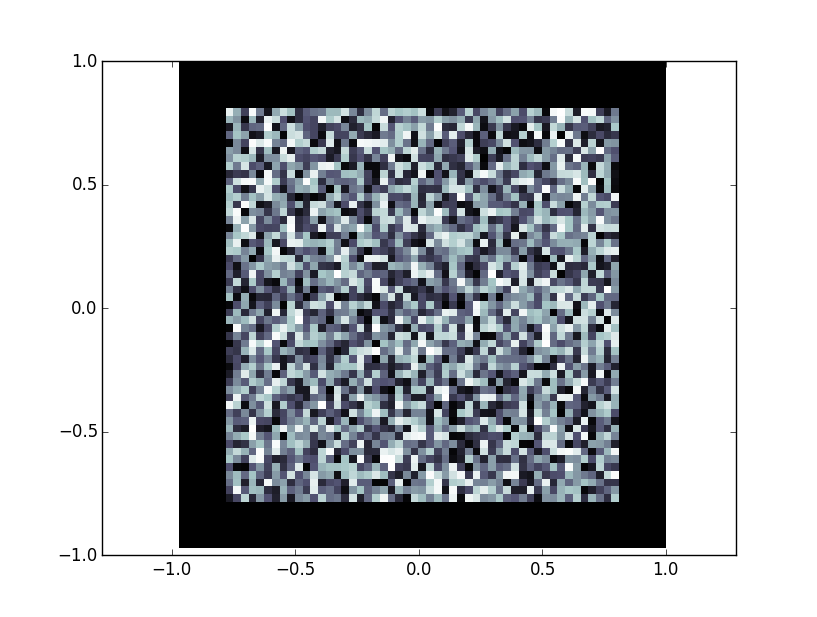}
\includegraphics[width=0.45\linewidth]{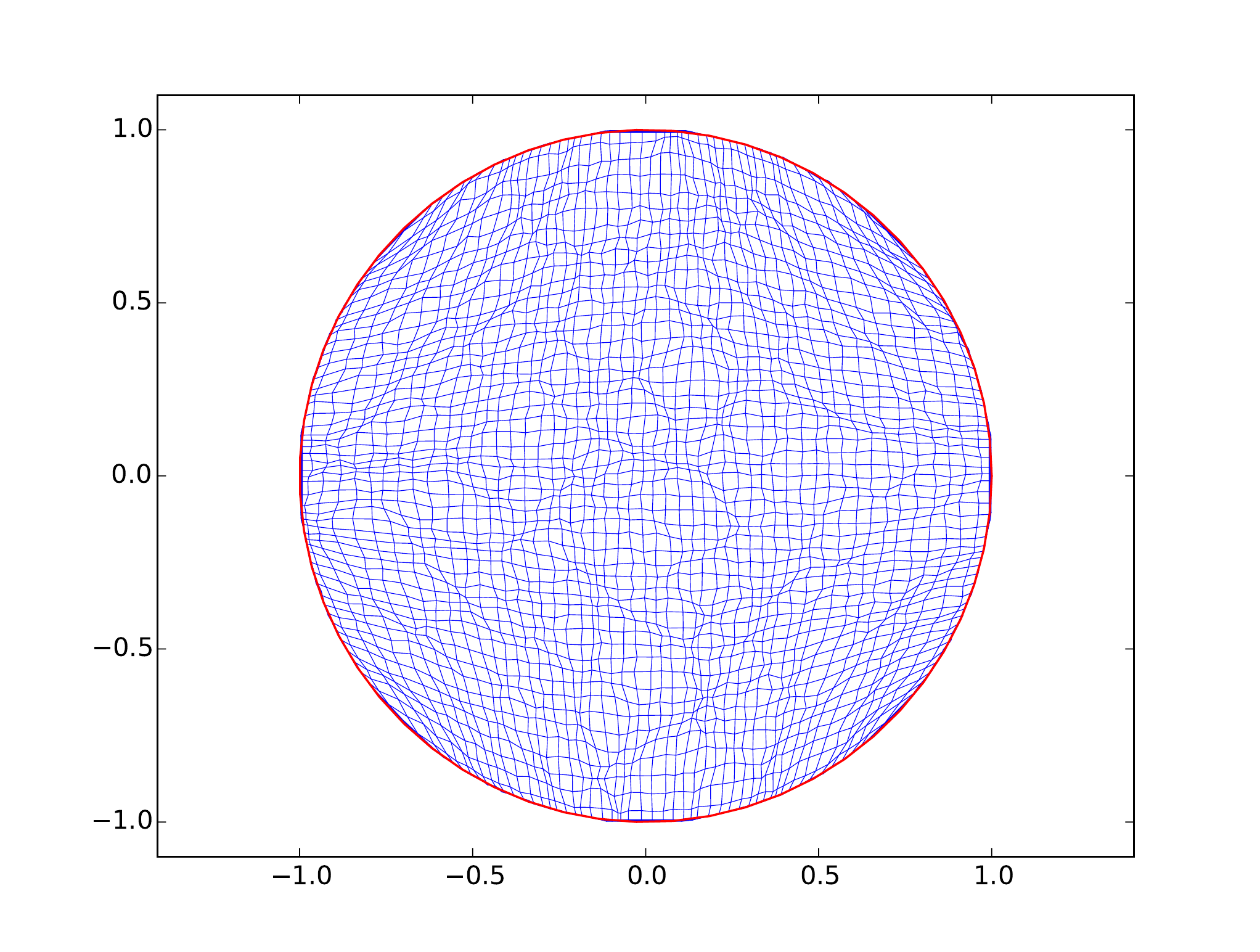}
\caption{Left: Source density $f$. Right: Corresponding optimal map deformation of the grid.}
\label{inhomo2}
\end{figure}

\subsection{ Experiments with inhomogeneous target densities $g$  }

 In this section we show the results obtained with  a constant 
 density square  $\SC$ (not represented in the figures) mapped to heterogeneous densities $g$ defined on a ball $\TG$, $N = 128$ and stencil width is $5$.  
 Figure~\ref{inhomotg1} shows the case of one Gaussian and a mixture 
 of Gaussians. The boundary of the target is the red circle.  The density $g$ and its gradient need to be defined in the numerical code 
 as functions which can be evaluated anywhere in the ambient space of $\TG$ as the map 
 can hit anywhere including  in $\R^2\setminus \TG$. If $g$ is given analytically (our case) this is easy 
 we just extend it by a constant out of $\TG$, otherwise one has to resort to interpolation.

 \begin{figure}[h!]
\centering
\includegraphics[width=0.45\linewidth]{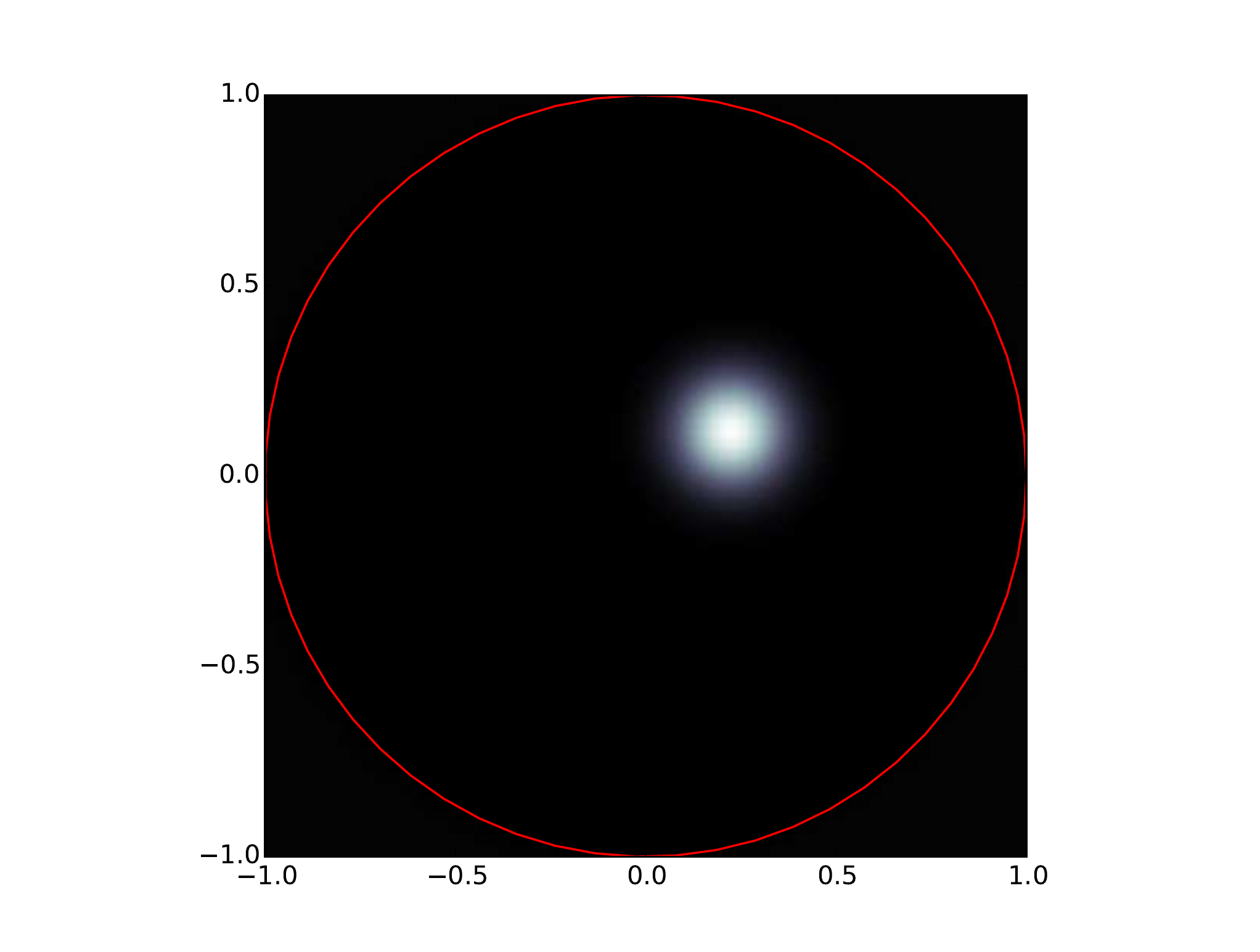}
\includegraphics[width=0.45\linewidth]{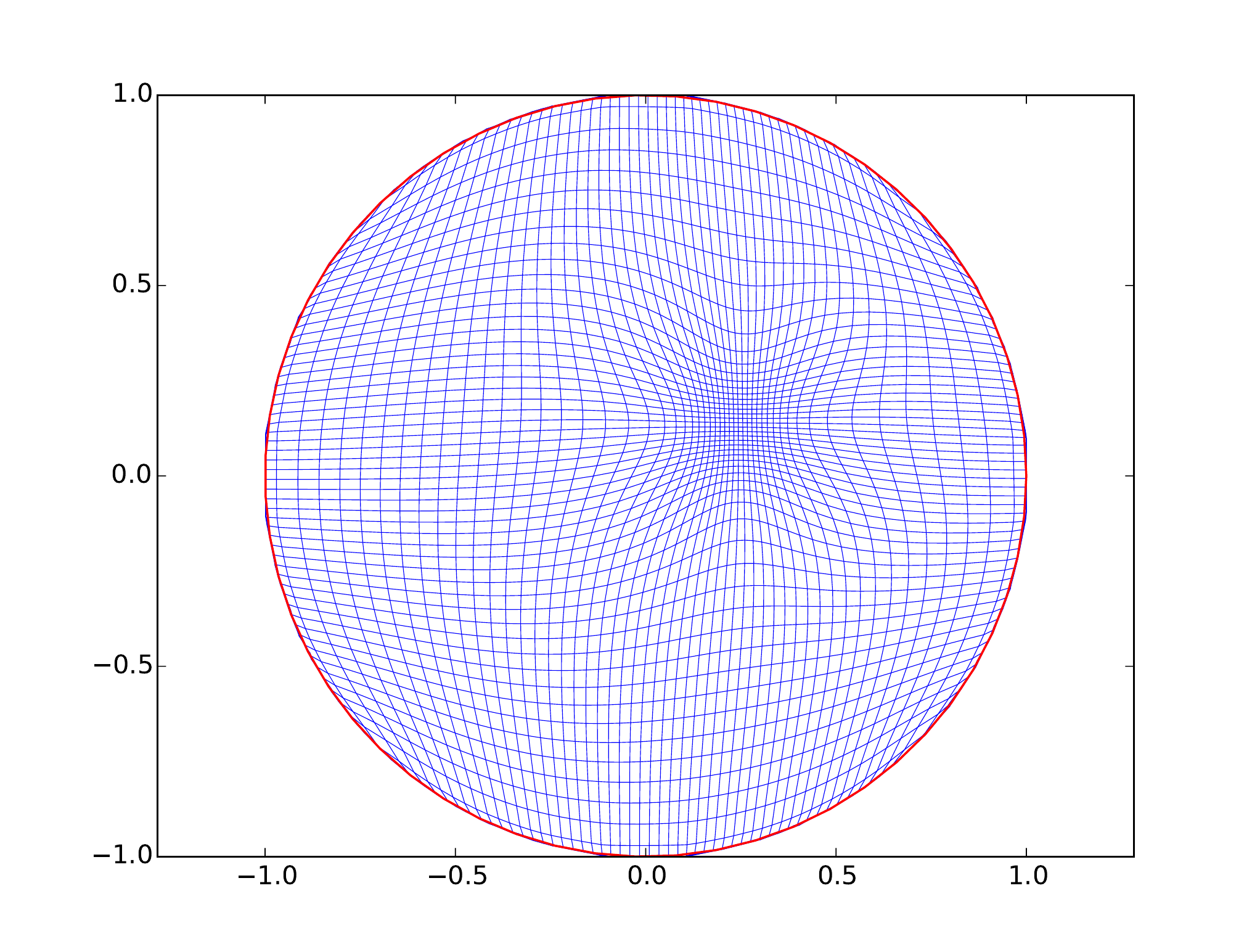}
\includegraphics[width=0.45\linewidth]{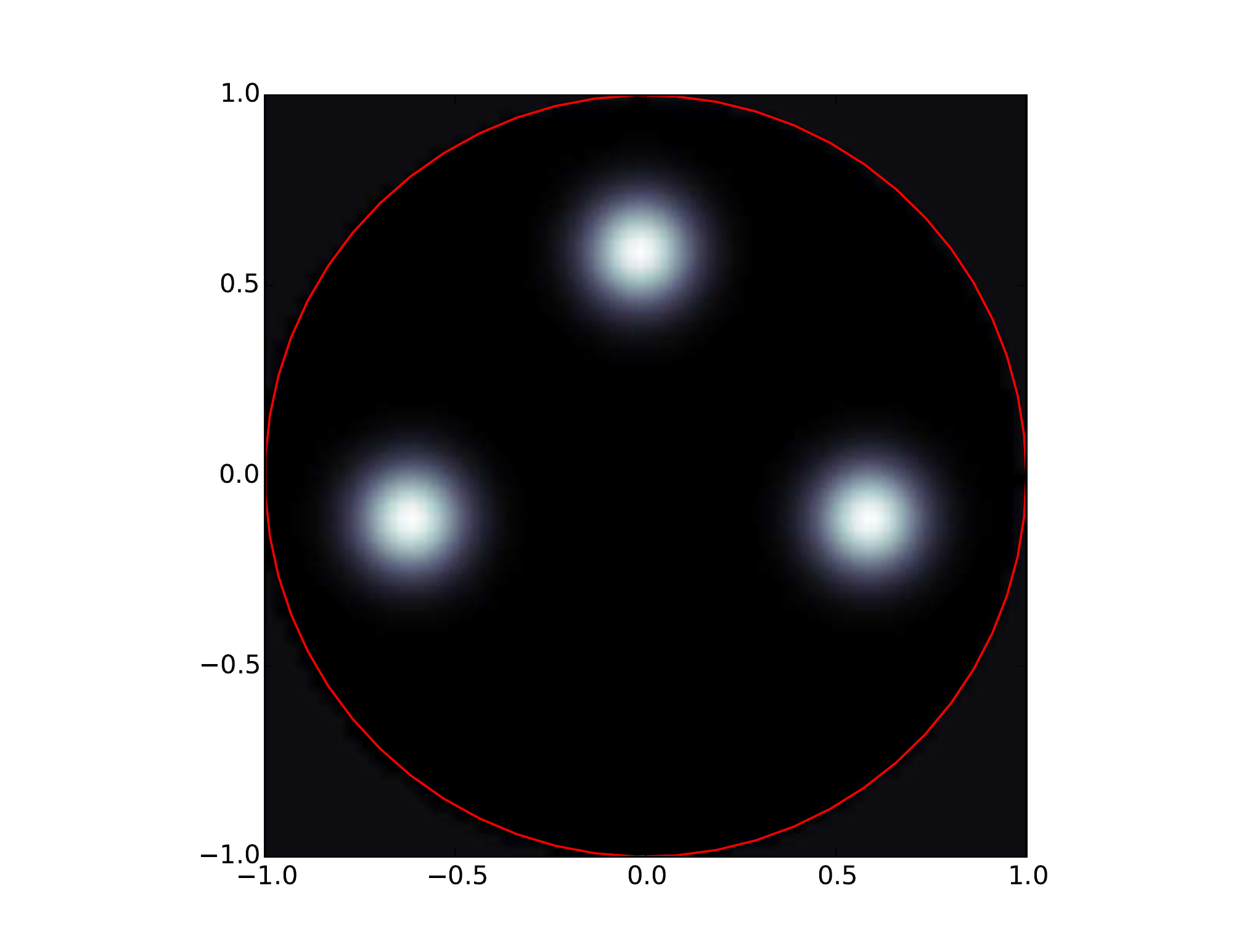}
\includegraphics[width=0.45\linewidth]{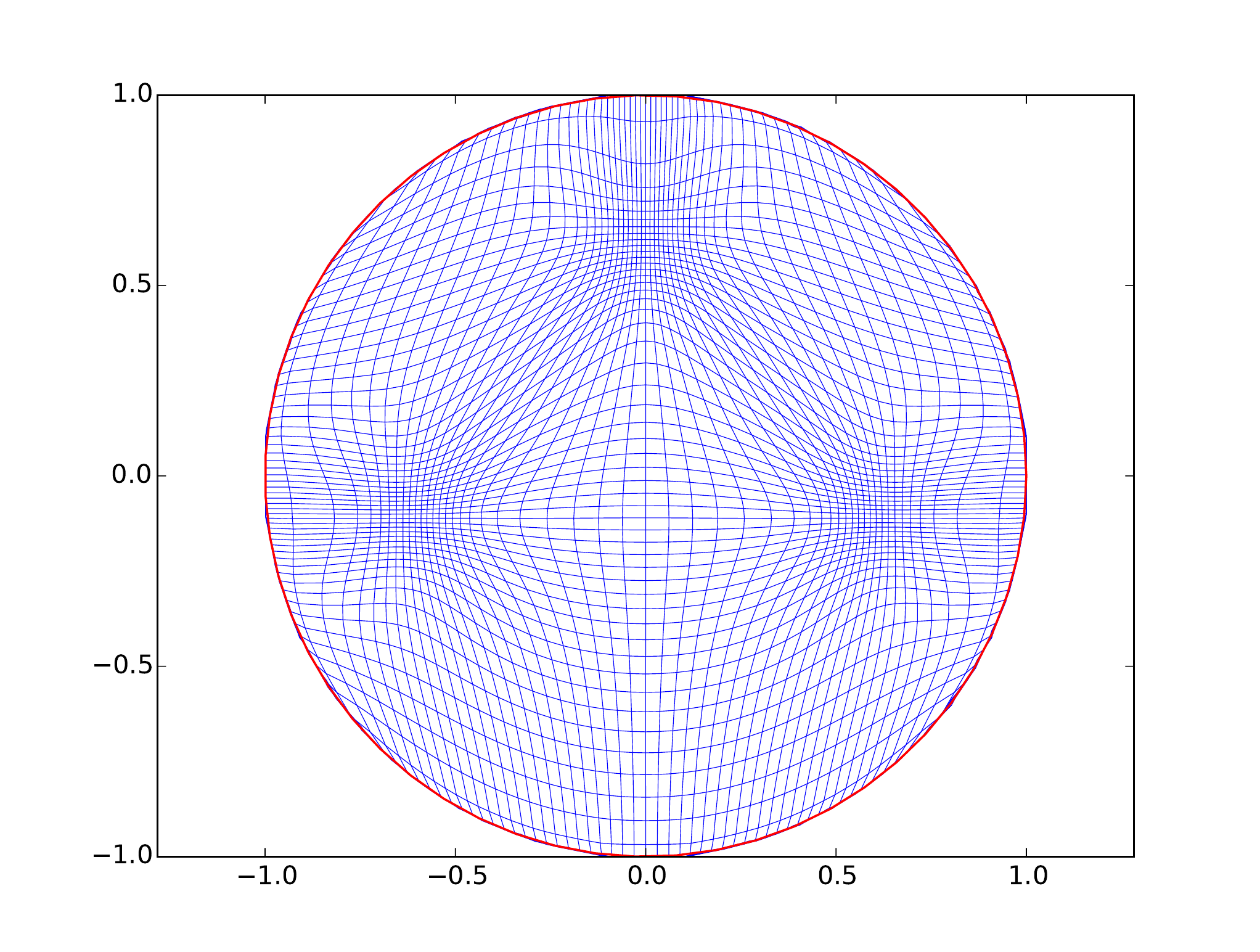}
\caption{Left: Target density $g$.  The boundary of the target is the red circle. Right: Corresponding optimal map deformation of the grid.}
\label{inhomotg1}
\end{figure}


\subsection{ Experiments with non convex and non connected sources} 
\label{sec:nonconvex}
In this section we map constant source densities $f$ with non convex and 
non connected supports to a constant density ball $\TG$.  The solution is 
still $C^1$ but is is know that the inverse mapping, i.e. the Legendre-Fenchel 
transform of the potential, has gradient singularities.  This has been analyzed 
by Figalli  in~\cite{figalli}.   Figures \ref{hole1} to \ref{hole5zoom} show different 
densities an the associated map deformations. The zero densities inclusions 
created singular structures in the target which correspond to gradient 
singularities of the dual map. These structures are consistent with those 
predicted by Figalli.  
We use $N = 128$ and the stencil width is $5$ . 

\begin{figure}[h]
\centering
\includegraphics[width=0.45\linewidth]{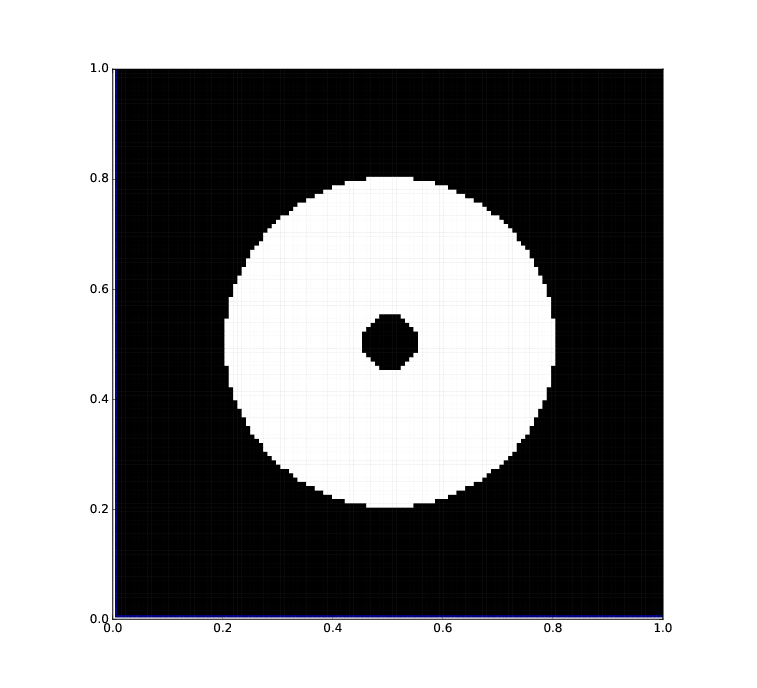} 
\includegraphics[width=0.45\linewidth]{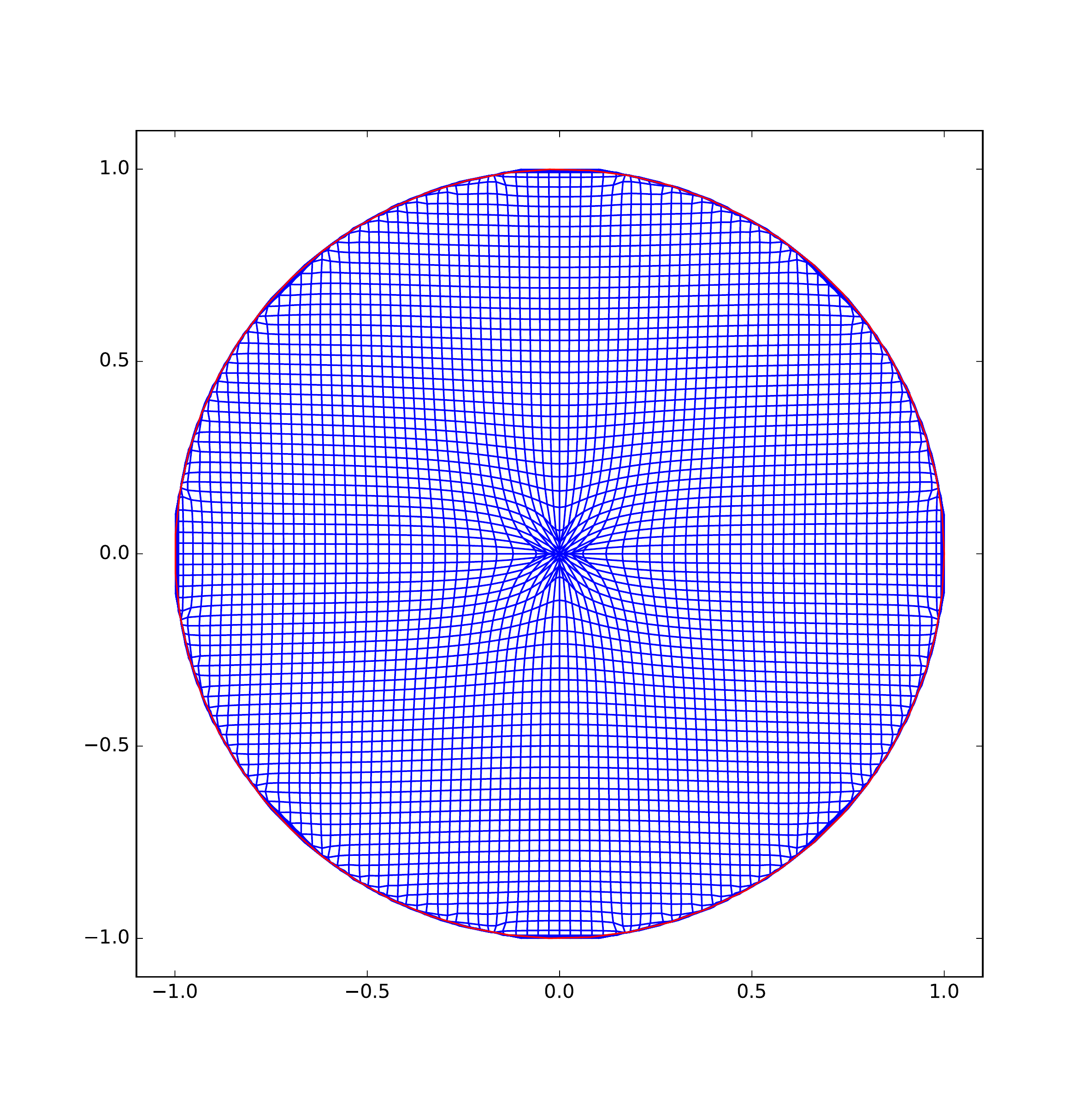}\\ 
\includegraphics[width=0.45\linewidth]{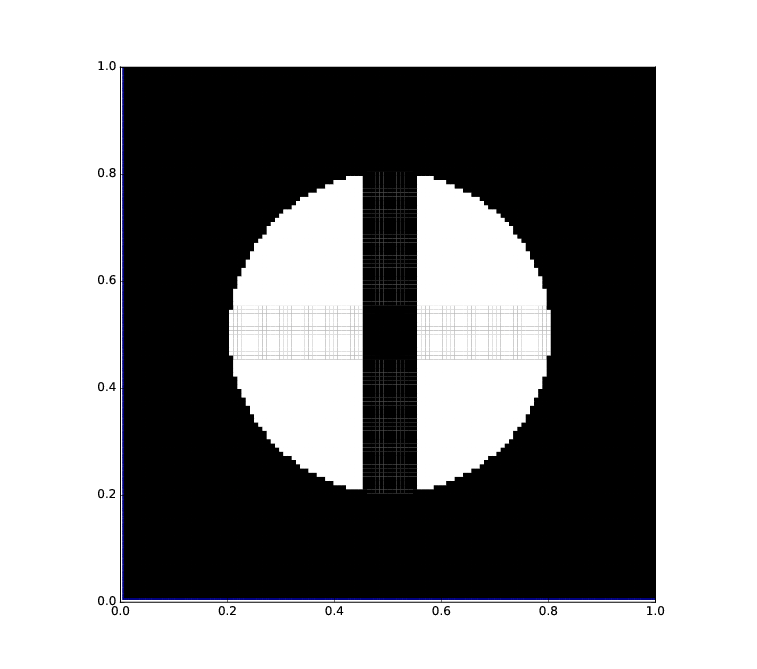} 
\includegraphics[width=0.45\linewidth]{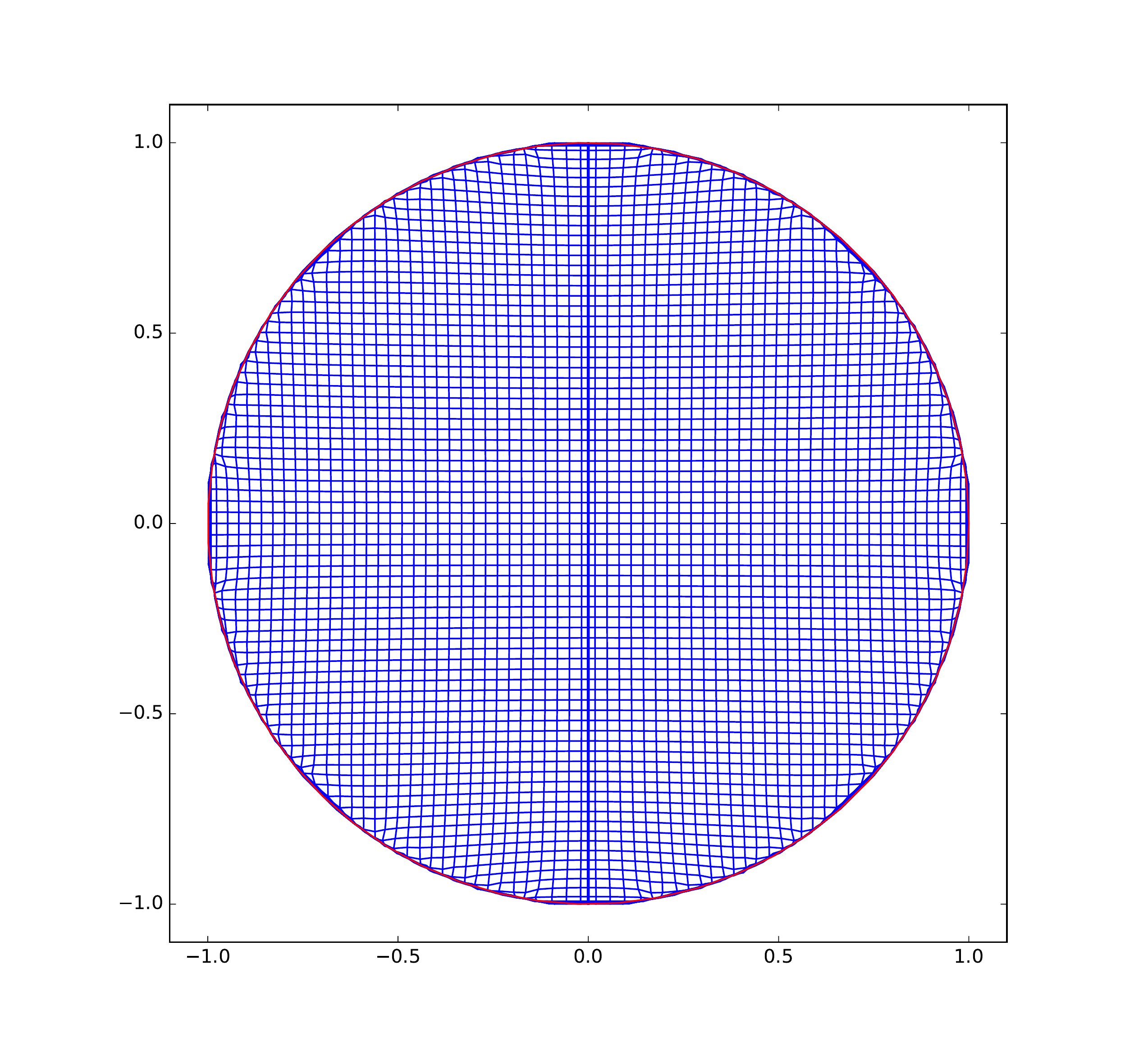} 
\caption{Left column: Source density $f$. Right column: Corresponding optimal map deformation of the grid.}
\label{hole1} 
\end{figure}

\begin{figure}[h]
\centering
\includegraphics[width=0.45\linewidth]{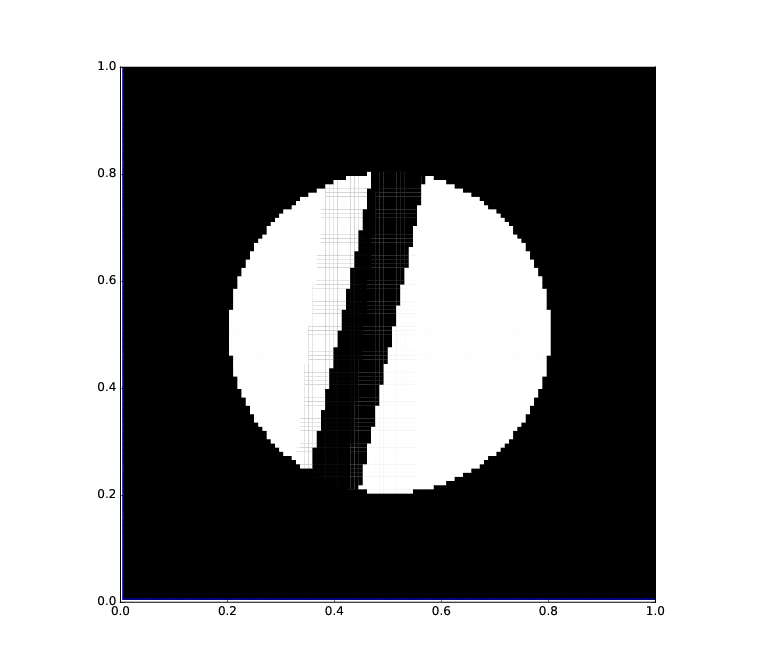} 
\includegraphics[width=0.45\linewidth]{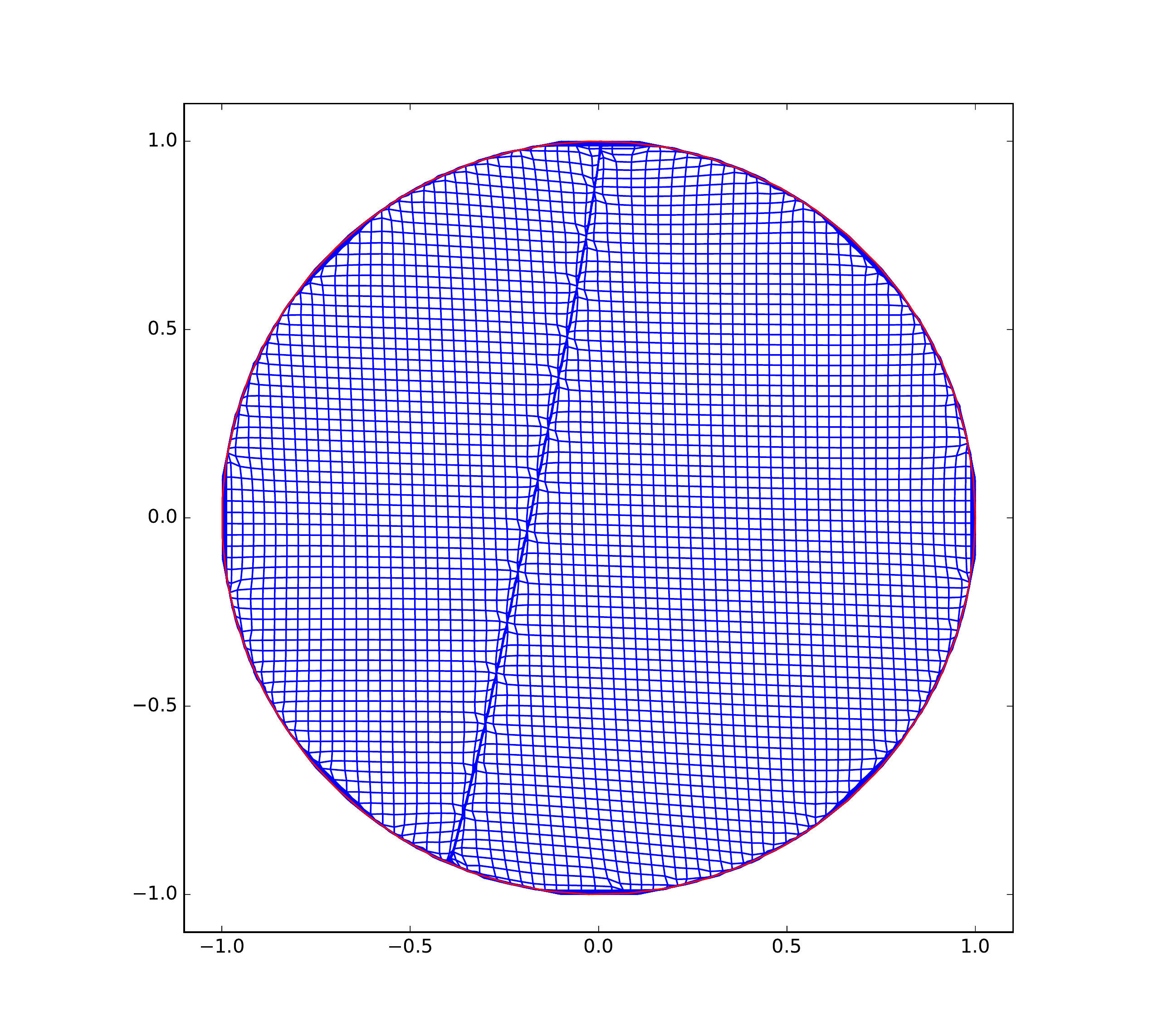}\\ 
\includegraphics[width=0.45\linewidth]{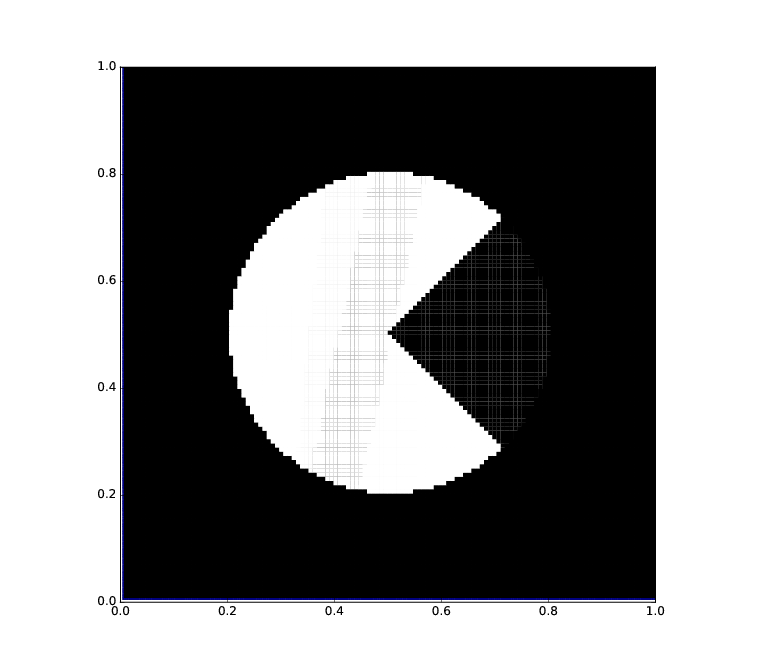} 
\includegraphics[width=0.45\linewidth]{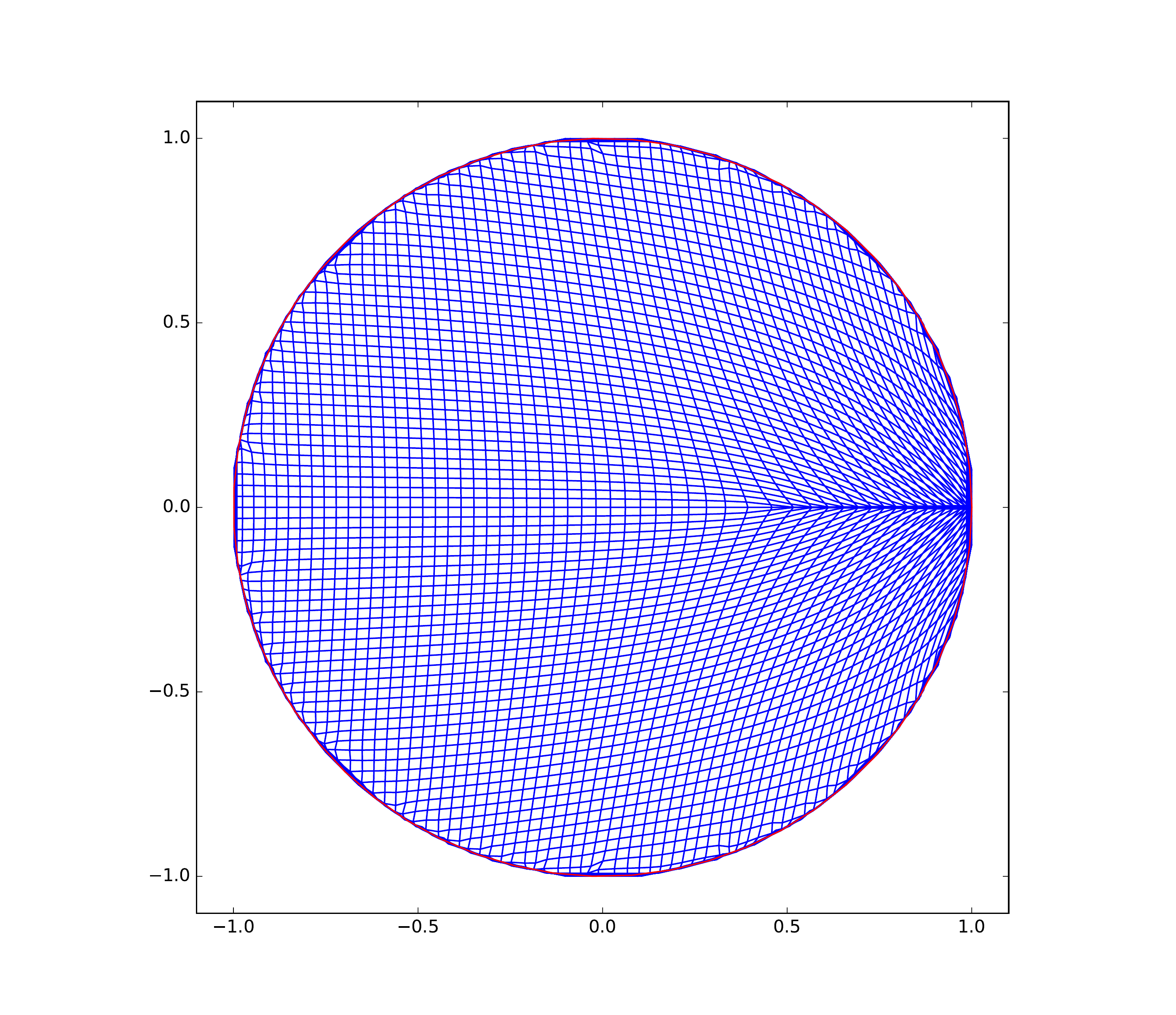} 
\caption{Left column: Source density $f$. Right column: Corresponding optimal map deformation of the grid.}
\label{hole3} 
\end{figure}

\begin{figure}[h]
\centering
\includegraphics[width=0.45\linewidth]{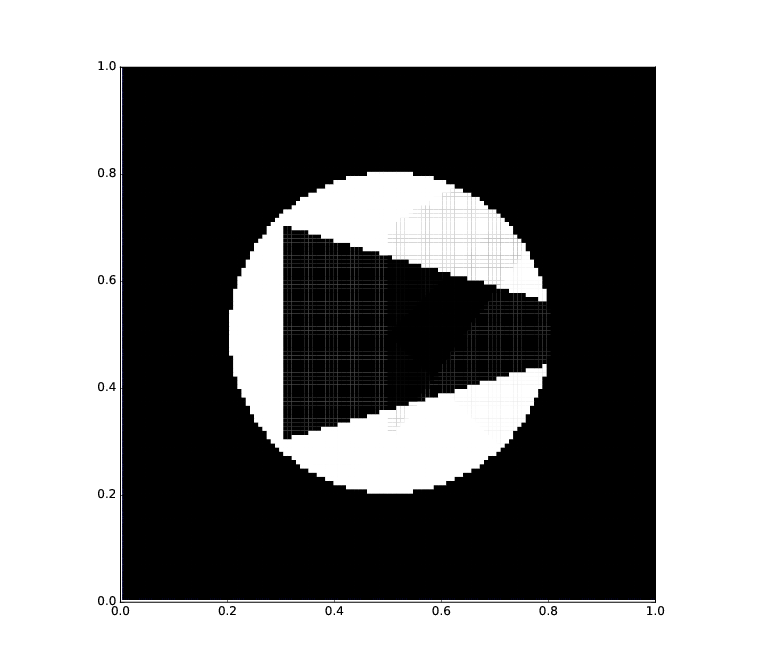} 
\includegraphics[width=0.45\linewidth]{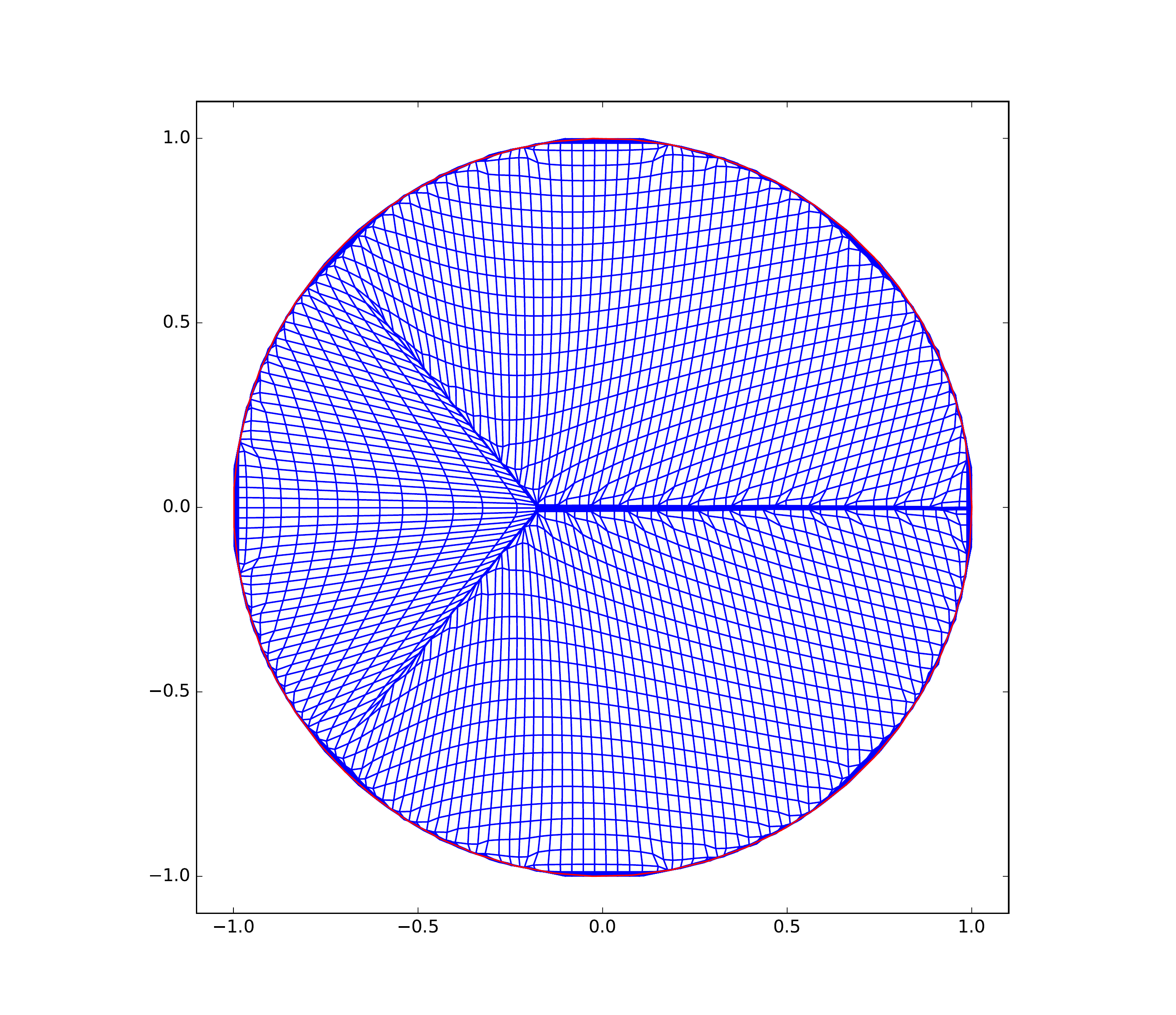} 
\caption{Left: Source density $f$. Right: Corresponding optimal map deformation of the grid.}
\label{hole5} 
\end{figure}

\begin{figure}[h]
\centering
\includegraphics[width=0.45\linewidth]{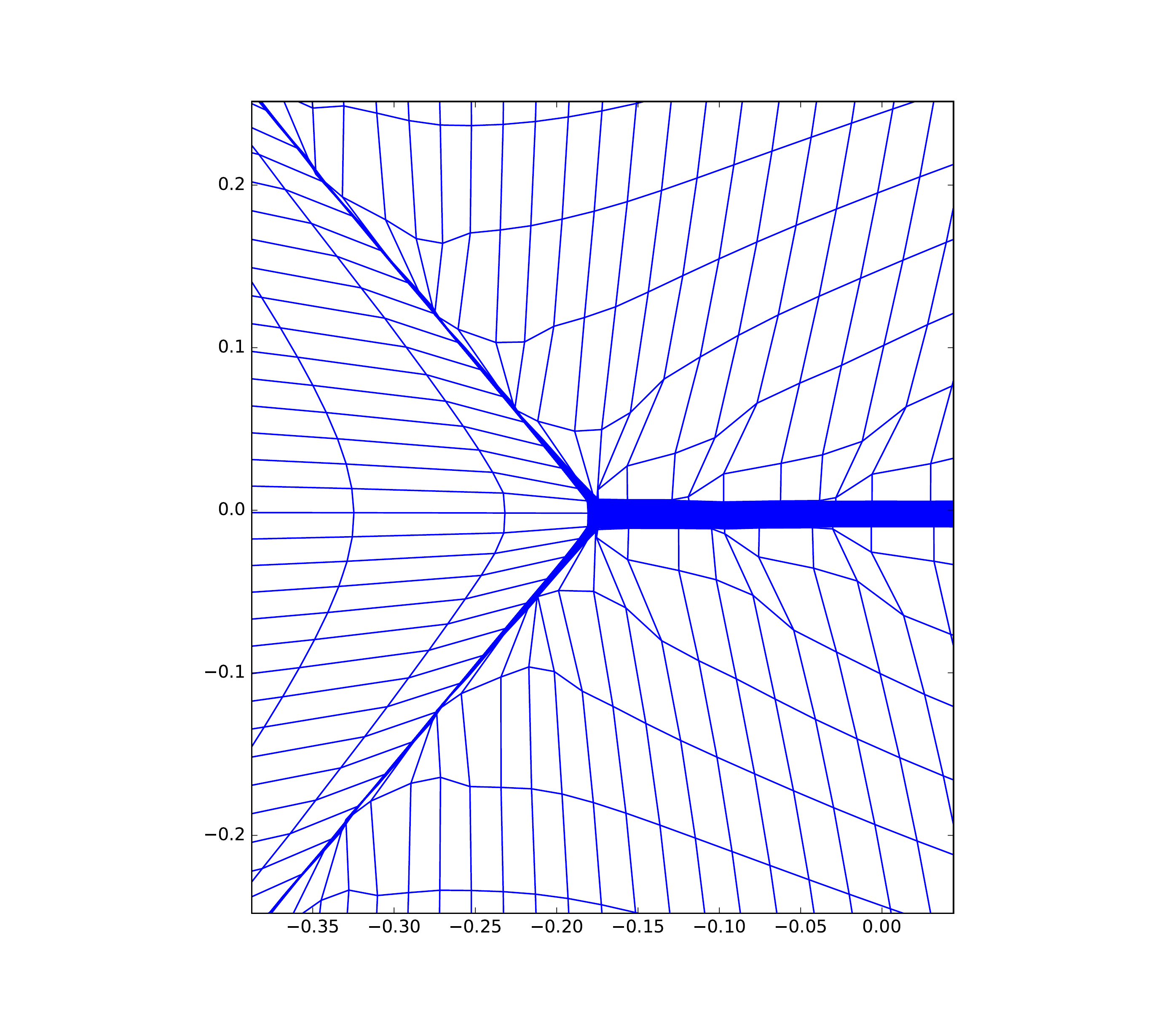} 
\caption{Zoom of Figure~\ref{hole5}, Right.}
\label{hole5zoom} 
\end{figure}

\revision{Eventually, in Figure~\ref{mapS}, we display an experiment which shows the $\ell^\infty$-norm of $e$ for the optimal superbase chosen by $\MA$ or the optimal vector chosen by $\tMAo$. This illustrates the fact that the MA-LBR operator uses very compact stencils (in that case the superbase $\{(1,1),(0,1),(1,0)\}$ or its rotation is used everywhere in $\SC$), whereas the extension outside $\SC$ uses more elongated vectors depending on where the points are mapped in $\oTG$. 
  
\begin{figure}[h]
\centering
\includegraphics[width=0.45\linewidth]{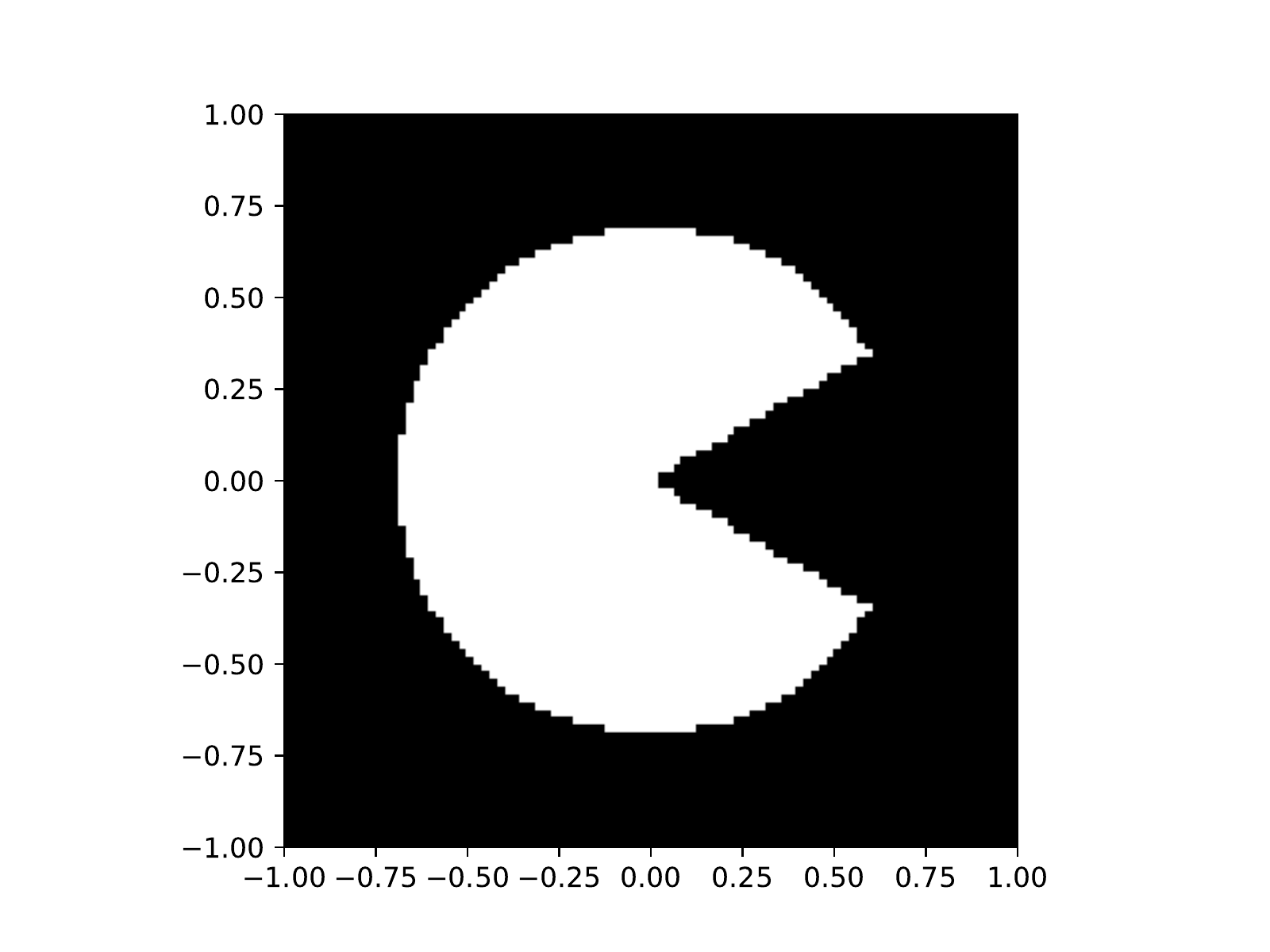} 
\includegraphics[width=0.45\linewidth]{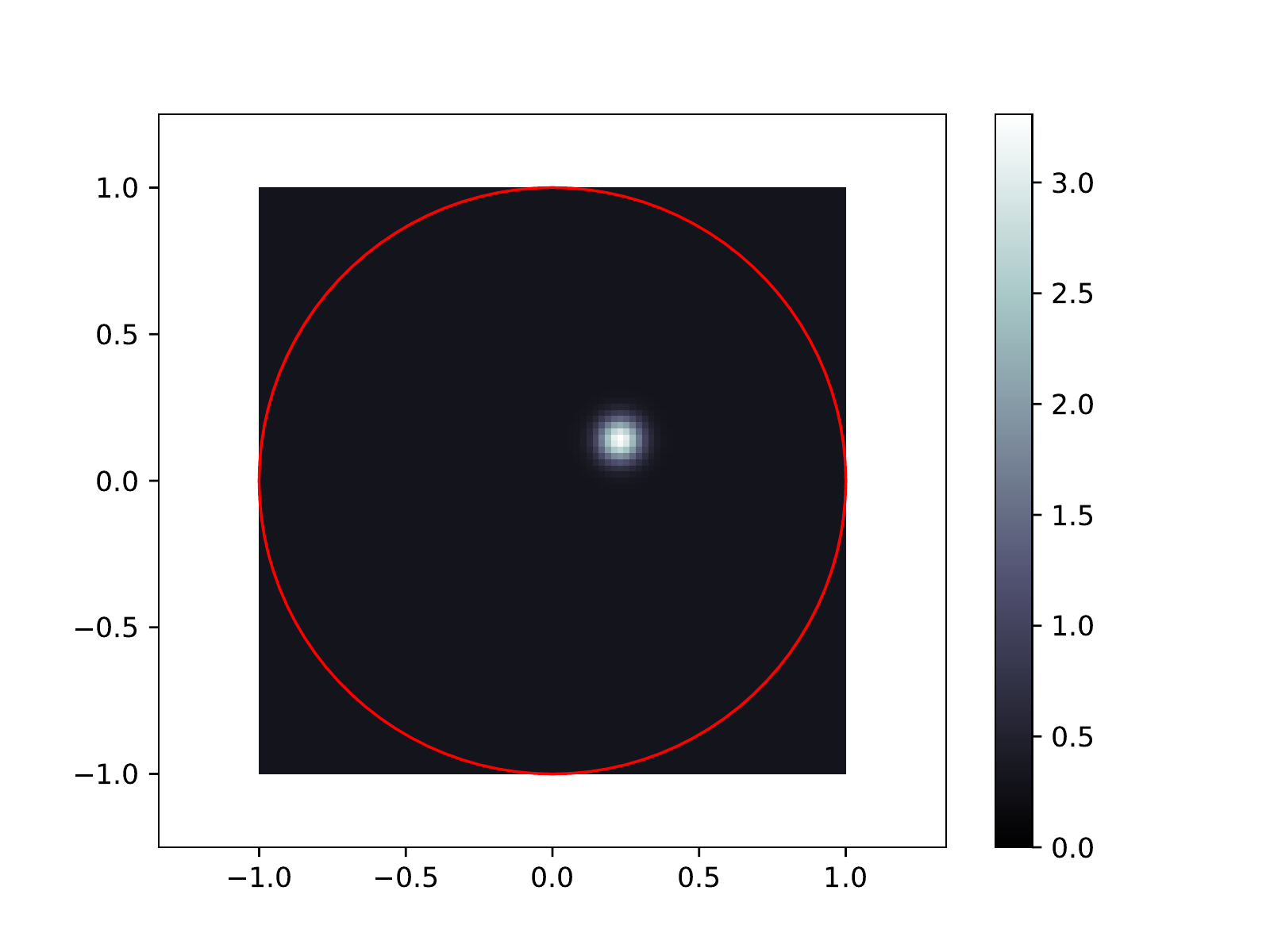} 
\includegraphics[width=0.47\linewidth,clip,trim=0.6cm 0 0 0]{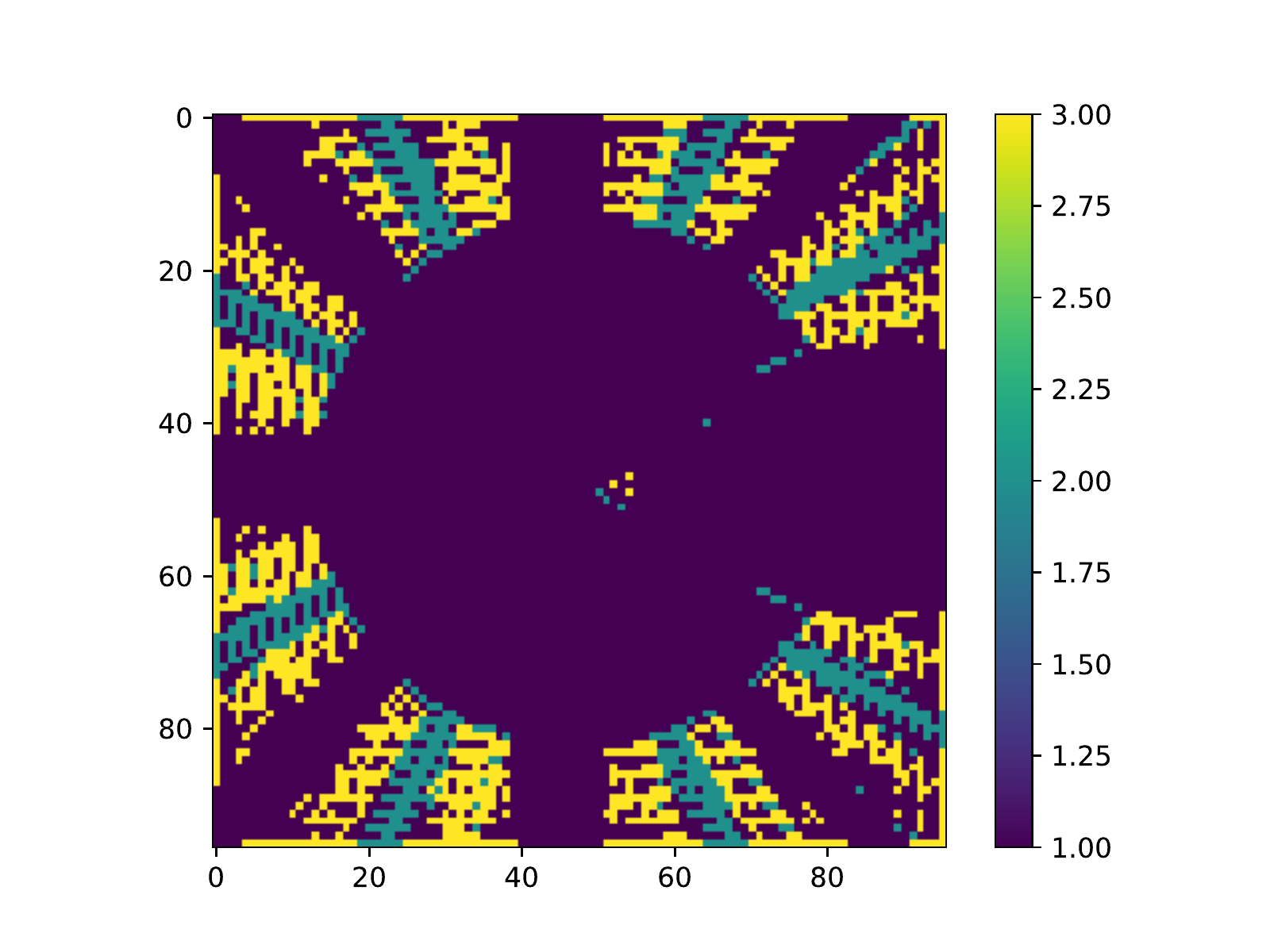} 
\includegraphics[width=0.36\linewidth]{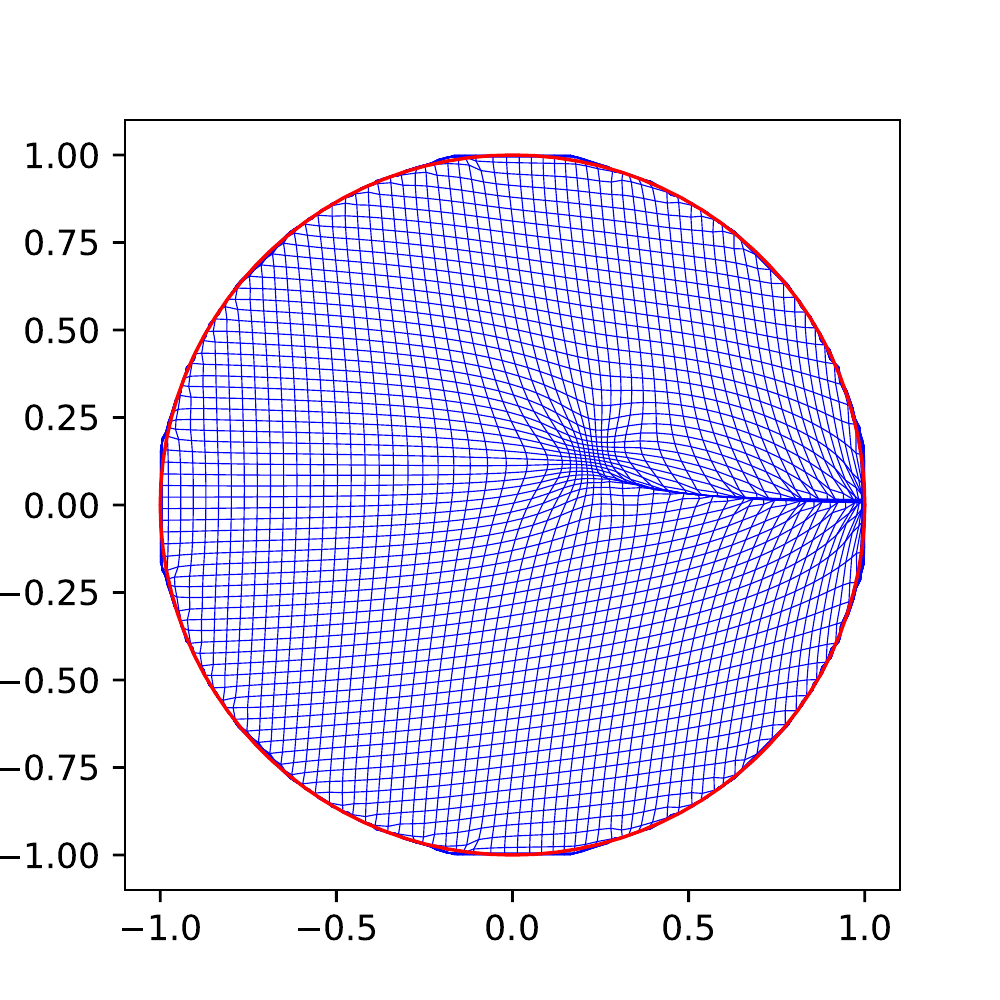} 
\caption{From left to right, top to bottom: Source density, target density and set $\oTG$, $\|e\|_{\infty}$ for the $e$ vector of the optimal superbase, optimal map deformation.}
\label{mapS} 
\end{figure}

}


\section{Conclusion} 
In this work, we have proposed an novel way to impose the BV2 boundary condition for the Monge-Amp\`ere equation in schemes such as MA-LBR.
 The idea consists in slightly extending the domain of the solution so as to capture the behavior of ``minimal'' Brenier solutions to the optimal transport problem. Our proof of convergence as the grid stepsize goes to zero does not appeal to the theory of viscosity solutions, but rather to simple arguments combined with standard optimal transport results.
The numerical experiments faithfully reproduce the typical behavior of optimal transport solutions. 

 Although numerically we have not encountered any particular difficulty with the resolution of the scheme, the existence of a solution to the discrete problem remains an open problem that we leave for future work.
 \nopagebreak
 \section*{Acknowledgments}
 The authors warmly thank Jean-Marie Mirebeau for suggesting the proof of Lemma~\ref{lem:subgradh2}.

\appendix
\section{Convergence of the finite difference schemes}\label{sec:cvfindiff}
The aim of this section is to prove that~\eqref{eq:cvgrad} holds for the centered, forward and backward finite difference schemes, respectively $\gradnc$ (see~\eqref{d1hC}), $\gradnf$ and $\gradnb$ (see~\eqref{d1hFB}).

We have defined in Section~\ref{sec:discretize}, the values $\left(\udisc[x]\right)_{x\in \domD\cap \gridn}$, and in Section~\ref{sec:convergence} their ``interpolation'' $\utc$ using~\eqref{eq:defutc} as a convex function. In Lemma~\ref{lem:vC1}, we have shown that this interpolation converges (along some sequence) towards some function $v\in\Cder^{1}(\RR^2)$. We wish to prove that the discrete gradients $\gradnc \udisc$, $\gradnf\udisc$ and $\gradnb\udisc$ converge in some sense to the gradient $\nabla v$.

\subsection{Convergence of the gradient of convex functions}
We begin with some general results on convex functions. The following lemma is standard, and mainly follows from the results in~\cite[Section VI.6.2]{hiriart-lemarechal-1996}.
However, we provide a proof below, as this result is not explicitly stated there.

\begin{lemma}{\label{lem:cvugradcvx}}
Let $\Omega\subset \RR^2$ be a nonempty open convex set, $C>0$, and $\{v_n\}_{n\in\NN}$ be a sequence of finite-valued convex functions on $\Omega$ which converges pointwise on $\Omega$ towards some function $v:\Omega\rightarrow \RR$ as $n\to+\infty$.

  If $v\in \Cder^{1}(\Omega)$, then for all compact set $K\subset \Omega$, 
  \begin{align}
    \lim_{n\to+\infty}\sup_{x\in K}\sup_{\abs{x'-x}\leq C\stsizen}\sup_{y\in \partial v_n(x')}\abs{\nabla v(x)-y}.
  \end{align}

\end{lemma}

\begin{proof}
  Let us recall that the convexity of $v_n$ implies that $v$ is convex, and $v_n$ converges uniformly on the compact sets of $\Omega$ towards $v$. Let $\rho>0$ such that $\rho\leq \frac{1}{4}\dist(K,\RR^2\setminus \Omega$), and let $K'\eqdef\enscond{x'\in \RR^2}{\dist(x',K)\leq \rho}\subset \Omega$.

  Observe that $\partial v_n(K')$ is bounded independently of $n$. Indeed, for any $x'\in K'$, any $y_n\in \partial v_n(x')$, if $y\neq 0$, 
  \begin{align}
    v_n\left(x'+\rho\frac{y_n}{\abs{y_n}}\right) &\geq v_n(x')+\rho\abs{y_n},\\
    \mbox{hence } \abs{y_n}&\leq \frac{1}{\rho}\left(v_n\left(x'+\rho\frac{y_n}{\abs{y_n}}\right)-v_n(x')\right),
  \end{align}
  and the right-hand side is uniformly bounded by uniform convergence of $v_n$ on $\enscond{x''\in\Omega}{\dist(x,K)\leq 2\rho}$.

  Now, assume by contradiction that there is some $\varepsilon>0$,  some (not relabeled) subsequence $x_{n}\in K$, $x'_{n}\in B(x_{n},C\stsize_n)$, and $y_{n}\in \partial v_{n}(x'_{n})$ with $\abs{y_{n}-\nabla v(x_{n})}\geq \varepsilon$.

  By compactness, there exists $y\in \RR^2$, $x\in K$, such that (up to an additional extraction) $y_{n}\to y$ and $x_{n}\to x$. Passing to the limit in the subgradient inequality
  \begin{align}
    \forall x''\in \Omega,\quad    v_{n}(x'')&\geq v_{n}(x_{n}')+\dotp{y_{n}}{x''-x'_{n}},\\
  \mbox{we obtain}\quad\quad\quad \forall x''\in \Omega,\quad   v(x'')&\geq v(x)+\dotp{y}{x''-x}.
\end{align}
Hence $y=\nabla v(x)$ which yields a contradiction with $\abs{y_{n}-\nabla v(x_{n})}\geq \varepsilon$ for $n$ large enough. This yields the claimed result.
\end{proof}

Now, we consider the finite difference scheme applied to a convex function.

\begin{lemma}\label{lem:cvfindisccvx}
  Let $\Omega\subset \RR^2$ be a nonempty open convex set, and $\{v_n\}_{n\in\NN}$ be a family of finite-valued convex functions on $\Omega$ which converges pointwise on $\Omega$ towards some function \revision{$v\in \Cder^{1}(\Omega)$}. 
Define $V_n[x]\eqdef v_{n}(x) $ for all $x\in \Omega\cap \gridn$.
Then, for all compact set $K\subset \Omega$,
\begin{align}
  \lim_{n\to+\infty}\sup_{x\in K\cap \gridn} \abs{\gradnc V_n[x]- \nabla v(x)}
\end{align}
 where the centered finite difference operator $\gradnc$ is defined in~\eqref{d1hC}.
\end{lemma}

\begin{proof}
  By~\cite[Theorem VI.2.3.4]{hiriart-lemarechal-1996}, we see that 
\begin{align}
  \delta^{\stsizen}_{(1,0)} V_n [x] =V_{n}[x+\stsizen(1,0)]-V_{n}[x] = \stsizen\int_0^1 \max_{y\in \partial v_n(x+s\stsizen(1,0))}\dotp{y}{(1,0)}\d s,
\end{align}
so that by Lemma~\ref{lem:cvugradcvx}, the above quantity converges uniformly on $K$ towards $\frac{\partial v}{\partial x_1}(x)$.

Similarly, the quantities $\delta^{\stsizen}_{(-1,0)}V_n [x]$, $\delta^{\stsizen}_{(0,1)}V_n [x]$ and  $\delta^{\stsizen}_{(1,0)}V_n [x]$ respectively converge to $-\frac{\partial v}{\partial x_1}(x)$, $\frac{\partial v}{\partial x_2}(x)$ and $-\frac{\partial v}{\partial x_2}(x)$ uniformly.
\end{proof}

\subsection{Finite difference schemes for the subsampled sequences}
Now, let us turn to the values $\left(\udisc[x]\right)_{x\in \domD\cap \gridn}$ defined in Section~\ref{sec:discretize} and the function  $\utc$ defined in~\eqref{eq:defutc}. In the following, $v$ denotes the function constructed in Lemma~\ref{lem:vC1}, as the limit of $(\utc)_{n\in\NN}$. Unfortunately, we cannot directly apply Lemma~\ref{lem:cvfindisccvx} to $\utc$ as it is not really an interpolation of $\left(\udisc[x]\right)_{x\in \domD\cap \gridn}$: there could be some points $x\in \domD\cap \gridn$ such that $\utc(x)<\udisc[x]$ hence $\gradnc\udisc$ is not a priori the centered finite difference scheme applied to $\utc$.

However, the first two properties of Proposition~\ref{prop:udisc} express the fact that $\udisc$ is directionally convex, in the sense of~\cite[Appendix A]{Mirebeau2016}. In particular, defining the subsampled grids
\begin{align*}
  \gridUn&\eqdef \stsize(2\ZZ)\times (2\ZZ),\quad \gridDn\eqdef \stsize(2\ZZ+1)\times (2\ZZ),\\
  \quad \gridTn&\eqdef  \stsize(2\ZZ)\times (2\ZZ+1), \quad \gridQn=\stsize(2\ZZ+1)\times(2\ZZ+1),
\end{align*}
there exist convex lower semi-continuous functions $\vUn$, $\vDn$, $\vTn$, $\vQn : \RR^2\rightarrow \RR\cup\{+\infty\}$ such that 
\begin{align}
  \forall x\in \gridUn\cap \domD,\quad  \vUn(x)=\udisc[x],
\end{align}
and similarly, replacing $(I)$ with $(II)$, $(III)$, $(IV)$.

The following Lemma shows that those convex functions are not too far from our interpolation $\utc$, at least far from the boundary of $\domD$.
We define $\domDn\eqdef \enscond{x\in \domD}{x+2\stsizen\setdir\subset \domD}$.

\begin{proposition}\label{prop:boundcvx}
The following inequalities hold.
\begin{align}
  \forall x\in \gridUn\cap \domD,\quad   \vUn(x)&\geq \utc(x),\label{eq:boundcvx1}\\
  \forall x\in \gridUn\cap \domDn,\quad \utc(x)&\geq \vUn(x)-C\stsizen,\label{eq:boundcvx2}
\end{align}
where $C=4\max\enscond{\sigtg(e)}{e\in \{(\pm1,0),(0,\pm1)\revision{\}}}$.

The corresponding inequalities also hold for $\vDn$, $\vTn$, $\vQn$.
\end{proposition}
\begin{proof}
The first inequality readily follows from the construction of $\utc$,
\begin{align*}
 \forall x\in \gridUn\cap \domD, \quad  \utc(x)\leq \udisc[x]=\vUn(x).
\end{align*}

Now we deal with \eqref{eq:boundcvx2}. Let $x\in \gridUn\cap \domDn$. We first prove that $\partial \vUn(x)\subset \oTGn$. For any $y\in \partial \vUn(x)$, all $e\in \setdir$, and using Proposition~\ref{prop:udisc},
\begin{align*}
  \dotp{y}{2\est}\leq \vUn(x+2\est) -\vUn(x)=\udisc[x+2\est]-\udisc[x]\leq\sigtg(2\est)
\end{align*}
Hence $\dotp{y}{e}\leq \sigtg(e)$ for all $e\in\setdir$, that is $y\in \oTGn$.
Moreover, by the subgradient inequality,
\begin{align}
  \forall x'\in \gridUn\cap \domD,\quad \udisc[x']=\vUn(x')\geq \vUn(x)+\dotp{y}{x'-x}. \label{eq:cvxsubgrad}
\end{align}
For all $x''\in \gridn\cap \domD$, there exists $x'\in \gridUn\cap \domD$ such that $\normi{x''-x'}\leq \stsizen$. Let $e\eqdef \frac{1}{\stsizen}({x''-x'})$ satisfying $\normi{e}\leq 1$ (hence $\abs{\dotp{y}{e}}\leq 2\max\enscond{\sigtg(e)}{e\in \{(\pm1,0),(0,\pm1)}$).
Using $ \udisc[x']\leq  \udisc(x'')+\sigtg(-\est)$ in~\eqref{eq:cvxsubgrad} we get
\begin{align*}
  \udisc(x'')+\sigtg(-\est)&\geq  \vUn(x)+\dotp{y}{x''-x} -\dotp{y}{\est},\\
  \mbox{hence}\quad \udisc(x'')&\geq \dotp{y}{x''-x}+\vUn(x)-C\stsize.
\end{align*}
As a result, the affine mapping $L:x''\mapsto \left(\dotp{y}{x''-x}+\vUn(x)-C\stsize\right)$ minorizes $\udisc[x'']$ at each $x''\in  \gridn\cap \domD$, and $\nabla L=y\in \oTGn$. We deduce that $\utc\geq L$. In particular, $\utc(x)\geq L(x)=\vUn(x)-C\stsizen$.
\end{proof}

The next proposition shows that, along sequences, $\utc$ and $\vUn$ share the same limit as $\stsize\to 0^+$.
\begin{proposition}\label{prop:cvcvx}
  Let $K\subset \inte(\domD)$ be a compact set. 
  
  There exists $\vlimo\in \Cder(K)$ such that for $n$ large enough $\vUn\in \Cder^{}(K)$ and up to a subsequence, \begin{align}
    \lim_{n\to+\infty} \norm{\vUn-\vlimo}_{L^\infty(K)}=\lim_{n\to+\infty} \norm{\vDn-\vlimo}_{L^\infty(K)}&=0,\\
    \lim_{n\to+\infty} \norm{\vTn-\vlimo}_{L^\infty(K)}=\lim_{n\to+\infty} \norm{\vQn-\vlimo}_{L^\infty(K)}&=0,\\
\lim_{n\to+\infty} \norm{\utcn-\vlimo}_{L^\infty(K)}&=0.
  \end{align}
\end{proposition}
\begin{proof}
  The first point follows from Proposition~\ref{prop:udisc} which ensures that $(\udisc[x])_{n\in\NN}$ is bounded uniformly in $x\in\gridn\cap \domD$. As a result, the convex function $\vUn$ is uniformly bounded (hence uniformly Lipschitz) in a neighborhood of $K$ and Ascoli-Arzel\`a's theorem ensures the convergence of $\vUn$ towards some $\vlimo$ up to a subsequence. 

  By Proposition~\ref{prop:boundcvx}, we see that $\utcn$ must converge (pointwise) towards $\vlimo$ on a dense subset of $K$. Since $K$ is compact and the functions are uniformly Lipschitz, the convergence is uniform on $K$.

  Since Proposition~\ref{prop:boundcvx} also holds for $\vDn$, $\vTn$, $\vQn$, we deduce that those convex functions also converge uniformly (along the same subsequence) towards $\vlimo$.
\end{proof}

In Lemma~\ref{lem:subgradh}, we show that $\{\utc\}_{n\in\NN}\subset \Cder{}(\RR^2)$ is precompact for the topology of uniform convergence on compact subsets of $\RR^2$. As a consequence of Proposition~\ref{prop:boundcvx}, if $\utcn$ converges along some subsequence towards some function $\vlim$, then $\vUn$, \ldots $\vQn$ converge uniformly towards $\vlim$ on compact subsets of $\inte(\domD)$, along the same subsequence.
Now, in Lemma~\ref{lem:vC1}, we show that $\vlim\in \Cder^{1}(\RR^2)$. We may now state the main result of this section.

\begin{proposition}
  Assume that (up to a subsequence) $\utcn$ converges uniformly on the compacts subsets of $\RR^2$ towards some function $\vlim\in \Cder^{1}(\RR^2)$. Then, for all $K\subset \inte(\domD)$ compact, 
\begin{align}
  \lim_{n\to+\infty}\max_{x\in K\cap \gridn}\abs{\gradnc\udisc[x]-\nabla v(x)}=0,\label{eq:unifgradc}\\
  \lim_{n\to+\infty}\max_{x\in K\cap \gridn}\abs{\gradnf\udisc[x]-\nabla v(x)}=0,\label{eq:unifgradf}\\
  \lim_{n\to+\infty}\max_{x\in K\cap \gridn}\abs{\gradnb\udisc[x]-\nabla v(x)}=0.\label{eq:unifgradb}
\end{align}
\end{proposition}
\begin{proof}
  We note that from Proposition~\ref{prop:cvcvx}, $\vUn$, \ldots $\vQn$ also converge uniformly on compact subsets of $\inte(\domD)$ towards $\vlim$.

  We begin with the forward and backward differences defined in~\eqref{d1hFB}. Let $e=(1,0)$ or $e=(0,1)$. By monotonicity of the slopes (see Proposition~\ref{prop:udisc}),
  \begin{align}
    \frac{\udisc[x+2\est]-\udisc[x]}{2\stsizen} \geq  \frac{\udisc[x+\est]-\udisc[x]}{\stsizen}\geq  \frac{\udisc[x]-\udisc[x-\est]}{\stsizen}\geq  \frac{\udisc[x]-\udisc[x-2\est]}{2\stsizen}\label{eq:sandwichslope}
  \end{align}
  
  To fix ideas, let us assume that $x\in \gridUn$. Then $x\pm 2\est\in \gridUn$ and applying Lemma~\eqref{lem:cvfindisccvx} to $\vUn$ (in some compact set $K'$, $K\subset K'\subset \inte(\domD)$), we see that the left and right-hand sides of \eqref{eq:sandwichslope} converge towards $\dotp{\nabla\vlim(x)}{e}$, uniformly in $x\in \gridUn\cap K$. The same argument holds for $x$ in $\gridDn$, $\gridTn$ or $\gridQn$, so that we get~\eqref{eq:unifgradf} and \eqref{eq:unifgradb}.

  By linear combination, we deduce~\eqref{eq:unifgradc}.
\end{proof}

\section{MA-LBR overestimates the subgradient} 
\label{sec:apxmasub}

The following proof of Lemma~\ref{lem:subgradh2} was suggested to us by J.-M. Mirebeau.
We denote by $\partial_x u$ the subgradient of a convex function $u$ at a point $x$.

\begin{lemma}
Let $\Omega$ be a convex neighborhood of a point $x \in \RR^d$, and let  $u,v : \Omega \to \R^d$ be convex. If $u \leq v$, and $u(x)=v(x)$ then $\partial_x u \subset \partial_x v$. In particular $|\partial_x u| \leq |\partial_x v|$.
\end{lemma}
The following Lemma is a consequence of the Brunn-Minkowski inequality (see~\cite{schneider1993convex}).
\begin{lemma}
Let $\Omega$ be a convex neighborhood of a point $x \in \R^d$, and let  $u,v : \Omega \to \R^d$ be convex. Let $w = (1-t) u+t v$ with $0 \leq t\leq 1$. Then $(1-t)\partial_x u + t \partial_x v \subset \partial_x w$, where the sign plus denotes a Minkowski sum. In particular $(1-t) |\partial_x u|^\frac 1 d + t |\partial_x v|^\frac 1 d \leq |\partial_x w|^\frac 1 d$.
\end{lemma}

We denote by $\underline u$ the largest convex function bounded above by $u$.

\begin{proposition}
\label{ref:propOverEstim}
Let $X \subset \RR^d$ be a finite point set, and let $x \in X$. Let $Y \subset X$ be symmetric w.r.t.\ the point $x$, i.e.\ $\forall y \in Y$ one has $2x-y \in Y$.
Let $u : X \to \R$, and let $v : Y \to \R$ be defined by $v(y) := \frac 1 2 (u(y)+u(2 x-y))$. Then
\begin{equation}
|\partial_x \underline u| \leq |\partial_x \underline v|.
\end{equation}
\end{proposition}

\begin{proof}
If $u(x) > \underline u(x)$, then $|\partial_x \underline u(x)| = 0$ and there is nothing to prove. We thus assume $u(x) = \underline u(x)$. 
Introduce the restriction $u_Y := u_{|Y}$, and its symmetry $u_{-Y} := u_Y(2x-\cdot)$, and note that $\frac 1 2 (\underline {u_Y}+\underline {u_{-Y}}) \leq \underline v$.

The equality $u(x) = \underline u(x)$ implies $u_Y(x) = \underline {u_Y}(x)$, hence also $u_{-Y}(x) = \underline {u_{-Y}}(x)$ and $v(x) = \underline v(x)$.
We have 
\begin{align*}
  \underline u &\leq \underline {u_Y} \quad \text{hence } |\partial_x \underline u| \leq |\partial_x \underline {u_Y}|,\\
  \frac 1 2 (\underline{u_Y} + \underline {u_{-Y}})   &\leq \underline v \quad  \text{ hence } \frac 1 2 |\partial_x \underline {u_Y}|^\frac 1 d +\frac 1 2 |\partial_x \underline {u_{-Y}}|^\frac 1 d  \leq |\partial_x \underline v|^\frac 1 d.
\end{align*}
The announced result follows since $|\partial_x \underline {u_Y}| = |\partial_x \underline {u_{-Y}}|$.

\end{proof}

The final step to prove  Lemma~\ref{lem:subgradh2}  is given by Remark 1.10 in \cite{malbr} which shows that 
$  \MA(\udisc)[x]\Vn = |\partial_x \underline v (x) |$.


\pagebreak 


\bibliographystyle{alpha}
\bibliography{bibbv2}
\end{document}